\documentclass[
onefignum,onetabnum]{siamart190516}

\usepackage{caption,subcaption}
\usepackage{color}
\usepackage{lipsum}
\usepackage{amsfonts}
\usepackage{graphicx}
\usepackage{epstopdf}
\usepackage{algorithmic}
\ifpdf
  \DeclareGraphicsExtensions{.eps,.pdf,.png,.jpg}
\else
  \DeclareGraphicsExtensions{.eps}
\fi

\newsiamremark{remark}{Remark}
\newsiamremark{hypothesis}{Hypothesis}
\crefname{hypothesis}{Hypothesis}{Hypotheses}
\newsiamthm{claim}{Claim}

\headers{Sampling the X-ray transform on simple surfaces}{F. Monard and P. Stefanov}

\title{Sampling the X-ray transform on simple surfaces}

\author{Fran\c{c}ois Monard\thanks{Department of Mathematics, University of California Santa Cruz, Santa Cruz, CA 95064
  (\email{fmonard@ucsc.edu}, \url{http://people.ucsc.edu/\string~fmonard/}). F.M. is partially supported by NSF-CAREER grant DMS-1943580.}
\and Plamen Stefanov\thanks{Department of Mathematics, Purdue University, West Lafayette, IN 47907 
    (\email{stefanov@math.purdue.edu}, \url{https://www.math.purdue.edu/\string~stefanov/}). P.S. is partially supported by the NSF under grant DMS-1900475.}}

\usepackage{amsopn}

\newcommand{\g}[2]{ \langle #1, #2 \rangle_g }
\newcommand{\G}[2]{ \langle #1, #2 \rangle_G }
\newcommand{\two}{ {\mathrm{II}} }
\newcommand{\HH}{ H}
\def\Cm{ {\mathbb{C}} }
\def\Dm{ {\mathbb{D}} }
\def\Rm{ {\mathbb{R}} }
\def\Zm{ {\mathbb{Z}} }
\def\Sm{ {\mathbb{S}} }
\newcommand{\R}{{\bf R}}
\newcommand{\semi}{semi\-classical}
\DeclareMathOperator{\Id}{Id}
\renewcommand{\r}[1]{(\ref{#1})}
\renewcommand{\Xi}{B}
\newcommand{\PDO}{$\Psi$DO}
\newcommand{\HPDO}{$h$-$\Psi$DO}
\DeclareMathOperator{\supp}{supp}
\DeclareMathOperator{\sinc}{sinc}
\DeclareMathOperator{\WFH}{WF_{\it h}}
\newcommand{\be}[1]{\begin{equation}\label{#1}}
\newcommand{\ee}{\end{equation}}
\def\D{\mathcal{D}}
\def\C{\mathcal{C}}
\def\t {p}

\def\ss{\mathfrak{s}}

\newcommand{\rev}[1]{#1}
\newcommand{\note}[1]{#1}

\ifpdf
\hypersetup{
  pdftitle={Sampling the X-ray transform on simple surfaces},
  pdfauthor={F. Monard and P. Stefanov}
}
\fi

\begin{document}

\maketitle

\begin{abstract}
  We study the problem of proper discretizing and sampling issues related to geodesic X-ray transforms on simple surfaces, and illustrate the theory on simple geodesic disks of constant curvature. Given a notion of band limit on a function, we provide the minimal sampling rates of its X-ray transform for a faithful reconstruction. In Cartesian sampling, we quantify the quality of a sampling scheme depending on geometric parameters of the surface (e.g. curvature and boundary curvature), and the coordinate system used to represent the space of geodesics. When aliasing happens, we explain how to predict the location, orientation and frequency of the artifacts.
\end{abstract}



\section{Introduction}

A semiclassical theory of sampling was initiated in \cite{Stefanov2018,Stefanov2020}, addressing in the large sampling issues in linear inverse problems where the forward operator is a classical Fourier Integral Ope\-rator (FIO): given a classical FIO $A$ and an unknown essentially band-limited function $f$, how to derive sampling requirements on $Af$ guaranteeing a good (unaliased) reconstruction of $f$ when, say, $A$ is elliptic? If the sampling rate of $Af$ is fixed, what limit does that impose to the resolution of the recovered $f$? 
If $Af$ is undersampled, can one predict the location, frequency and orientation of the artifacts?

Even if $A$ admits a reasonable inverse $A^{-1}$ at the continuous level, and the data is noiseless, a numerical inversion consists of applying a discretized version of $A^{-1}$ to a discrete data set $Af$, evaluated on a discrete reconstruction grid. At this point, aliasing can occur at three places: (1) when the sampling of the data is too coarse to resolve all details of $Af$; (2) when the resolution of the reconstruction grid is too coarse to resolve all details of $f$; (3) when some building blocks of $A^{-1}$ are discretized using step sizes which do not capture all details throughout the inversion. 

The main ingredients to avoid such aliasing and derive minimal sampling rates are (i) the Nyquist-Shannon theorem, ruling Cartesian sampling rates in terms of the band limit and (ii) an understanding of how band limits (or more generally, semiclassical wavefront sets) get mapped through the action of FIOs. All of this is also dependent on a sampling strategy, and on a notion of band limit, which is relative to a choice of coordinates. This realization forces one to consider that, while the analysis of geometric inverse problems at the continuous level should be somewhat independent of the choice of coordinates, one should look for ``good'' coordinate systems for purposes of optimal sampling. 

As advocated in \cite{Stefanov2018}, the appropriate viewpoint to approach sampling is {\em semiclassical} for the following reasons: although classical singularities detect jumps (and as such, local features on images), they do not encode band limit, as a nonempty wavefront set is conic, and hence, unbounded; on the other hand, the semiclassical wavefront set keeps track of the detail content of a (semiclassical) function, including size of details. Once this is established, further work in \cite{Stefanov2018} shows that the Nyquist-Shannon has a natural interpretation in semiclassical terms, and that classical FIOs induce natural maps on semiclassical singularities through their canonical relation. Further, the process of aliasing (due to periodization in the Fourier domain) can be naturally described as a semiclassical FIO with an easy-to-describe canonical relation. In what follows, these principles become the building blocks of how to predict sampling rates, and how to describe the aliasing artifacts when the sampling requirements are not met.

While \cite{Stefanov2018} illustrated the theory on the Euclidean Radon transform and the thermoacoustic tomography problem, see also \cite{Chase_Sampling}, the present article aims at illustrating these principles further, in the case where $A$ is the geodesic X-ray transform (denoted $I_0$ below) on simple surfaces. This operator, a prototype of integral geometry, is currently under active study for its connections with imaging sciences such as X-ray Computerized Tomography, travel-time tomography, see, e.g., \cite{SUVZ-survey, SUV_anisotropic}, and polarimetric neutron tomography (via its non-abelian analogues), see \cite{Ilmavirta2019,Paternain2021} for recent expository works on the topic. In the Euclidean case, this operator is the theoretical backbone of X-ray CT, for which resolution issues and optimal sampling geometries have been extensively studied, see e.g. \cite{Natterer1993, Natterer2001}. On sampling issues for non-Euclidean integral-geometric transforms, a recent study was also undertaken in \cite{Katsevich2020,Katsevich2017c,Katsevich2019a}, under the {\it a priori} assumption that the unknown function has jump discontinuities, attempting to resolve jumps via a study of edge responses on the data side. As such functions are never bandlimited, the present methods would always predict aliasing, and thus the methods involved there are of a different nature as they exploit a different type of sparsity.

In what follows, we first lay out the theory for a general simple surface, in particular we look for convenient representations of the canonical relation of the X-ray transform in two coordinate systems, see \cref{lem:etaFB,lem:etaWT}. The first one is the well-known fan-beam coordinate system, while the second recovers parallel coordinates in the Euclidean case, up to a squareroot rescaling near the boundary of the manifold of geodesics, see \cref{sec:canrelcoords,rem:eucl}. Previous coordinates systems were also introduced in \cite{Assylbekov2018}, although the second one, introduced in \cref{sec:parallel}, seems to be new. Such coordinates exist on any simple surface and, unlike fan-beam coordinates, encode global knowledge about the scattering relation of the surface. 

For the purpose of numerical illustration, we then work with simple geodesic disks of constant curvature (Constant Curvature Disks for short, or CCDs), a two-parameter family of simple surfaces indexed by their diameter and curvature, on which the X-ray transform was recently studied in \cite{Mishra2019,Monard2019a}. Constant Curvature Disks deserve to be first analyzed for the following reasons: all geometric quantities can be explicitly computed; the circular symmetry simplifies the exposition; the inversion formula for $I_0$, first established in \cite{Pestov2004}, is exact at the continuous level, unlike in non-constant curvature where a smoothing error term appears. In addition, this family is rich enough that one can witness two kinds of degeneracies of simplicity (boundary curvature going to zero and emergence of conjugate points at the boundary, or boundary length going to infinity), both of which have quantifiably detrimental impact on the sampling rates. 

We present numerical experiments in \cref{sec:canrel}, illustrating several principles, which we briefly mention along with some of the conclusions. After choosing a notion of band limit, a coordinate system in data space and a sampling strategy, our experiments show that sampling rates are indeed sharply dictated by the canonical relation of $I_0$, and in turn, by Jacobi fields of the underlying surface. These quantities allow to  predict {\em a priori} minimal sampling rates that help avoid aliasing, and they help measure the relative performance of two coordinate systems in a given situation. Depending on the geometry chosen, the shape of the frequency set of $I_0 f$ allows for denser non-overlapping periodizations (than if it was fit into a Cartesian box), and this in turn drives coarser sampling limits, in a similar fashion to \cite{Natterer1993}, an observation going back to \cite{PetersenM}. Moreover, for integrands $f$ with a frequency content that is simple to describe, we explain the artifacts on $f$ that result from undersampling $I_0 f$, in particular, how to predict their location, orientation and frequency.  A notable feature in this case is that when inverting $I_0$ using undersampled data, certain features may change {\em both in position and direction}. It is also surprising that in this case, undersampled X-ray data can give rise to artifact reconstructions with {\em higher} frequency than the original $f$. 
This is to be contrasted with ``traditional aliasing'', where the subsampled function can only have lower frequency content, at the same location. In the present case, spectral folding in data space can send singularities to places whose preimages by the canonical relation of $I_0$ have higher frequency than their original preimages before aliasing.

\smallskip
{\bf Outline.} The remainder of the article is structured as follows. We first review the general material on semiclassical sampling in \cref{sec_sampling}. \Cref{sec:canrel} covers the canonical relation of the X-ray transform, first for a general surface (\cref{sec:canrel} through \cref{sec:canrelcoords}), then more explicitly in the case of CCDs (\cref{sec:CCDs}). \Cref{sec:numerics} covers numerical illustrations in the case of CCDs. Proofs of certain results from \cref{sec:canrel} are relegated to \cref{sec:app}.

\section{Review of semiclassical sampling}  \label{sec_sampling}
We review briefly the asymptotic approach to sampling developed by the second author in  \cite{Stefanov2018}, see also \cite{Stefanov2020}. 

\subsection{The classical Nyquist-Shannon theorem} 

The  Nyquist-Shannon sampling theorem, see, e.g., \cite{Natterer2001, Epstein-book}, says that a function $f\in L^2(\Rm^n)$ satisfying the band limit condition 
\begin{align}
    \supp \hat f\subset [-B,B]^n \qquad \text{for some } B\in (0,\infty)
    \label{eq:band limit}
\end{align}
is uniquely determined by its samples $f(sk)$, $k\in \mathbb{Z}^n$, as long as $0<s\le \pi/B$. The function $f$ can be recovered from its samples by the following interpolation formula
\be{P1}
f(x) = \sum_{k\in\mathbb{Z}^n} f(sk)\chi\Big(\frac{1}{s}(x-sk)\Big), \quad \chi(x):= \prod_{j=1}^n\sinc(\pi x^j),
\ee
\note{where $\sinc(x)=\sin x/x$. } 
The mapping $L^2(\Rm^n)\ni f\mapsto \{f(sk)\}_k \in \ell^2$, restricted to functions satisfying \cref{eq:band limit}, is an isometry with the $\ell^2$ norm scaled as below:
\be{P2}
\|f\|^2_{L^2}= s^n\sum_{k\in\mathbb{Z}^n}|f(sk)|^2,
\ee
and is unitary, when $s=\pi/B$. Such functions are called band limited with band limit $B$ in each variable; but we may regard the box $[-B,B]^n$ as the band limit. They must be entire functions which limits the application of this theorem in this formulation. In particular,  $f$ cannot be compactly supported without being zero. Another inconvenience is that the sum \r{P1} is in general infinite and the sinc functions in \r{P1} are not well localized near $x_k=sk$. The latter can be improved if we assume $s<\pi/B$ (strictly), then $\chi$ could be chosen to be \note{in the} Schwartz class. 
The isometry property \cref{P2} allows us to localize the sum by a truncation and to consider approximately band limited functions, at the expense of ``small'' errors. 

If the band limit condition \cref{eq:band limit} does not hold, there is no uniqueness, and recovery using \r{P1} leads to artifacts; this is called aliasing. 

\smallskip
\noindent{\bf The semiclassical version of Nyquist-Shannon.} In \cite{Stefanov2018}, we took an asymptotic point of view. Specifically, in the study of (Cartesian) sampling, upon fixing a suitably chosen interpolation kernel $\chi$, and freezing a relative step $s\in (0,\infty)$, a function $f$ gives rise to an $h$-dependent family of interpolants 
\begin{align*}
    f^{\rm int}_h(x) = \sum_{k\in\mathbb{Z}^n} f(shk)\chi\Big(\frac{1}{sh}(x-shk)\Big),
\end{align*}
and the appropriate framework to study the error $\|f-f^{\rm int}_h\|_{L^2}$, as $h\to 0$, is the semiclassical analysis, see, e.g., \cite{Zworski_book, Martinez_book}. In other words, the sampling step is $sh$ with $h\to0$, and we are interested in the asymptotic behavior. 
This circle of ideas further allows to consider $h$-dependent $f$'s, and in order to formulate the main adaptation of the Nyquist-Shannon theorem, we first recall some facts about the semiclassical analysis, see also \cite{Zworski_book, Martinez_book}. 

In what follows, we consider functions which may depend on $h$, with the  requirement to be {\bf tempered} in the sense that $\|f\|_{H^s}=O_s(h^{-N})$ for some $s$ and $N$. We  omit the subscript $h$ in $f_h$ often. We define the semiclassical Fourier transform by rescaling the dual variable $\xi$ to $\xi/h$ and setting
\[
\mathcal{F}_hf(\xi) =\int e^{-i x\cdot\xi/h} f(x)\,d x.
\]
The {\bf semiclassical wave front set} $\WFH(f)$ of a tempered $f_h$ is the complement of those $(x_0,\xi^0)\in\Rm^{2n}$ for which there exists a $C_0^\infty$ function $\phi$ so that $\phi(x_0)\not=0$ with
\be{WFH_temp}
\mathcal{F}_h(\phi f_h) (\xi)  = O(h^\infty)\quad \text{for $\xi$ in a neighborhood of $\xi^0$}.
\ee
It is a closed set defined invariantly in $T^*\Rm^n$ but it is not a conic set in general. Its elements are called {\bf (semiclassical) singularities} even if the function has no classical singularity. A function $f_h\in C_0^\infty(\Rm^n)$ is then {\bf semiclassically band limited} if $\WFH(f_h)$ is compact. Such functions are also called \textit{localized in phase space} in \cite{Zworski_book}, and we give several alternative characterizations in \cite{Stefanov2018}. We view $\WFH(f)$ as the band limit of $f$ now, i.e., we lift that notion to the phase space.

An example of such $f_h$ can be constructed as follows. Take any distribution $f_0$ supported in some compact set $K\subset\Rm^n$. Let $\hat\phi\in C_0^\infty$ be supported in $\Sigma\subset\Rm^n$, and set $\phi_h= h^{-n} \phi(\cdot/h)$. Then $f_h := f_0*\phi_h$ is semiclassically  band limited with $\WFH(f)\subset K\times\Sigma$. This example is a special case of the following one: for every $p\in C_0^\infty(\Rm^n\times\Rm^n)$, $p(x,hD)f_0$ is \semi ly  band limited with $\WFH(f)\subset \supp p$, where $P=p(x,hD)$ is the \HPDO\ 
\be{hPDO}
Pf(x) = (2\pi h)^{-n}\iint e^{i (x-y)\cdot\xi/h} p(x,\xi) f(y)\,d y\, d\xi.
\ee
This  symbol belongs to the symbol class $S^{m,k}(\Omega)$, where $\Omega\subset \Rm^n$ is an open set, as the smooth functions $p(x,\xi)$ on $\Rm^{2n}$, depending also on $h$, satisfying the symbol estimates
\be{hpdo1}
|\partial_x^\alpha \partial_\xi^\beta p(x,\xi)|\le C_{K,\alpha,\beta}h^k\langle \xi\rangle^{m-|\beta|}
\ee
for $x$ in any compact set $K\subset\Omega$. We are interested here in symbols compactly supported in the $\xi$ variable; then the factor $\langle \xi\rangle^{m-|\beta|}$ above can be omitted. 

The symbols we use here have an \note{asymptotic} expansion $p=h^kp_1+h^{k-1}p_2+\dots$, and we call $p_1$ the principal symbol. It is defined invariantly on $T^*\Rm^n$, where $\WFH(f)$ lives, as well. 

Let $\Sigma_h(f)$ be the \semi\ frequency set of $f$, defined as the projection $\pi_2 (\WFH(f))$ of $\WFH(f)$ onto the second factor\footnote{Note that $\pi_2$ is not coordinate invariant and is understood here in the sense of the trivialization $T^* \Rm^n\ni (x, \xi = \xi_i dx^i) \leftrightarrow (x,(\xi_1, \dots, \xi_n))\in \Rm^n \times \Rm^n$}. The \semi\ sampling theorem in \cite{Stefanov2018} is the following. 

\begin{theorem}\label{thm:sampling}
Let $f_h$ be \semi ly band limited with $\Sigma_h(f) \subset  \prod(-B_j,B_j)$ with some $B_j>0$. Let $\hat \chi_j \in L^\infty(\Rm)$ be supported in $[-\pi,\pi]^n$ and $\hat \chi_j(\pi \xi_j/B_j)=1$ for $\xi\in \Sigma_h(f) $. If $0<s_j \le \pi/B_j$, then 
\be{2.3}
f_h(x) = \sum_{k\in \mathbb{Z}^n} f_h(s_1 hk_1,\dots, s_nhk_n) \prod_j \chi_j\left(\frac{1}{s_j h}(x^j-s_jhk_j)\right) + O_{\mathcal{S}}(h^\infty)\|f\|_{L^2} ,
\ee
and 
\be{2.4}
\|f_h\|_{L^2}^2 = s_1\dots s_n h^n\sum_{k\in \mathbb{Z}^n} |f_h(s_1hk_1,\dots, s_nhk_n)|^2 + O(h^\infty)\|f\|_{L^2}^2.
\ee
\end{theorem}

\note{Above, $O(h^\infty)$ stands a function of $h$ bounded by $C_Nh^N$, as $h\ll1$, for every $N$. Similarly, $O_{\mathcal{S}}(h^\infty)$ denotes a function of $x$ and $h$ with every seminorm in the Schwartz class $\mathcal{S}(\R^n)$ being $O(h^\infty)$.}

Note that we chose the interval $(-B_j,B_j)$ to be open so that (the closed) $\Sigma_h(f) $ would be included in the box $\prod(-B_j,B_j)$ strictly, so we require some oversampling. The interpolating functions $\chi_j$ we chose do not need to be sinc ones; they can be and will be chosen later to be \note{in the} Schwartz class, i.e., we will require $\hat\chi_j\in C_0^\infty$. This is possible because of the gap between the boundary of $\prod(-B_j,B_j)$ and  $\Sigma_h(f) $. In practical applications, this is done approximately only. The Lanczos-3 interpolation, for example, is an approximate way of doing this when we oversample twice; we refer to \cite{Stefanov2020} for more details.

\subsection{Sampling classical FIOs} \label{sec:samplingFIOs}

\Cref{thm:sampling} above is just a rescaled and localized version of the classical sampling theorem, with an error estimate. The main point in \cite{Stefanov2018} is the sampling of $Af$, where $A$ is an FIO with a locally diffeomorphic canonical relation, knowing the \semi\ band limit of $f$ (but not of $Af$ a priori), the resolution limit on $f$ posed by a fixed sampling rate of $Af$,  the nature of the aliasing artifacts of $f$ when $Af$ is undersampled, and analysis of locally averaged measurements. We explain some of those results below without formulating them formally.

In applications, we take measurements $Af$ not pointwise but locally averaged like above with some $\phi_h$. The convolution with $\phi_h$ is a \HPDO\ with symbol $\phi(\xi)$. If the canonical relation of $A$ is a local diffeomorphism, by Egorov's theorem (more precisely, from its semiclassical version), $\phi_h*(Af_0)= APf_0+O(h^\infty)$ with such a $P=p(x,hD)$ and even if $f_0$ were $h$-independent, now $Pf_0$ would be $h$-dependent, and it would be  semiclassically band limited. More generally, we can replace $f_0$ by such a regularized version with an $O(h^\infty)$ loss. Another very similar situation occurs in numerical simulations. The size of the grid limits the highest frequency $f$ can have. Even if $f$ is given by a formula, we would have to regularize first, and then sample. A finer grid would require a regularization with a more concentrated kernel, etc., so we are really working with $\phi_h*f_0$, where $h$ is proportional to the grid size. 

Let $A$ be a classical FIO with a canonical relation $C$ which is a local diffeomorphism. The action of $A$ on $\WFH(f)$ is roughly speaking the same as if $A$ was \semi\ or if the wave front set was classical, away from the zero section. More precisely, by  \cite[Theorem~2.6]{Stefanov2018},
\be{FIO1}
\WFH(Af)\setminus 0\subset C(\WFH(f)\setminus 0).
\ee
In particular, \cref{FIO1} is true if $A$ is an \HPDO\ or just a \PDO. 

\smallskip
\noindent{\bf Conditions for unaliased sampling.} For a classical FIO $A$ and a semiclassically band limited $f=f_h$, relation \cref{FIO1} implies that $Af$ is also semiclassically band limited, and the sampling requirements of $Af$ as in \cref{thm:sampling}, are determined by the smallest symmetric box containing $\Sigma_h(Af)$. Such a box will depend on the choice of the coordinate system chosen to represent the data space in which we will assume a Cartesian sampling of $Af$. Specifically, first fix a choice of coordinate system $y = (y^1, \dots, y^m)$ for the data space $Y$  giving rise to a projection $\pi_2$ onto the dual variable ($\pi_2 (y, \eta_i dy^i) = (\eta_1, \dots, \eta_m)$). In these coordinates, relation \cref{FIO1} implies the set-theoretic upper bound
\begin{align}
    \pi_2 (C (\WFH(f))) \subset \prod_{j=1}^m (-b_{y^j}, b_{y^j}), \quad \text{where} \quad b_{y^j}:= \sup_{\eta\in \pi_2 (C(\WFH(f)))} |\eta_j|.
    \label{eq:bnumbers}    
\end{align}

The numbers $b_{y^j}$ are thus the band limits of $Af$ relative to the coordinates $y^j$. As such, the notion of band limit is not invariantly defined, in that it depends on the local coordinates in data space. Then a non-aliased Cartesian sampling of $Af$ will be achieved provided that each sampling rate $h_j$ along the $y^j$ axis satisfies 
\begin{align}
    h_j \le \pi/ b_{y^j}.  
    \label{eq:samplingAf}
\end{align}

The computation of the numbers $b_{y^j}$ thus depends on (i) the {\it a priori} assumption on an upper bound of $\WFH (f)$ such as \cref{eq:band limit} (or more generally, \cref{eq:band limit2} below), (ii) the canonical relation of $A$ and (iii) the choice of coordinate system on $Y$. 

\smallskip
\noindent{\bf Aliasing artifacts.} If the sampling condition \cref{eq:samplingAf} is not satisfied, the aliasing artifacts for the reconstructed $f$ are described by $h$-FIOs; they are not local in general, unlike aliasing for $f$ alone, i.e., when formally $A=\Id$. More precisely, artifacts produced when we apply \cref{thm:sampling} \note{when}  the band limited condition $\Sigma_h(f) \subset  \prod(-B_j,B_j)$ is not satisfied, can be characterized as adding $\sum_k G_k f_h$ to $f_h$, where $G_k$ are h-FIOs with canonical relations given by the shifts
\be{sh}
S_k:(x,\xi) \longmapsto    (x,\xi+{2\pi} k/s)
\ee
restricted to $\xi$ such that $\xi+{2\pi} k/s$ is in the Nyquist box above. 
Note that only $\xi$ shifts, and $x$ does not. This is known as ``frequency folding'' in the applied literature. Now, if the measurement $Af$ is aliased, we apply this to $Af$, and the induced artifacts in the reconstructed $f$ would be sums of compositions $A^{-1}G_kAf_h$ which are h-FIOs under our assumptions on $A$ with canonical relations $\kappa^{-1}\circ S_k\circ \kappa$, where $\kappa$ is that of $A$. In practical applications, $k$ with   components $0, \pm1$  are the most significant ones. Such canonical relations move not just $\xi$ but they move $x$ as well in general. 

\subsection{Further upsampling strategies to reduce the sampling limit} \label{sec:slanted}

As noted in \cite{PetersenM},  the classical sampling theorem, as well as our \cref{thm:sampling} can be generalized in two directions. First, $\Sigma_h(f)$ does not need to be contained in a box. To be more precise, it always is, since it is bounded but  we assume now that  $\Sigma_h(f)\subset\mathcal{B}$ with some open bounded set $\mathcal{B}$, which may  lead to a more efficient sampling configuration. We assume that the images of $\Sigma_h(f)$ under the translations $\xi\mapsto \xi+2\pi(W^*)^{-1} k$, $k\in \mathbb{Z}^n$, are mutually non-intersecting (the tiling condition), for some invertible matrix $W$. This is the second generalization since we do not need $W=\Id$. If the set $\mathcal{B}$ is close enough to $\Sigma_h(f)$, it would satisfy the same tiling condition. 
Then we can apply the linear transformation $x=Wy$. For the dual variables $\xi$ and $\eta$ we have $\eta=W^*\xi$. The Nyquist condition then can be generalized  as follows: we want the images of $\Sigma_h(f)$ under the translations $\xi\mapsto \xi+2\pi(W^*)^{-1} k$, $k\in \mathbb{Z}^n$, to be mutually non-intersecting (the tiling condition), for some invertible matrix $W$. Then \cref{thm:sampling} remains true with samples $shWk$ and the additional factor $|\det W|$ in front of the sum in \r{2.3} and \r{2.4}; and with $\chi$ so that $\hat \chi=1$ near $\Sigma_h(f)$, $\supp\hat \chi\subset \mathcal{B}$.

\section{The canonical relation of the X-ray transform on a simple surface}
\label{sec:canrel}

\subsection{Geometry of $SM$} 

Let $(M,g)$ be a non-trapping Riemannian surface with strictly convex boundary (in the sense of $\partial M$ having strictly positive second fundamental form), with unit tangent bundle $SM\stackrel{\pi}{\longrightarrow} M$ and geodesic flow $\varphi_t\colon SM\to SM$, defined on the set
\begin{align}
    \D = \big\{ (x,v,t): (x,v)\in SM,\ t\in [\rev{-}\tau(x,-v), \tau(x,v)] \big\}, 
    \label{eq:D}
\end{align}
where $\tau(x,v)\ge 0$ is the first exit time of the geodesic emanating from $(x,v) \in SM$, everywhere finite under the non-trapping assumption. We model the space of geodesics over the inward boundary $\partial_+ SM$, where 
\begin{align*}
    \partial_\pm SM := \big\{(x,v)\in SM,\ x\in \partial M,\ \pm g(v,\nu_x)\ge 0\big\},
\end{align*}
where $\nu_x$ is the inward normal at $x\in \partial M$. We define the scattering relation
\begin{align}
    S\colon \partial_+ SM \to \partial_- SM, \qquad S(x,v):= \varphi_{\tau(x,v)}(x,v).
    \label{eq:scatrel}
\end{align}
\rev{With $X:= \frac{d}{dt}|_{t=0} \varphi_t(x,v)$ the infinitesimal generator of the geodesic flow, $V$ the infinitesimal generator of the circle action on the fibers of $SM$, and $X_\perp := [X,V]$, the triple $(X,X_\perp,V)$ forms a frame of $SM$}
with structure equations
\begin{align*}
    [X,V] = X_\perp, \qquad [X_\perp, V] = -X, \qquad [X,X_\perp] = -\kappa(x) V, \quad (\kappa: \text{Gauss curvature}). 
\end{align*}

On $\partial SM$, denote $\mu(x,v) = g_x(v,\nu_x)$ and $\mu_\perp = V\mu = g_x(v_\perp,\nu_x)$, where $\nu_x$ is the inner unit normal to $x\in \partial M$. A basis of $T_{(x,v)}(\partial_+ SM)$ is given by $\left(V_{(x,v)}, H_{(x,v)}\right)$, where $H = \mu_\perp X + \mu X_\perp$ is the horizontal lift of the oriented unit tangent vector to $\partial M$. 

On the domain of definition of the geodesic flow \cref{eq:D}, we define the scalar Jacobi fields $a(x,v,t)$ and $b(x,v,t)$ as the unique functions solving
\begin{align}
    \ddot a + \kappa(\gamma_{x,v}(t)) a = \ddot b + \kappa(\gamma_{x,v}(t)) b = 0, \qquad \left[ \begin{smallmatrix} a & b \\ \dot a & \dot b \end{smallmatrix} \right] (x,v,0) = \left[ \begin{smallmatrix} 1 & 0 \\ 0 & 1 \end{smallmatrix} \right],
    \label{eq:ab}
\end{align}
where $\varphi_t(x,v) = (\gamma_{x,v}(t), \dot\gamma_{x,v}(t))$. We call $(M,g)$ simple if in addition to the assumptions of being convex and non-trapping, there are no conjugate points in $M$. This is equivalent to saying that $b(x,v,t)\ne 0$ for all $(x,v,t)\in \D$ such that $t\ne 0$.

\subsection{Canonical relation of the X-ray transform - geometric version}
\label{sec:canrel_geom}

With $(M,g)$ a convex non-trapping Riemannian surface, one defines the geodesic X-ray transform $I_0\colon C_c^\infty(M) \to C_c^\infty(\partial_+ SM)$ as 
\begin{align}
    I_0 f (x,v) = \int_0^{\tau(x,v)} f(\gamma_{x,v}(t))\ dt, \qquad (x,v) \in \partial_+ SM.
    \label{eq:I0}
\end{align}
It is now well-documented that $I_0$ is a Fourier Integral Operator (FIO), see e.g. \cite{Guillemin1990,Monard2013b,Holman2015}. We now recall the computation of its canonical relation, done by factoring $I_0$ into two simpler FIOs and using the clean intersection calculus.  The calculation that follows would also work for weighted X-ray transforms with non-vanishing weight but we focus on the unweighted transform for simplicity.

The canonical projection $\pi\colon SM\to M$, and the basepoint map $F\colon SM\to \partial_+ SM$ defined by $F(x,v) = \varphi_{-\tau(x,-v)}(x,v)$ are both smooth submersions with $\pi$ proper, and as such define pull-backs $\pi^*\colon C_c^\infty(M)\to C_c^\infty (SM)$ ($\pi^* f(x,v):= f(x)$) and $F^*\colon C_c^\infty(\partial_+ SM)\to C^\infty(SM)$ ($F^* g (x,v) = g(F(x,v))$), as well as push-forwards $\pi_*\colon C_c^\infty(SM)\to C_c^\infty(M)$ and $F_* :  C_c^\infty(SM) \allowbreak \to  C_c^\infty(\partial_+ SM)$ via the relations
\begin{align*}
    \int_{SM} (\pi^* f) g\ d\Sigma^3 = \int_M f (\pi_* g)\ d\text{Vol}_g, \qquad  \int_{SM} (F^* h) g\ d\Sigma^3 = \int_{\partial_+ SM} h\ (F_* g) \mu\ d\Sigma^2, 
\end{align*}
to be true for all $f\in C_c^\infty(M)$, $g\in C_c^\infty(SM)$ and $h\in C_c^\infty (\partial_+ SM)$. 

Using Santal\'o's formula \cite{Sh-UW}, one finds that $I_0 = F_*\circ \pi^*$ where, as explained in \cite{Holman2015}, $\pi^*$ is an FIO of order $\frac{1-n}{4}$ (if $n = \dim M$ in general), and $F_*$ is an FIO of order $\frac{-1}{4}$, with canonical relations
\begin{align*}
	\C_{\pi^*} &= \left\{ (d\pi|_{(x,v)}^T \omega, \omega): (x,v)\in SM,\ \omega \in T^*_x M \backslash\{0\} \right\} \subset T^* SM \times T^* M, \\
	\C_{F_*} &= \left\{ (\eta, dF|_{(x,v)}^T \eta): (x,v)\in SM,\ \eta \in T^*_{F(x,v)} \partial_+ SM \backslash\{0\} \right\} \subset T^* (\partial_+ SM) \times T^* SM.
\end{align*}
The composition is proved to be clean, and in the current notation, the composition is carried out in \cite[Lemma 6.2]{Ilmavirta2019}, with $H_{(x,v)}$ here replacing $\nabla_T|_{(x,v)}$ there. In the statement below, for a vector \rev{ $v\in T_x M$, we denote $v_\perp^\flat\in T_x^* M$ the co-vector obtained by composing the almost-complex structure with the musical operator\footnote{In local coordinates $(x_1,x_2)$, if $v = v^1 \partial_{x^1} + v^2 \partial_{x^2}$, then $v_\perp^\flat = \sqrt{\det g} (-v^2 dx^1 + v^1 dx^2)$.}. }

\begin{lemma}\label{lem:composition} The canonical relation of $I_0$, $\C_{I_0} = \C_{F_*} \circ \C_{\pi^*}\subset T^* (\partial_+ SM)\times T^* M$, computed as
    \begin{align}
	    \C_{F_*} \circ \C_{\pi^*} &= \Big\{ (\lambda \eta_{x,v,t}, \lambda \omega_{x,v,t}),\ (x,v)\in \partial_+ SM,\ t\in (0,\tau(x,v)),\ \lambda\in \Rm \Big\}, 
	\label{eq:relation}
    \end{align}
    where $\omega_{x,v,t} := (\dot\gamma_{x,v}(t))_\perp^\flat \in T^*_{\gamma_{x,v}(t)} M$, and where $\eta_{x,v,t}\in T^*_{(x,v)} \partial_+ SM$ is uniquely defined by the relations
    \begin{align}
	\eta_{x,v,t} (V) = b(x,v,t), \qquad \eta_{x,v,t} (H) = -\mu a(x,v,t). 
	\label{eq:linsys2}
    \end{align}
\end{lemma}
To view this as two graphs $C_{\pm} \colon T^* M \to T^* (\partial_+ SM)$, the intuition goes as follows: 
\begin{itemize}
    \item Given $\omega\in T^* M$, there exist exactly two quadruples $(x_\pm,v_\pm,t_\pm,\lambda_\pm)$, 
	\begin{align*}
	    \omega = \lambda_+ (\dot\gamma_{x_+,v_+}(t_+))_\perp^\flat = \lambda_- (\dot\gamma_{x_-,v_-}(t_-))_\perp^\flat,
	\end{align*}
	where $(x_\pm,v_\pm)\in \partial_+ SM$, $t_\pm\in (0,\tau(x_\pm,v_\pm))$ and $\pm \lambda_\pm >0$. Moreover, we have the relations $(x_+,v_+) = S(x_-,-v_-)$ (with $S$ defined in \cref{eq:scatrel}), $\tau(x_+,v_+) = \tau(x_-,v_-) = t_+ + t_-$, and $\lambda_+ = -\lambda_-$. 
    \item With $\eta_{x,v,t}$ as defined in \cref{eq:linsys2}, we can then define 
	\begin{align}
	    C_\pm (\omega) = \lambda_\pm \eta_{x_{\pm}, v_{\pm}, t_{\pm}}.
	    \label{eq:Cpm}
	\end{align} 
	\rev{More invariantly (i.e. solely in terms of the flow of integration), if $\omega\in T_x^* M$, $w\in S_x M$ is the unique vector such that $(w_\perp)^\flat \in \Rm_{+} \omega$ and with $(x_\pm, v_\pm, t_\pm)$ as above, the canonical relation can be expressed as 
	    \begin{align*}
		C_\pm (\omega) = (d\varphi_{t_\pm})^T_{(x_\pm, v_\pm)} (d\pi^T_{(x, \pm w)} \omega).
	    \end{align*}
	}
	Note the relation 
	\begin{align}
	    C_+ (-\omega) = C_- (\omega), \qquad \omega\in T^* M.
	    \label{eq:Cpm_sym}
	\end{align}
\end{itemize}
We refer to \cref{fig:CanRel} for an illustration.

\begin{figure}[htpb]
    \centering
    \includegraphics{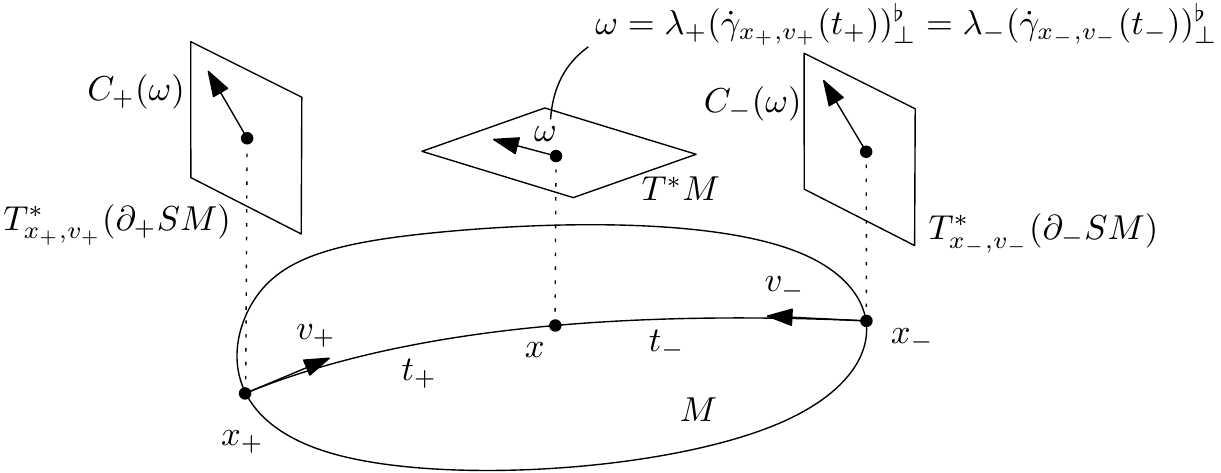}
    \caption{An illustration of the canonical relation of $I_0$. The functions $C_\pm(\omega)$ are defined in \cref{eq:Cpm}. }
    \label{fig:CanRel}
\end{figure}

As previously pointed out in \cref{sec_sampling}, the canonical relation of $I_0$ will induce maps from the semi-classical wavefront set of a function $f$ to that of its X-ray transform $I_0 f$. Indeed, applying \cref{FIO1} to $I_0$, we obtain, for every $f_h$ localized in phase space: 
\begin{align}
    \WFH (I_0 f) \backslash 0 \subset C_+ (\WFH (f) \backslash 0) \cup C_- (\WFH (f) \backslash 0),
    \label{eq:WFhmapI0}
\end{align}
where $C_{\pm}$ are defined in \cref{eq:Cpm}.

\subsection{Canonical relation of the X-ray transform in coordinate systems} \label{sec:canrelcoords}

We now explain how to view \cref{eq:linsys2} in two coordinate systems. In \cref{sec:fanbeam}, we cover fan-beam coordinates $(s,\alpha)$, where $s\in [0,L]$ parameterizes a point on $\partial M$ in $g$-arclength, and $\alpha\in [-\pi/2,\pi/2]$ parameterizes the angle of a tangent vector w.r.t. $\nu_s$. In \cref{sec:parallel}, we cover another coordinate system reminiscent to parallel geometry which holds globally on any simple surface.
    
\subsubsection{Fan-beam coordinates \texorpdfstring{$(s,\alpha)$}{(s,alpha)}}
\label{sec:fanbeam}

Suppose $\partial M$ is parameterized by arclength w.r.t. $g$, that is
\begin{align}
    \partial M = \{\Gamma(s):\ s\in [0,L]\}, \qquad g(\dot\Gamma(s), \dot\Gamma(s)) = 1, \qquad s\in [0,L],
    \label{eq:arclength}
\end{align}
and let boundary normal coordinates
\begin{align*}
    \Phi: [0,\epsilon]_\rho \times (\partial M)_s\ni (\rho,s) \mapsto \gamma_{s,\nu_s} (\rho) = \exp_s (\rho \nu_s) \in M,
\end{align*}
with $\nu_s$ the unit inward normal at $s\in \partial M$. In these coordinates, there exists a function $w(\rho,s)$ such that the metric tensor takes the form $\Phi^* g = d\rho^2 + e^{2w(\rho,s)} ds^2$, where $w(0,s) = 0$ for all $s\in [0,L]$ since $\partial M$ is parameterized in arclength, and  $-\partial_\rho w (0,s) = \two(s) >0$, with the latter being the second fundamental form at $\Gamma(s)\in \partial M$. A corresponding co-vector will be parameterized as $\xi d\rho + \eta ds$. In these variables, the geodesic equations are Hamilton's equations for the Hamiltonian ${\mathcal H}(\rho,s,\xi,\eta) = \frac{1}{2} (\xi^2 + e^{-2w(\rho,s)}\eta^2)$. On $\{{\mathcal H} = 1/2\}$, we further parameterize $\xi = \cos\alpha$, $\eta = - e^w \sin\alpha$ for some $\alpha\in \Sm^1$. The geodesic equation then becomes
\begin{align*}
    \dot \rho = \cos\alpha, \qquad \dot s = - e^{-w} \sin\alpha, \qquad \dot \alpha = - (\partial_\rho w) \sin\alpha.
\end{align*}
In these coordinates, 
\begin{align*}
    X &= \cos\alpha \,\partial_\rho - e^{-w} \sin\alpha \, \partial_s - \sin\alpha (\partial_\rho w) \partial_\alpha, \qquad V = \partial_\alpha, \\
    X_\perp &= [X,V] = \sin\alpha \,\partial_\rho + e^{-w} \cos\alpha\, \partial_s + \cos\alpha (\partial_\rho w) \partial_\alpha.
\end{align*} 
Moreover, we have $\mu = \cos\alpha$ and $\mu_\perp = V\mu = -\sin\alpha$, and in particular, 
\begin{align*}
    H = (\mu_\perp X + \mu X_\perp)|_{\rho=0} = \partial_s - \two(s) \partial_\alpha. 
\end{align*}
We arrive at the conclusions that 
\begin{align}
    d\alpha(H) = -\two(s), \qquad d\alpha(V) = 1, \qquad ds(H) = 1, \qquad ds(V) = 0.
    \label{eq:coords1transf}
\end{align}

\begin{lemma}\label{lem:etaFB}
    In fan-beam coordinates, the co-vector defined in \cref{eq:linsys2} has local expression $\eta = \eta_s ds + \eta_\alpha d\alpha$, where 
    \begin{align}
	\eta = \eta_s\ ds + \eta_\alpha\ d\alpha, \qquad \eta_\alpha = b(x,v,t), \qquad \eta_s = \two(s) b(x,v,t) -\mu a(x,v,t). 
	\label{eq:etaFB}
\end{align}    
\end{lemma}

\begin{proof} We use \cref{eq:coords1transf} to arrive at 
    \begin{align*}
	\eta(V) = \eta_\alpha, \qquad \eta(H) = \eta_s - \two(s) \eta_\alpha \rev{ \stackrel{\eqref{eq:linsys2}}{=} -\mu a(x,v,t)}.
    \end{align*}
    Equation \cref{eq:etaFB} follows by inverting and combining with \cref{eq:linsys2}.    
\end{proof}

\subsubsection{Generalized parallel coordinates $(w,\t)$}
\label{sec:parallel}

\rev{We now introduce another coordinate system which is reminiscent of the classical parallel geometry (see \cref{rem:eucl}). Unlike fan-beam coordinates, these coordinates behave simply with respect to the involution of $\partial_+ SM$ fixing functions in the range of the X-ray transform. To make this point more precise, let us define the {\bf antipodal scattering relation} as the map $S_A\colon \partial_+ SM\to \partial_+ SM$ given by $S_A(x,v) := S(x,-v)$. Upon noticing the elementary property
\begin{align}
    \gamma_{x,v}(t) = \gamma_{S_A(x,v)}(\tau-t), \qquad t\in [0,\tau], \qquad \tau = \tau(x,v) = \tau(S_A(x,v)), 
    \label{eq:SA}
\end{align}
then we directly have that $I_0 f(x,v) = I_0 f(S_A(x,v))$ for any function $f$ and any $(x,v)\in \partial_+ SM$. 

On any simple Riemannian surface $(M,g)$ with boundary length $L$, we now construct a global coordinate chart $(w,\t)\in [0,L]\times [0,L/2]$ on $\partial_+ SM$ such that 
\begin{align}
    (w,\t)\circ S_A = (w+ L/2 \mod L, L/2- \t),
    \label{eq:wpsym}
\end{align}
thereby making the symmetries of $I_0 f$ more direct. \Cref{lem:etaWT} below also shows that the canonical relation of $I_0$ in these coordinates has symmetries making sampling constraints independent of whether one is attempting to resolve a singularity $\omega\in TM$ from data near $C_+(\omega)$ or $C_-(\omega)$. 

The construction of these coordinates is somewhat related to the two-boundary-point model: define the map
\begin{align}
    \Psi\colon \partial_+ SM\to (\partial M)_x \times (\partial M)_{x'}, \qquad \Psi(x,v) := (x,\pi(S(x,v))),
    \label{eq:Psi}
\end{align}
with $S$ the scattering relation defined in \cref{eq:scatrel}. } Using the arclength parameterization \cref{eq:arclength} of $\partial M$ (thought of as an $L$-periodic map), we can locally identify $(\partial M)_x \times (\partial M)_{x'}$ with subsets of $\Rm^2$ through the map $(s,s')\mapsto (\Gamma(s),\Gamma(s'))$, and the tangent space at a point $(\Gamma(s),\Gamma(s'))$ is spanned by $\partial_s, \partial_{s'}$. \rev{The proof of the following lemma is given in Appendix \ref{sec:app}}. 

\begin{lemma}\label{lem:Psi}
    $\Psi$ defined in \cref{eq:Psi} satisfies, for every $(x,v)\in \partial_+ SM$, 
    \begin{align}
	d\Psi (V) = \frac{-b(x,v,\tau)}{\mu\circ S}\ \partial_{s'}, \qquad d\Psi (H) = \partial_s + \frac{\mu}{\mu\circ S} a(x,v,\tau)\ \partial_{s'}. 
	\label{eq:dPsi}
    \end{align}    
    In particular, if $(M,g)$ is simple, $\Psi$ is a global diffeomorphism.
\end{lemma}

From \cref{lem:Psi}, when $(M,g)$ is simple, we can think of $(s,s')$ as a global coordinate chart from $\partial_+ SM$ to an appropriate portion of $\Rm^2$. \rev{To describe such a portion, a two-point configuration $(\Gamma(s),\Gamma(s'))$ can be described fixing $s\in \Rm$ and $s'\in [s,s+L)$. This representation becomes injective once we fix a fundamental domain for the $\Zm$ action $k\cdot (s,s') = (s+kL, s'+kL)$ on $\{s\in \Rm,\ s\le s'<s+L\}$. Upon rotating the axes by introducing the coordinates
\begin{align}
    w = \frac{s+s'}{2}\circ\Psi, \qquad \t = \frac{s'-s}{2}\circ\Psi,
    \label{eq:wt}
\end{align}
the coordinate $\t$ naturally lies in $[0,L/2]$, while a segment of length $L$, say $[0,L]$, for $w$ will give the fundamental domain $[0,L]_w\times [0,L/2]_\t$. To see that $(w,\t)$ satisfy \cref{eq:wpsym}, observe that $\Psi(x,v)$ and $\Psi(S_A(x,v))$ have their points interchanged. Therefore, in the region $\{s\le s'<s+L\}$, the map $S_A$ corresponds to the transformation $(s,s')\mapsto (s',s+L)$, under which the coordinates $(w,\t)$ transform according to \cref{eq:wpsym}.

In the coordinates $(w,\t)$, we can think of $\partial_+ SM = [0,L]_w \times [0,L/2]_\t$ with $(0,\t)$ and $(L,\t)$ identified for all $\t\in [0,L/2]$. The boundary components of $\partial_+ SM$, $\{(\Gamma(s),v)\in \partial_+ SM,\ v = \partial_s\}$ and $\{(\Gamma(s),v)\in \partial_+ SM,\ v = -\partial_s\}$ identify with $\{\t=0\}$ and $\{\t=L/2\}$, respectively. 

We end this section by making explicit the canonical relation of $I_0$ in these coordinates.
}

\begin{lemma}\label{lem:etaWT}
    In parallel coordinates $(x,v) = (w,\t)$, the co-vector defined in \cref{eq:linsys2} has local expression 
    \begin{align}
	\begin{split}
	    \eta &= \eta_w\ dw + \eta_\t\ d\t, \quad \text{where} \quad \left\{
	    \begin{array}{cl}
		\eta_w &= \mu(S_A(x,v)) \ \frac{b(x,v,t)}{b(x,v,\tau)} - \mu(x,v) \frac{b(S_A(x,v),\tau-t)}{b(S_A(x,v),\tau)}, \\
		\eta_\t &= \mu(S_A(x,v))\ \frac{b(x,v,t)}{b(x,v,\tau)} + \mu(x,v) \frac{b(S_A(x,v),\tau-t)}{b(S_A(x,v),\tau)}.
	    \end{array}
	    \right. 
	\end{split}
	\label{eq:etaWT}
    \end{align}    
\end{lemma}

\begin{proof}
    We first show that 
    \begin{align}
	\eta_w = \frac{\mu a_\tau - \mu\circ S}{b_\tau} b(x,v,t) - \mu a(x,v,t), \qquad \eta_\t = \frac{-\mu a_\tau - \mu\circ S}{b_\tau} b(x,v,t) + \mu a(x,v,t),
	\label{eq:firstExpr}
    \end{align}
    where we denote for short $a_\tau (x,v) = a(x,v,\tau(x,v))$ and similarly for $b_\tau$. To this end, we compute the change of basis from $(H,V)$ to $(\partial_w, \partial_\t)$:
    \begin{align*}
	dw(V) &= d\Psi(V) \left( \frac{s+s'}{2} \right) \stackrel{\cref{eq:dPsi}}{=} \frac{-b_\tau}{2\mu\circ S}, \qquad \text{and similarly,} \\
	dw(H) &= \frac{\mu\circ S + \mu a_\tau}{2 \mu\circ S}, \qquad d\t(V) = \frac{-b_\tau}{2\mu\circ S}, \qquad d\t(H) = \frac{-\mu\circ S+ \mu a_\tau}{2\mu\circ S}.
    \end{align*}
    Now we seek for a covector in the form $\eta = \eta_w\ dw + \eta_\t\ d\t$ solving \cref{eq:linsys2}. This looks like
    \begin{align*}
	\left[
	    \begin{array}{cc}
		-b_\tau & -b_\tau \\ \mu a_\tau + \mu\circ S & \mu a_\tau - \mu\circ S	
	    \end{array}
	\right] \!
	\left[
	    \begin{array}{c}
		\eta_w \\ \eta_\t
	    \end{array}
	\right]\! =\! 2\mu\circ S \left[
	    \begin{array}{c}
		\eta(V) \\ \eta(H)
	    \end{array}
	\right] \stackrel{\cref{eq:linsys2}}{=} 2\mu\circ S \left[
	    \begin{array}{c}
		b(x,v,t) \\ -\mu a(x,v,t) 
	    \end{array}
	\right]\! .
    \end{align*}
    Inverting this system yields \cref{eq:firstExpr}. Now to go from \cref{eq:firstExpr} to \cref{eq:etaWT}, we first rewrite \cref{eq:firstExpr} as 
    \begin{align*}
	\eta_w &= \frac{1}{b_\tau} (-\mu\circ S\ b(x,v,t) + \mu (a_\tau b(x,v,t) - b_\tau a(x,v,t)), \\
	\eta_\t &= \frac{1}{b_\tau} (-\mu\circ S\ b(x,v,t) - \mu (a_\tau b(x,v,t) - b_\tau a(x,v,t)).
    \end{align*}
    Now, first notice that $-\mu \circ S = \mu \circ S_A$ (equivalent to saying that $\cos(\theta+\pi) = -\cos\theta$). Finally, \cref{eq:etaWT} follows upon noticing that 
    \begin{align*}
	a_\tau(x,v) b(x,v,t) - b_\tau(x,v) a(x,v,t) = - b(S_A(x,v),\tau-t), \qquad 0\le t\le \tau = \tau(x,v).
    \end{align*}
    That this is true follows from the fact that both left- and right-hand sides are the unique solution to the initial value problem
    \begin{align*}
	\ddot c (x,v,t) + \kappa(\gamma_{x,v}(t)) c(x,v,t) = 0, \qquad c(x,v,\tau) = 0, \qquad \dot c(x,v,\tau) = 1,
    \end{align*}
    using \cref{eq:SA}. Evaluating at $t=0$ also yields the relation $b_\tau(x,v) = b_{\tau}(S_A(x,v))$, and hence \cref{eq:etaWT} follows. 
\end{proof}

\subsection{Sensitivity of \cref{eq:WFhmapI0} to geometric parameters and coordinate systems}\label{sec:sensitivity}

What drives most of the sampling heuristics that follow is the mapping property \cref{eq:WFhmapI0}, which at the elementary level, is encoded in the covectors $\eta_{x,v,t}$ given in Equation \cref{eq:linsys2}. Relative to a given coordinate system, the larger the dual components of these covectors, the finer the sampling of $I_0 f$ will have to be in order to resolve details of $f$ located near $\gamma_{x,v}(t)$ with direction $(\dot\gamma_{x,v}(t))_\perp$. From expressions \cref{eq:etaFB} and \cref{eq:etaWT}, we can therefore draw the following general conclusions, some of which will be illustrated in \cref{sec:numerics}: 
\begin{enumerate}
    \item The components of \cref{eq:etaFB} can become large due to large Jacobi fields along long geodesics. In particular, sampling in fan-beam coordinates is sensitive to Jacobi field magnitude, and can lead to prohibitive sampling requirements in negative curvature at ``long'' geodesics. By contrast, the components of \cref{eq:etaWT}, expressed as ratios of Jacobi fields, are better behaved in that case. 
    \item The components of \cref{eq:etaFB} are not sensitive to the emergence of conjugate points per se. By contrast, as the quantity $b_\tau$ becomes small for $(x,v)$ not necessarily tangential (predicting the emergence of conjugate pairs close to $\partial M$ but outside of $M$), the components of $\eta$ in parallel coordinates \cref{eq:etaWT} can become large, thus degrading the sampling rates near those geodesics. This is to be expected since this coordinate system cannot be made global in the presence of conjugate points. 
    \item The components of $\eta_{x,v,t}$ and $\eta_{S_A(x,v)}(\tau-t)$ in fan-beam coordinates \cref{eq:etaFB}, images of the same singularity under the canonical relation of $I_0$, can have quite different sizes, and as such, can drive different sampling requirements on the same details of $f$. By contrast, the expression \cref{eq:etaWT} in parallel coordinates of both covectors is identical in size, and this issue does not appear. 
\end{enumerate}

\subsection{Simple geodesics disks in constant curvature spaces}
\label{sec:CCDs}

As the numerical illustrations that follow will focus on a benchmark case of simple geodesic disks in constant curvature spaces, we gather all important formulas and computations in the present section. Following notation in \cite{Mishra2019,Monard2019a}, let $\Dm_R = \{ (x,y)\in \Rm^2, x^2+y^2\le R^2\}$, and for $\kappa\in \Rm$ such that $|\kappa R^2|<1$, equip $\Dm_R$ with the metric
\begin{align}
    g_\kappa(z) = c_\kappa(z)^{-2} |dz|^2, \qquad c_\kappa(z) = 1+\kappa |z|^2, \qquad z := x+iy,
    \label{eq:gkappa}
\end{align}
of constant curvature $4\kappa$. For conciseness of notation, complex notation will be used, writing a point $z = x+iy \in \Dm_R$, and a tangent vector $v = v_x \partial_x + v_y \partial_y\in T_{(x,y)}\Dm_R$ as $v_x + i v_y$ (the almost-complex structure corresponds to multiplication by $i$). In what follows, the dependence on $R$ and $\kappa$ of many quantities will be left implicit. 

\subsubsection{Relevant geometric quantities}

By direct computation, the boundary $\partial \Dm_R$ has $g_\kappa$-length 
\begin{align}
    L = \frac{2\pi R}{1+\kappa R^2},
    \label{eq:length}
\end{align}
and a parameterization by $g_\kappa$-arclength is given by $\Gamma(s) = R e^{i 2\pi s/L}$ for $s\in [0,L]$, with inward-pointing normal $\nu(s) = i\dot\Gamma(s) = (1+\kappa R^2) i e^{i 2\pi s/L}$. Its second fundamental form is constant and, using the Gauss-Bonnet theorem on $\Dm_R$, equal to 
\begin{align}
    \two(s) = - g_\kappa(\nabla_{\dot\Gamma(s)} \dot\Gamma (s), \nu(s)) = \frac{1-\kappa R^2}{R}, \qquad s\in [0,L],
    \label{eq:two}
\end{align}
with $\nabla$ the $g_\kappa$-Levi-Civita connection. Equations \cref{eq:length} and \cref{eq:two} tell us how simplicity degenerates in limiting cases: as $\kappa R^2 \to 1$, the boundary curvature goes to zero, as $(\Dm_R, g_\kappa)$ becomes a hemisphere (of totally geodesic boundary); as $\kappa R^2\to -1$, the boundary length becomes infinite, and so does any geodesic, as $(\Dm_R, g_\kappa)$ becomes a full, noncompact hyperbolic space. 

\subsubsection{Fan-beam coordinates}

The fan-beam chart in this case is given by 
\begin{align*}
    [0,L]\times [-\pi/2,\pi/2]\ni (s,\alpha) \mapsto (x = \Gamma(s), v = c_\kappa(\Gamma(s)) e^{i (2\pi s/L + \pi + \alpha)}),
\end{align*}
and following \cite[Sec. 2]{Mishra2019}, the unique unit-speed $g_\kappa$-geodesic passing through $(s,\alpha)$ at $t=0$ has expression
\begin{align}
    \gamma_{s,\alpha}(t) = T_{s,\alpha} (x(t)) = e^{2\pi i s/L} \frac{R- x(t) e^{i\alpha}}{1+\kappa R e^{i\alpha} x(t)}, 
    \label{eq:geodesic}
\end{align}
where 
\begin{align}
    x(t) := \left\{
    \begin{array}{ll}
	\frac{1}{\sqrt{-\kappa}} \tanh (\sqrt{-\kappa}\ t), & t\in \Rm,\ \kappa <0, \\
	t, & t\in \Rm,\ \kappa = 0, \\
    \frac{1}{\sqrt{\kappa}}\tan(\sqrt{\kappa}\ t), & t\in \left( -\frac{\pi}{2\sqrt{\kappa}}, \frac{\pi}{2\sqrt{\kappa}} \right),\ \kappa>0.
    \end{array}
\right.
\label{eq:xoft}
\end{align}
Such a geodesic exits the domain at the smallest $t>0$ such that $|T_{s,\alpha}(x(t))|^2 = R^2$. Solving this equation produces a unique solution $x(\tau) = \frac{2R}{1-\kappa R^2} \cos \alpha$, which then determines $\tau(s,\alpha) = \tau(\alpha)$:
\begin{align}
    \tau(\alpha) = \left\{
    \begin{array}{ll}
	\frac{1}{\sqrt{-\kappa}} \tanh^{-1} \left( \frac{2\sqrt{-\kappa}R}{1-\kappa R^2} \cos\alpha \right), & \kappa<0, \\
	2R\cos\alpha, & \kappa = 0, \\
	\frac{1}{\sqrt{\kappa}} \tan^{-1} \left( \frac{2\sqrt{\kappa}R}{1-\kappa R^2} \cos\alpha \right), & \kappa>0.
    \end{array}
    \right.
    \label{eq:tau}
\end{align}
Since the curvature is constant, the scalar Jacobi functions $a(x,v,t)$ and $b(x,v,t)$ are independent of $(x,v)$, with expressions 
\begin{align}
    a(t) = \left\{
    \begin{array}{ll}
	\cosh (2\sqrt{-\kappa}\ t) & \kappa<0, \\
	1 & \kappa = 0, \\
	\cos (2\sqrt{\kappa}\ t) & \kappa>0,
    \end{array} 
    \right. \qquad 
    b(t) = \left\{
    \begin{array}{ll}
	\frac{1}{2\sqrt{-\kappa}} \sinh (2\sqrt{-\kappa}\ t) & \kappa<0, \\
	t & \kappa = 0, \\
	\frac{1}{2\sqrt{\kappa}} \sin (2\sqrt{\kappa}\ t) & \kappa>0.
    \end{array}
    \right.
    \label{eq:CCD_ab}
\end{align}

\begin{remark} Expressions \cref{eq:xoft}-\cref{eq:CCD_ab} can all be regrouped into a single expression upon setting $\sqrt{\kappa} = i \sqrt{-\kappa}$ when $\kappa<0$ and using the relations between trig and htrig functions.     
\end{remark}

The following generalized trigonometric identities hold regardless of the sign of $\kappa$: 
\begin{align}
    a(t \pm t') = a(t)a(t') \mp 4\kappa b(t) b(t'), \qquad b(t \pm t') = b(t) a(t') \pm a(t) b(t'), \qquad t,t'\in \Rm.
    \label{eq:trig}
\end{align}

\smallskip
\noindent{\bf Computation of $\eta_{s,\alpha,t}$:} Fixing $(s,\alpha)\in \partial_+ S_{(\kappa)} \Dm_R$, we can now compute the expression of $\eta_{s,\alpha,t}$ in \cref{eq:etaFB}: for $t\in (0,\tau(\alpha))$
\begin{align}
    \begin{split}
	\eta_{s,\alpha, t} &= (\two b(t) - \cos \alpha a(t))\ ds + b(t)\ d\alpha \\
	&= \left( \frac{1-\kappa R^2}{R} b(t) - \cos\alpha\ a(t) \right)\ ds + b(t)\ d\alpha.	
    \end{split}
    \label{eq:etasal}    
\end{align}

\subsubsection{Generalized parallel coordinates}

From \cref{eq:geodesic} and \cref{eq:tau}, the other endpoint of $\gamma_{s,\alpha}$ has expression $T_{s,\alpha} (x(\tau(\alpha))) = \Gamma(s')$, where
\begin{align}
    s'(s,\alpha) = s + \frac{L}{2\pi} \left( \pi + 2 \ss(\alpha) \right), \qquad \ss(\alpha) := \tan^{-1} \left( \frac{1-R^2\kappa}{1+R^2\kappa} \tan\alpha\right).
    \label{eq:endpoint}
\end{align}
The generalized parallel coordinates are then given by 
\begin{align}
    \begin{split}
	w(s,\alpha) &= \frac{s'+s}{2} \mod L = s + \frac{L}{2\pi} \left( \frac{\pi}{2} + \ss(\alpha) \right) \mod L, \\
	\t(s,\alpha) &= \frac{s'-s}{2} = \frac{L}{2\pi} \left( \frac{\pi}{2} + \ss(\alpha) \right), 	
    \end{split}
    \label{eq:CCD_FB2WT}
\end{align}
with inverse map $(w,\t)\mapsto (s,\alpha)$ given by
\begin{align}
    s(\t,w) = w - \t \mod L, \qquad \alpha(\t,w) = \ss^{-1} \left( \frac{2\pi \t}{L} - \frac{\pi}{2} \right). 
    \label{eq:CCD_WT2FB}
\end{align}

\begin{remark}[Euclidean case]\label{rem:eucl} In the Euclidean unit disk, where $L = 2\pi$ and $\ss = id$, one obtains
    \begin{align*}
	w = s + \alpha + \frac{\pi}{2}, \qquad \t = \alpha + \pi/2. 
    \end{align*}
    Traditional parallel coordinates for the Radon transform are $(q, \theta)$, where $q$ denotes the distance of the line to the origin, and $\theta$ is orthogonal to the direction of integration. In our case, $w$ recovers $\theta$, and using that $\alpha = \sin^{-1} q$, we have that 
    \begin{align*}
	\t = \frac{\pi}{2} + \sin^{-1} q.
    \end{align*}
    This change of variable is diffeomorphic from $[-1,1]$ to $[0,\pi]$ with singular Jacobian at the endpoints due to squareroot singularities (consider the relation $\sin p = \sqrt{1-q^2}$ near $p=0,\pi$).
\end{remark}

\smallskip
\noindent{\bf Computation of $\eta_{w,\t,t}$:} We start from \cref{eq:etaWT} and first notice that in this rotationally symmetric case, $\mu(S_A(x,v)) = \mu(x,v)$, and thus
\begin{align*}
     \eta = \eta_w\ dw + \eta_\t\ d\t, \qquad \eta_w = \mu \frac{b(t) - b(\tau-t)}{b(\tau)}, \qquad \eta_p = \mu \frac{b(t) + b(\tau-t)}{b(\tau)}.
\end{align*}
Now observe from the trigonometric identities \cref{eq:trig} that 
\begin{align*}
    b(t_1) \pm b(t_2) = 2 b\left( \frac{t_1\pm t_2}{2} \right) a\left( \frac{t_1 \mp t_2}{2} \right), \qquad t_1, t_2\in \Rm. 
\end{align*}
Hence
\begin{align*}
    \eta_w = \mu \frac{2 b(t-\tau/2) a(\tau/2)}{2b(\tau/2) a(\tau/2)} = \mu \frac{b(t-\tau/2)}{b(\tau/2)}, \quad \eta_p = \mu \frac{2 b(\tau/2) a(t-\tau/2)}{2b(\tau/2) a(\tau/2)} = \mu \frac{a(t-\tau/2)}{a(\tau/2)}.
\end{align*}

We record the conclusion here: 
\begin{align}
    \eta_{w,\t,t} = \mu\left( \frac{b(t-\tau/2)}{b(\tau/2)}\ dw + \frac{a(t-\tau/2)}{a(\tau/2)}\ d\t \right), \quad 0\le t\le \tau. \quad (\mu = \mu(\t),\ \tau = \tau(\t)).
    \label{eq:etaCCD_WT}
\end{align}

\section{Numerical illustrations}
\label{sec:numerics}

In this section, we will use the language of {\em classical} sampling by replacing the small parameter $h>0$ with a large parameter (proportional to) the band limit $B$ and, on the discrete size, the meshsize will be replaced by the number of gridpoints $N$. All Fourier transforms will be classical (no $h$-rescaling) and, although discretized via Discrete Fourier Transforms, space variables and dual variables will be rescaled so as to represent discretization-independent quantities $x,\xi$ at the discrete level. The analysis of \cref{sec_sampling} then applies asymptotically as $B\to \infty$ and $N\to \infty$, though for the purposes of the experiments, we will see that these asymptotic results are well-illustrated for finite $B,N$. 

\subsection{Preliminaries: phantoms, band limit and bowties} \label{sec:prelim}

We will restrict the experiments below to the case $R=1$ (unit disk), varying $\kappa$. The phantoms used are displayed \cref{fig:functions}. 

\begin{figure}[htpb]
    \centering
    \begin{subfigure}{0.23\textwidth}
	\centering
	\includegraphics[height=0.13\textheight]{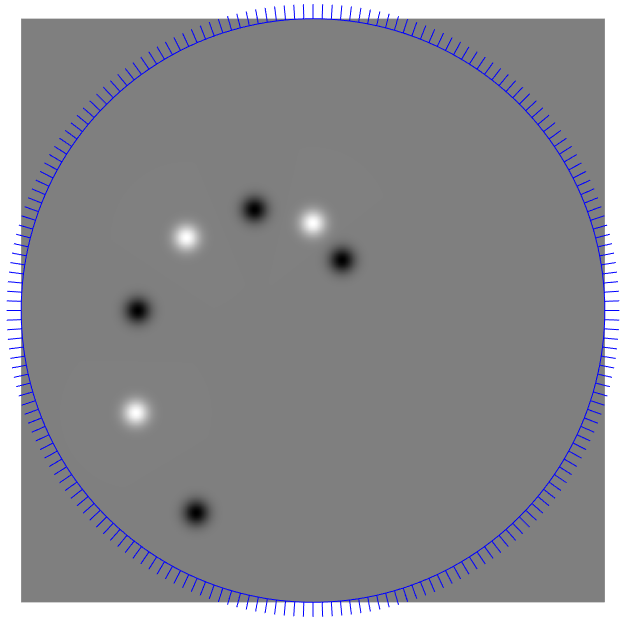} \\
	\includegraphics[height=0.14\textheight]{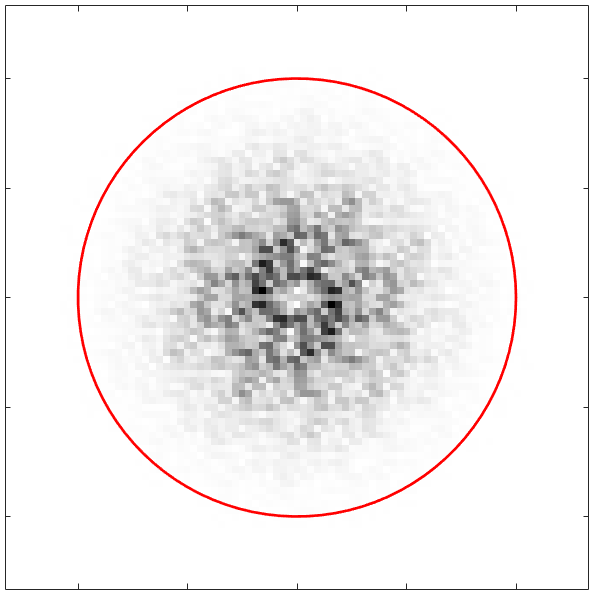} \\	
	\caption{$f_0$ defined in \\ Example 2}
	\label{fig:f0}
    \end{subfigure}
    \begin{subfigure}{0.23\textwidth}
	\centering
	\includegraphics[height=0.13\textheight]{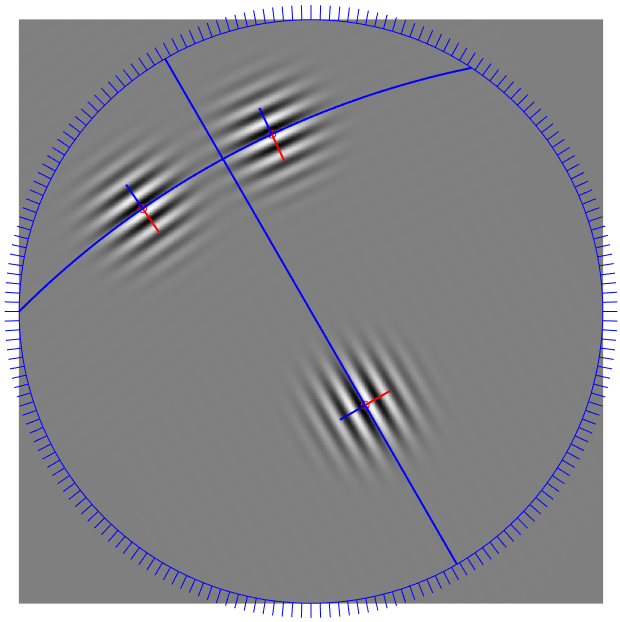} \\
	\includegraphics[height=0.14\textheight]{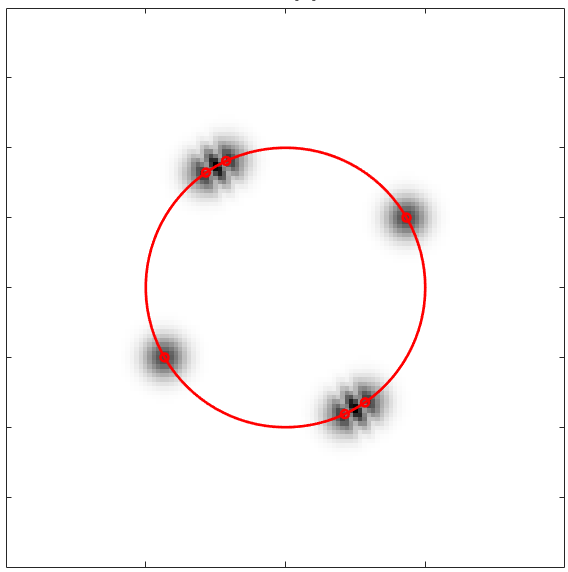}
	\caption{$f_1$ defined in \cref{eq:f1} \\ for $\kappa=0.3$}
	\label{fig:Ex1}    
    \end{subfigure}
    \begin{subfigure}{0.23\textwidth}
	\centering
	\includegraphics[height=0.13\textheight]{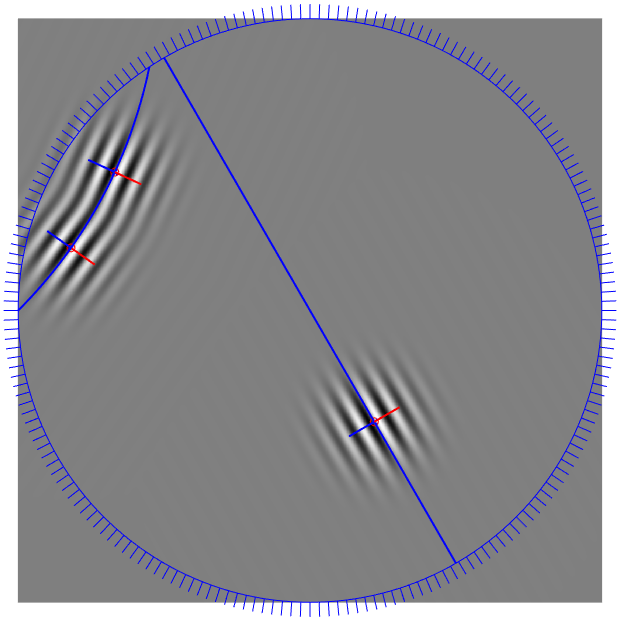} \\
	\includegraphics[height=0.14\textheight]{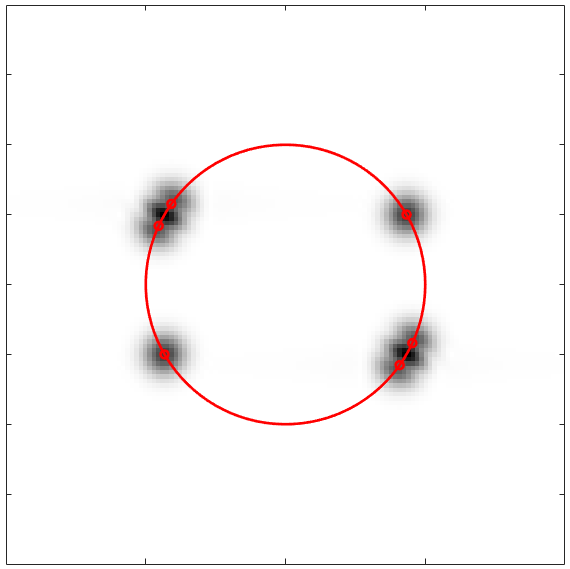}
	\caption{$f_1$ defined in \cref{eq:f1} \\ for $\kappa=-0.3$}
	\label{fig:Ex2}   
    \end{subfigure}
    \begin{subfigure}{0.25\textwidth}
	\centering
	\includegraphics[height=0.13\textheight]{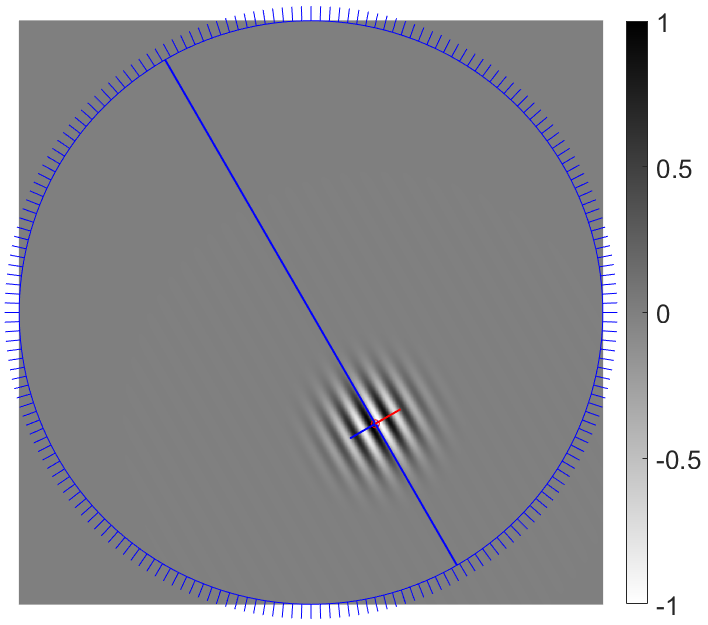} \\
	\includegraphics[height=0.14\textheight]{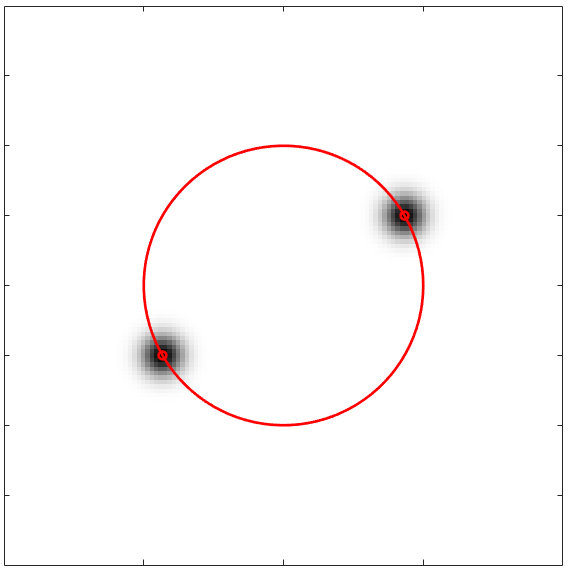}
	\caption{$f_4$, \\ $\kappa = -0.3$}
	\label{fig:Ex4}   
    \end{subfigure}    
    \caption{Top row: the functions appearing in all experiments, supported in the unit disk $\Dm_1$, all on the same color scale. Bottom row: modulus of their Fourier transform, with the circle $\{|\xi|=100\}$ in red ($\xi$: dual variable for the continuous Fourier transform).
    }
    \label{fig:functions}
\end{figure}

\smallskip
\noindent{\bf Band limits.} The band limit condition \cref{eq:band limit} need not be the most natural in geometric contexts, especially when mapping band limit requirements through the canonical relation of $I_0$. To be slightly more general, we can think that imposing a band limit on $f_h \colon M\to \Rm^n$ is done using a set-theoretic constraint of the form
\begin{align}
    \WFH (f) \subset B\cdot\Sigma:= \{(x,B\omega): (x,\omega)\in \Sigma\}
    \label{eq:band limit2}
\end{align}
for some 'unit' bounded set $\Sigma \subset T^* M$ and number $B>0$. Natural assumptions on $\Sigma$ are that its fibers are {\em absorbing} and {\em even}, in the following sense: for all $x\in M$, if $\omega \in \Sigma_x$, then $\lambda\omega\in \Sigma_x$ for all $\lambda\in [0,1]$ and $-\omega\in \Sigma_x$. We will then say that $f$ is $B\Sigma$-band limited whenever it satisfies \cref{eq:band limit2}.

For example, condition \cref{eq:band limit} corresponds to taking $\Sigma$ to be the unit Euclidean co-cube bundle of $M$. Below, however, it will be more natural to consider either of the following cases: 
\begin{itemize}
    \item If $M=\Rm^n$, assuming that $\hat f$ is supported in the (Euclidean) ball of radius $B$ corresponds to taking $\Sigma$ to be the unit Euclidean co-ball bundle of $M$, denoted $B_e^* M$. From the Fourier plots \cref{fig:functions}, all integrands considered satisfy a band limit condition of this form.
    \item $\Sigma$ could also be the unit co-ball bundle of $M$ associated with some metric. In this context of geodesic X-ray transform, assuming that the X-ray transform is computed using Riemann sums in intrinsic geodesic time with stepsize $h$, constraints on $h$ in terms of the band limit of $f$ would be the sharpest when expressed in terms of $\Sigma = B^* M$. 
    \item If $f$ has support restrictions, one may further enforce the condition $[\Sigma]_x = \{0\}$ for all $x$'s not belonging to the support of $f$.  
\end{itemize}

\smallskip
\noindent{\bf Bowties.} Under a band limit assumption of the form \cref{eq:band limit2} on a function $f$, to control the Fourier support of $I_0 f$, equation \cref{eq:WFhmapI0} will give 
\begin{align*}
    \WFH (I_0 f)\backslash\{0\} \subset C_+ (B\cdot\Sigma) \cup C_-(B\cdot \Sigma) = B\cdot (C_+ (\Sigma)\cup C_-(\Sigma)),
\end{align*}
where the last equality comes from the fiberwise homogeneity of the canonical relation. Since $\Sigma$ is fiberwise even, $C_+ (\Sigma) = C_+ (-\Sigma) = C_-(\Sigma)$, and thus the previous equation simplifies into
\begin{align*}
    \WFH (I_0 f)\backslash\{0\} \subset B\cdot C_+ (\Sigma),
\end{align*}
where the right-hand side is a set-theoretic upper bound on the Fourier content of $I_0 f$. In what follows, we will often draw \rev{the boundaries of the sets $[C_+(B\cdot\Sigma)]_{(x,v)} \subset T^*_{(x,v)} \partial_+ SM$} for $(x,v)\in \partial_+ SM$. Such sets depend on the \rev{base point $(x,v)$ and on the} coordinate system we choose\rev{. In} the examples below, they generally take the shape of vertically oriented 'bowties', see e.g. \cref{fig:Ex1_data} through \cref{fig:bowtiesE3}. \rev{This is because a given set $[C_+(B\cdot\Sigma)]_{(x,v)}$ is always a subset of (a radially scaled version of) $[C_+ (B^* M)]_{(x,v)} = \{\lambda \eta_{x,v,t},\ \lambda\in [-1,1],\ t\in [0,\tau(x,v)]\}$. Based on formulas \cref{eq:etasal} and \cref{eq:etaCCD_WT}, the portion of the latter set for $\lambda\in [0,1]$ is included in a cone in either the half-fiber $\{\eta_\alpha\ge 0\}$ in fan-beam coordinates, or $\{\eta_\t\ge 0\}$ in parallel coordinates, with aperture less than $\pi$ (in fact, exactly $\pi/2$ in parallel coordinates).} 

\subsection{Coordinate sytems} 

\smallskip
\noindent{\bf Example 1: change of coordinates.} The changes of coordinates $(s,\alpha) \leftrightarrow (w,\t)$ given in equations \cref{eq:CCD_FB2WT} and \cref{eq:CCD_WT2FB} are displayed in \cref{fig:coords}, by showing the images of Cartesian grids for each coordinate system.  

\begin{figure}[htpb]
    \centering
    \begin{subfigure}[b]{\textwidth}
	\centering
	\begin{tabular}{ccccc}
	    \includegraphics[trim = 30 30 30 30, clip, width=0.16\textwidth]{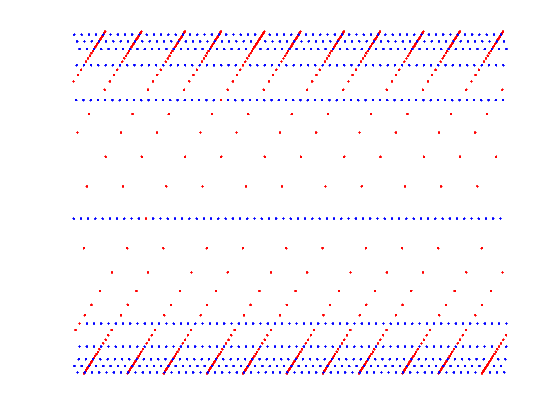} &
	    \includegraphics[trim = 30 30 30 30, clip, width=0.16\textwidth]{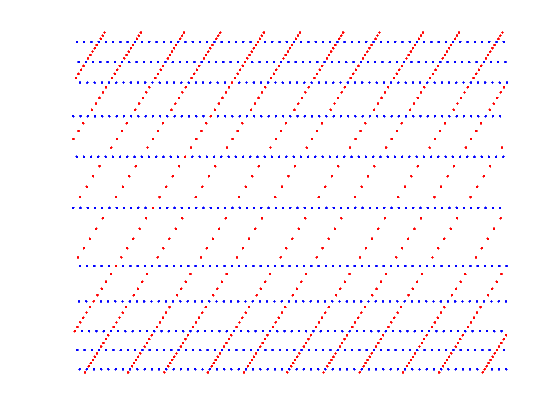} &
	    \includegraphics[trim = 30 30 30 30, clip, width=0.16\textwidth]{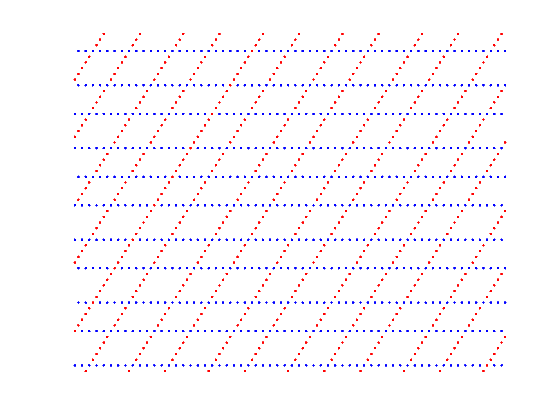} &
	    \includegraphics[trim = 30 30 30 30, clip, width=0.16\textwidth]{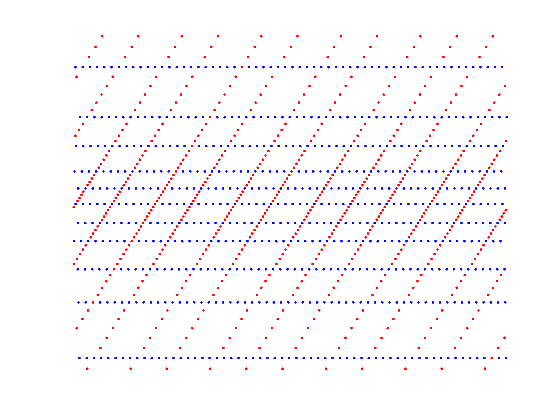} &
	    \includegraphics[trim = 30 30 30 30, clip, width=0.16\textwidth]{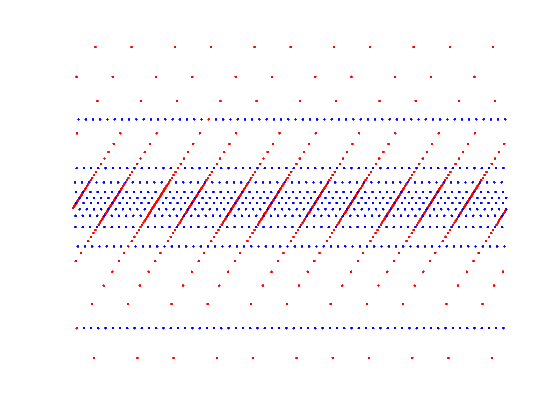} \\
	    $\kappa=-0.7$ & $\kappa = -0.3$ & $\kappa = 0$ & $\kappa=0.3$ & $\kappa = 0.7$
	 \end{tabular}
	\caption{The image of an equispaced Cartesian grid $(s,\alpha)\in [0,L]\times [-\pi/2,\pi/2]$ viewed in $(w,\t)\in [0,L]\times [0,L/2]$ (iso-$s$ in red, iso-$\alpha$ in blue).}
	\label{fig:FB2WP}
    \end{subfigure}
    \begin{subfigure}[b]{\textwidth}
	\centering
	\begin{tabular}{ccccc}
	    \includegraphics[trim = 30 30 30 30, clip, width=0.16\textwidth]{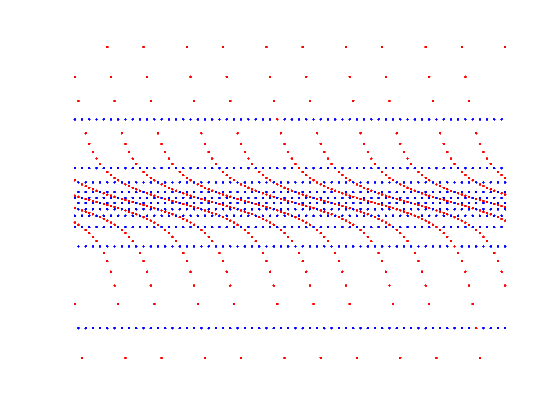} &
	    \includegraphics[trim = 30 30 30 30, clip, width=0.16\textwidth]{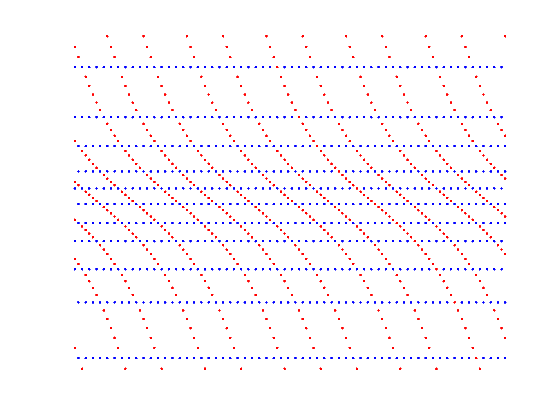} &
	    \includegraphics[trim = 30 30 30 30, clip, width=0.16\textwidth]{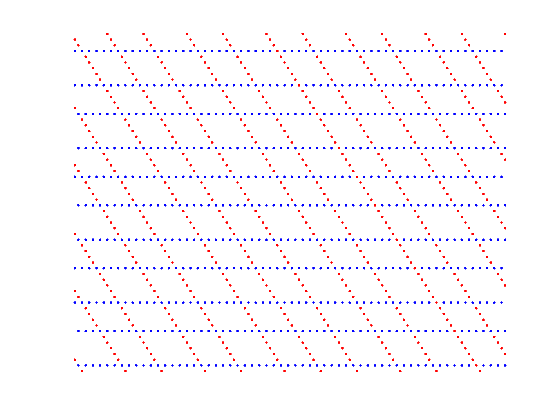} &
	    \includegraphics[trim = 30 30 30 30, clip, width=0.16\textwidth]{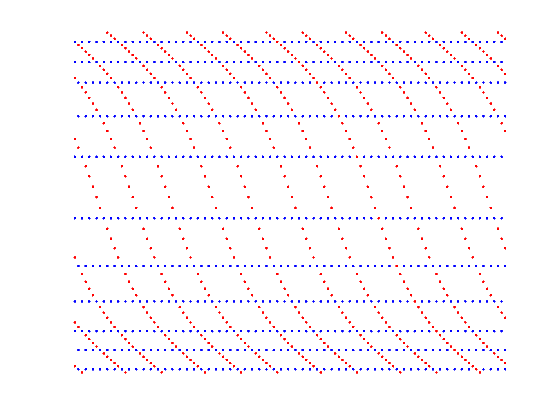} &
	    \includegraphics[trim = 30 30 30 30, clip, width=0.16\textwidth]{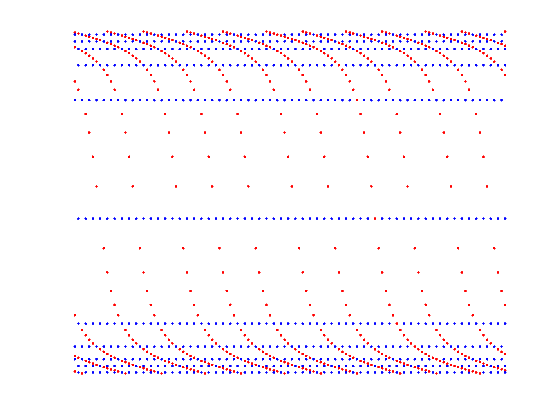} \\	
	    $\kappa=-0.7$ & $\kappa = -0.3$ & $\kappa = 0$ & $\kappa=0.3$ & $\kappa = 0.7$
	 \end{tabular}
	\caption{the image of an equispaced Cartesian grid $(w,\t)\in [0,L]\times [0,L/2]$ viewed in $(s,\alpha)\in [0,L]\times [-\pi/2,\pi/2]$ (iso-$w$ in red, iso-$\t$ in blue).}
	\label{fig:WP2FB}
    \end{subfigure}
    \caption{Example 1. Changes of coordinates $(s,\alpha)\leftrightarrow (w,\t)$.}
    \label{fig:coords}
\end{figure}

\smallskip
\noindent{\bf Example 2: $I_0 f$ for various values of $\kappa$ in both coordinate systems.} We visualize the X-ray transform of the function $f_0$ displayed in \cref{fig:f0} for various values of $\kappa$ in \cref{fig:If0}. The function $f_0$ is a sum of Gaussians of width $\sigma = 0.03$, hence $\hat f_0$ is a sum of centered, modulated Gaussians of width $\frac{1}{\sigma}$. Hence we declare the (essential) bandwidth of $f_0$ to be $B = \frac{3}{\sigma} = 100$. 

\begin{figure}[htpb]
    \centering
    \begin{subfigure}[b]{\textwidth}
         \centering
	 \begin{tabular}{ccccc}
	     \includegraphics[width=0.16\textwidth]{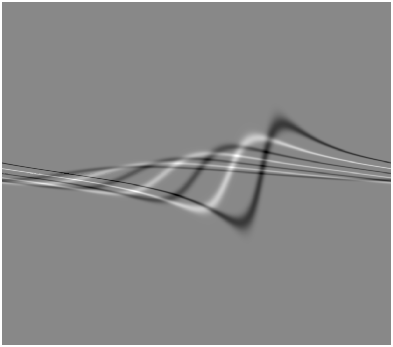} &
	     \includegraphics[width=0.16\textwidth]{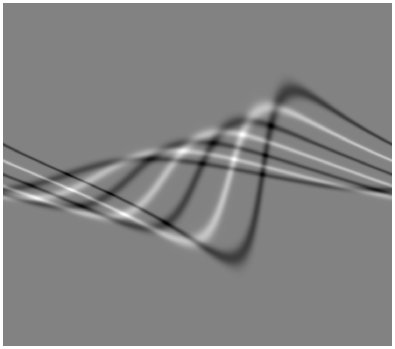} &
	     \includegraphics[width=0.16\textwidth]{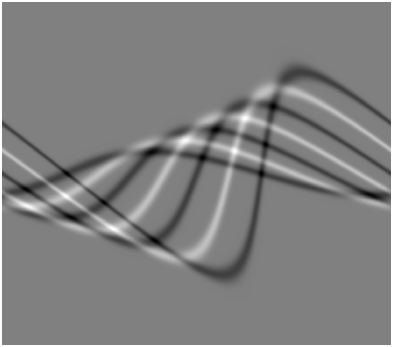} &
	     \includegraphics[width=0.16\textwidth]{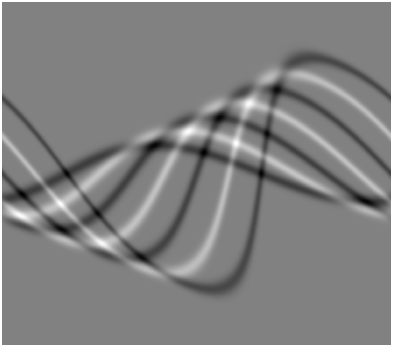} &
	     \includegraphics[width=0.16\textwidth]{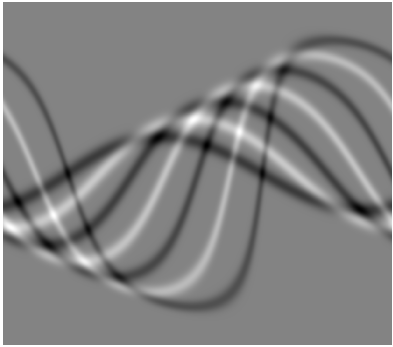} \\
	     $\kappa=-0.7$ & $\kappa = -0.3$ & $\kappa = 0$ & $\kappa=0.3$ & $\kappa = 0.7$
	 \end{tabular}
	 \caption{Cartesian $(s,\alpha)$ coordinates}
         \label{fig:If0_FB}
    \end{subfigure}
    \begin{subfigure}[b]{\textwidth}
	\centering
	\begin{tabular}{ccccc}
	    \includegraphics[width=0.16\textwidth]{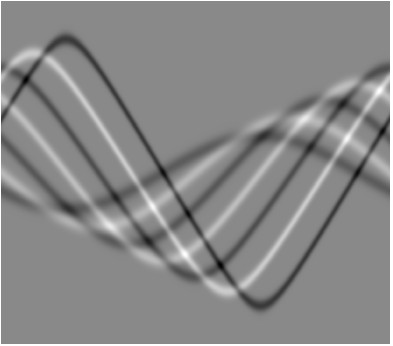} &
	    \includegraphics[width=0.16\textwidth]{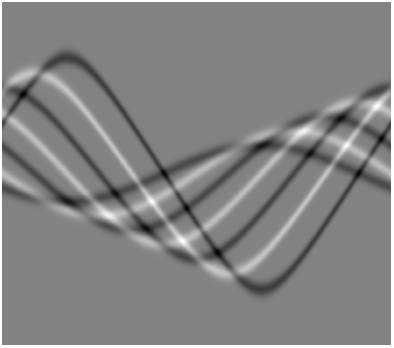} &
	    \includegraphics[width=0.16\textwidth]{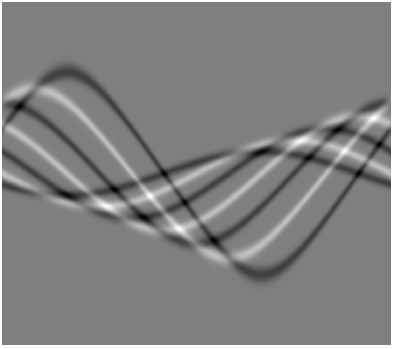} &
	    \includegraphics[width=0.16\textwidth]{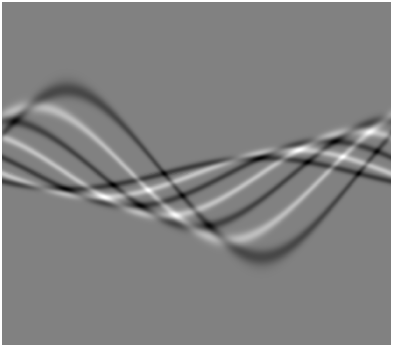} & 
	    \includegraphics[width=0.16\textwidth]{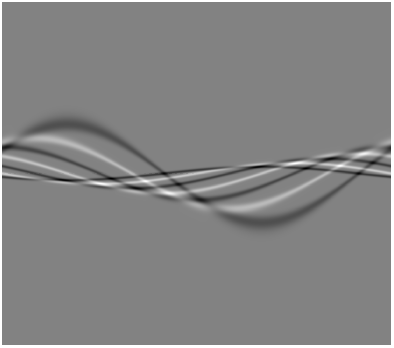} \\
	     $\kappa=-0.7$ & $\kappa = -0.3$ & $\kappa = 0$ & $\kappa=0.3$ & $\kappa = 0.7$
	 \end{tabular}
	\caption{Cartesian $(w,\t)$ coordinates}
	\label{fig:If0_WP}
    \end{subfigure}    
    \caption{Example 2. The X-ray transforms of the function displayed in \cref{fig:f0}.}
    \label{fig:If0}
\end{figure}

\subsection{Canonical relation and semiclassical wavefront sets}

In the experiments below, we will often use sums of coherent states of the form 
\begin{align}
    f_h(x; x_0, \xi_0) = \sin (x\cdot\xi_0 /h) e^{-|x-x_0|^2/2h}, \qquad \WFH (f) = \{(x_0, \pm \xi_0)\}.
    \label{eq:coherentstate}
\end{align}
For $0<h\ll 1$ fixed, the amplitude of the Fourier transform of $f_h$ is made of two gaussians centered at $\pm \xi_0/h$, each of width $\frac{1}{\sqrt{h}}$. We use those as proxies for functions which are localized in both space and frequency, and to illustrate integrands with simple frequency content. 

From Examples 3 and 4 below, we see that for $h$ small enough, a wave packet $f$ of the form \cref{eq:coherentstate} gives rise to two (distorted) wave packets on $I_0 f$, and using windowed Fourier transforms to extract the frequency content of each packet in data space shows good agreement with the predictions of \cref{eq:WFhmapI0}: the amplitude of the windowed Fourier transform of $I_0 f$ is made of spots centered at the values predicted by computing images of the original covector $\omega = (x_0, \xi_0/h)$ under the canonical graphs $C_\pm$ of $I_0$. 

In both Examples 3 and 4 below, resolution on $I_0 f$ is high enough that there is no aliasing in data space. 

\smallskip
\noindent{\bf Example 3:} In this example, $R=1$, $\kappa = 0.3$. We pick a function as the sum of three coherent states of the form \cref{eq:coherentstate}
\begin{align}
    f_1 = \sum_{k=1}^3 f_h (x; \gamma_{s_k,\alpha_k}(u_k \tau(\alpha_k)), e^{i\theta_k}), \qquad h = .01, 
    \label{eq:f1}
\end{align}
where $\theta_k$ is $\pi/2$-rotation of the direction of $\dot\gamma_{s_k,\alpha_k}(u_k \tau(\alpha_k))$, and where 
\begin{align*}
    s_1 = 5L/6,\;\;\; s_2=s_3 = L/2,\;\;\; \alpha_1 = 0,\;\;\; \alpha_2 = \alpha_3 = \pi/4,\;\;\; u_1 = u_2 = .3,\;\;\; u_3 = .6. 
\end{align*}
\Cref{fig:Ex1} displays the function $f_1$ from \cref{eq:f1} (top) and the modulus of its Fourier Transform (bottom). The function $f_1$ satisfies a band limit condition of the form \cref{eq:band limit2} with $\Sigma = B_e^* M$ \rev{(where the subscript '$e$' refers to Euclidean balls)}.

We explain the content of \cref{fig:Ex1_data}: In the left of \cref{fig:Ex1_dataFB}, we visualize $I_0 f$ in $(s,\alpha)$ coordinates, with $s\in [0,L]$ on the horizontal axis and $\alpha\in [-\pi/2,\pi/2]$ on the vertical axis. Each wave packet in $f$ gives rise to a distorted wave packed in $I_0 f$, at the points $(s_k,\alpha_k)_{k=1}^3$ (with blue bowties superimposed) and their images w.r.t. the antipodal scattering relation (with green bowties superimposed). At each $(s_k,\alpha_k)$, we draw a rescaled version of the boundary of the bowtie $[C_+(B_e^* M)]_{(s_k,\alpha_k)}$, which\rev{, in spite of living in the cotangent fiber $T^*_{(s_k,\alpha_k)}(\partial_+ SM)$, helps predict} the possible directions of oscillation at that point in data space, corresponding to oscillations of $f$ at the same, fixed frequency. The small dots on the bowtie correspond to the predicted direction of oscillations. Two of the wave packets look like checkerboard because there is interaction between the two wave packets on $f$ whose oscillations are conormal to the same geodesic. In the right of \cref{fig:Ex1_dataFB}, for each wave packet, we visualize the amplitude of the Fourier transform of $I_0 f$ cutoff around each $(s_k,\alpha_k)$ with a gaussian of width $\sigma = 0.2$ (the red circles are circles of radius $\sigma$), and superimpose the set $[C_+(B\cdot S_e^* M)]_{(s_k,\alpha_k)}$, together with red dots computed explicitly using the canonical graphs \cref{eq:Cpm}. \Cref{fig:Ex1_dataWP} repeats \cref{fig:Ex1_dataFB}, this time in $(w,\t)$ coordinates ($w\in [0,L]$ on the horizontal axis, and $\t\in [0,L/2]$ on the vertical axis).

\begin{figure}[htpb]
    \centering
    \begin{subfigure}[b]{\textwidth}
	\centering
	\includegraphics[height=0.21\textheight]{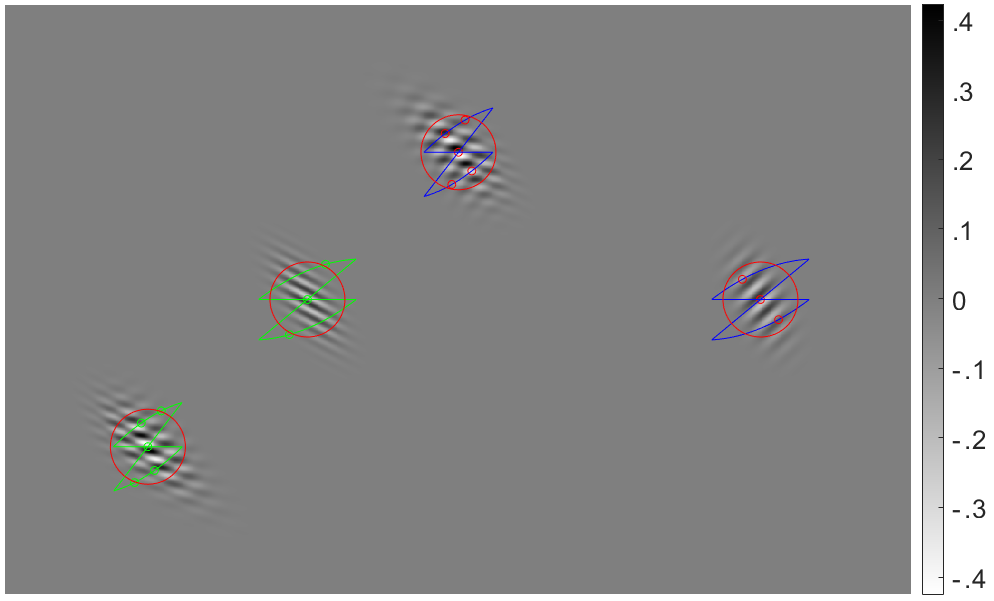} 
	\includegraphics[height=0.21\textheight]{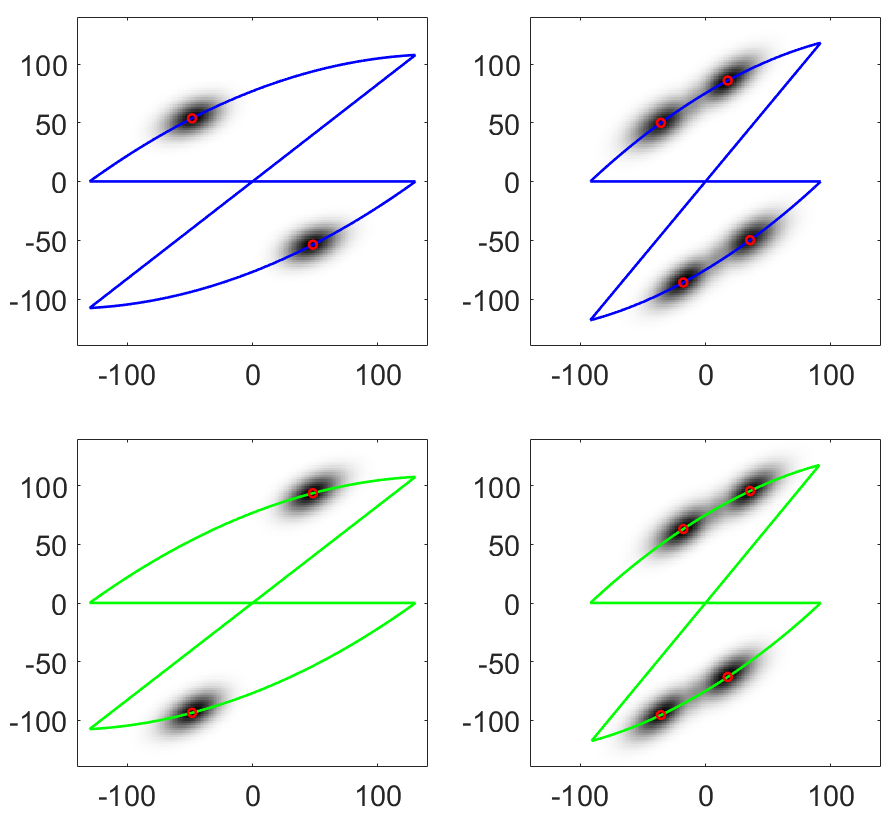}
	\caption{$(s,\alpha)$ coordinates}
	\label{fig:Ex1_dataFB}
    \end{subfigure}
    \begin{subfigure}[b]{\textwidth}
	\centering
	\includegraphics[height=0.19\textheight]{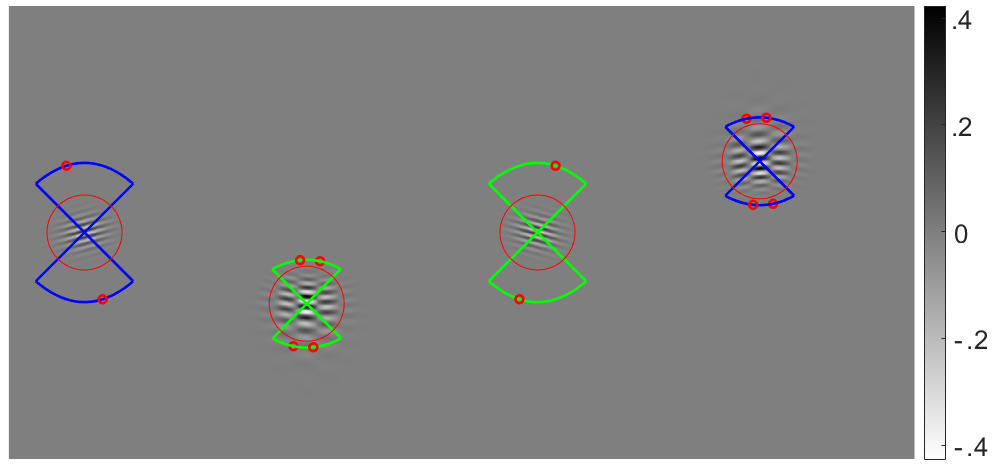} 
	\includegraphics[height=0.19\textheight]{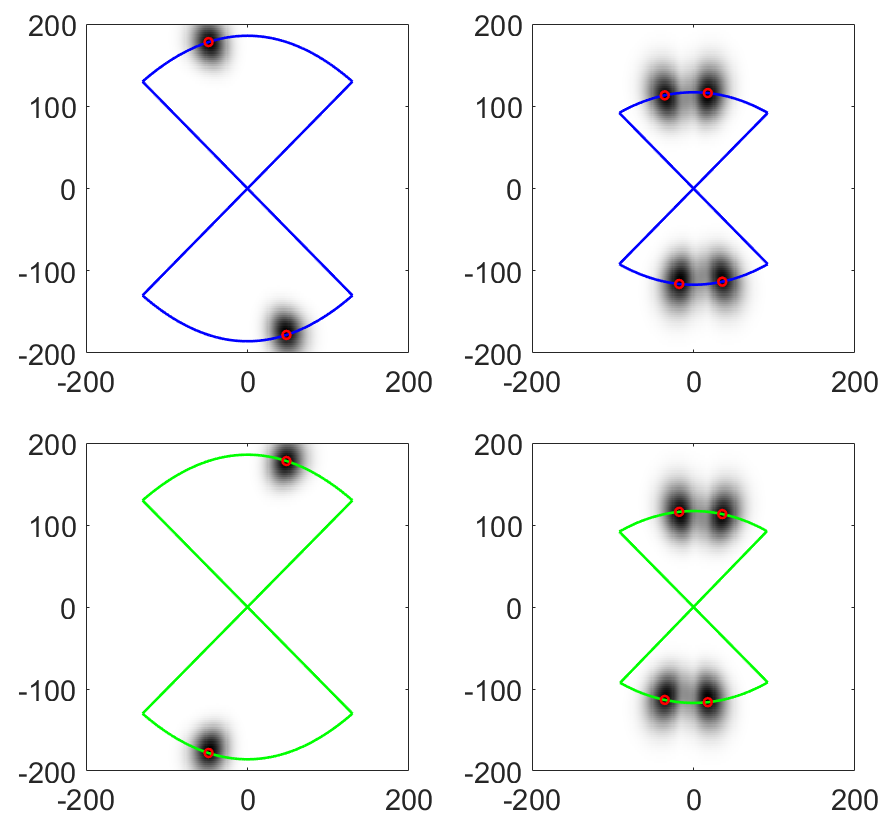}
	\caption{$(w,\t)$ coordinates}
	\label{fig:Ex1_dataWP}
    \end{subfigure}
    \caption{Example 3 $(\kappa=0.3)$ - Left: data $I_0 f$ where $f$ is visualized \cref{fig:Ex1}. Right: bowties (see explanation in the text)}
    \label{fig:Ex1_data}
\end{figure}

\smallskip
\noindent{\bf Example 4:} In this example, we reproduce the same as Example 3, except in curvature $\kappa = -0.3$. All other parameters are the same, the function is displayed \cref{fig:Ex2} and the outcomes are displayed \rev{in} \cref{fig:Ex2_data}.

\begin{figure}[htpb]
    \centering
    \begin{subfigure}[b]{\textwidth}
	\centering
	\includegraphics[height=0.15\textheight]{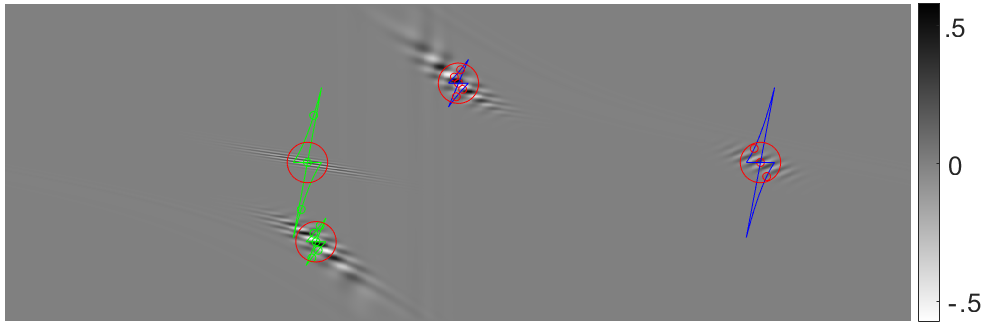} 
	\includegraphics[height=0.15\textheight]{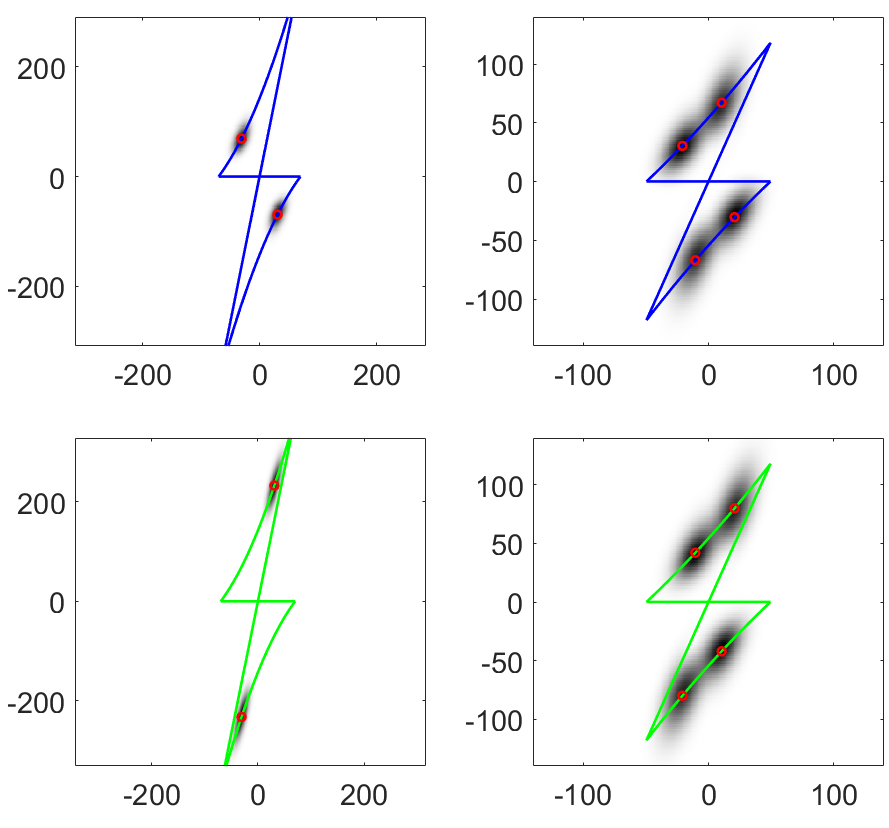}
	\caption{$(s,\alpha)$ coordinates}
	\label{fig:Ex2_dataFB}
    \end{subfigure}
    \begin{subfigure}[b]{\textwidth}
	\centering
	\includegraphics[height=0.19\textheight]{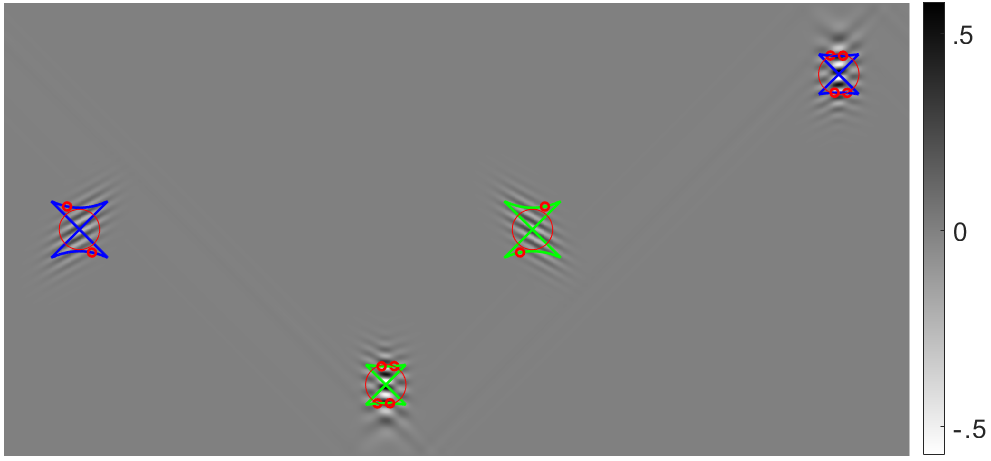} 
	\includegraphics[height=0.19\textheight]{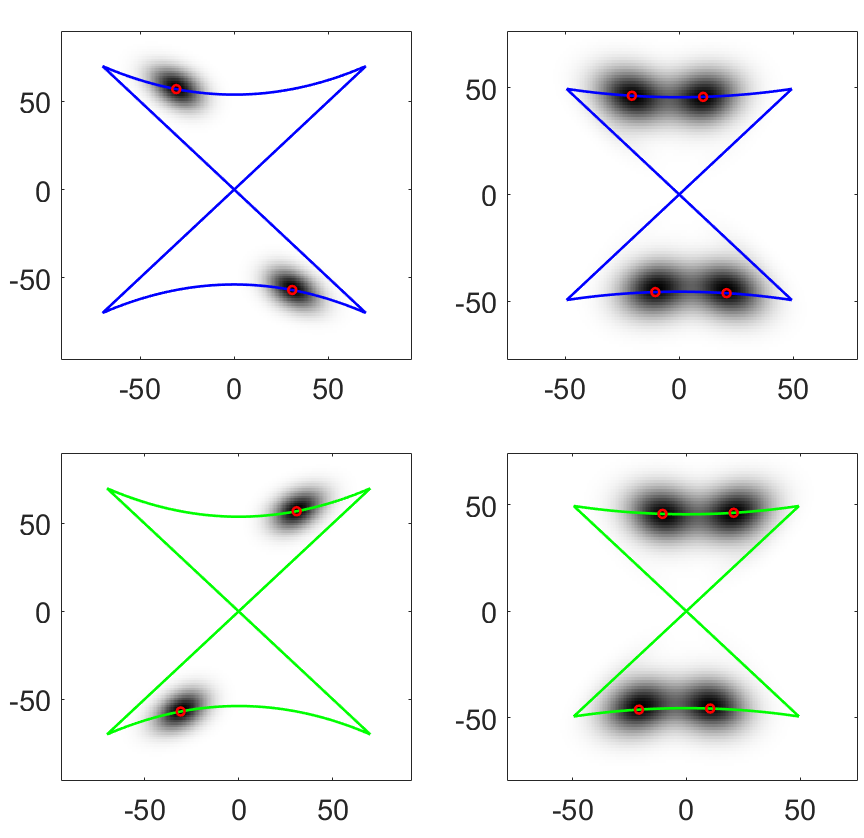}
	\caption{$(w,\t)$ coordinates}
	\label{fig:Ex2_dataWP}
    \end{subfigure}
    \caption{Example 4: Same as Example 3 with $\kappa=-0.3$. Left: data $I_0 f$ where $f$ is visualized \cref{fig:Ex2}. Right: bowties; frequency axes are rescaled to be discretization-independent (they are the sizes of the dual variable components for the continuous Fourier transform)}
    \label{fig:Ex2_data}
\end{figure}

\subsection{Bounding boxes and sufficient criteria for non-aliased reconstructions} \label{sec:boxes}

In this section we attempt to answer the following: given an essentially band limited function $f$,
\begin{description}
    \item[Q1] How finely do we have to sample $I_0 f$ to avoid aliasing on $I_0 f$?
    \item[Q2] How to reconstruct $f$ from unaliased data? (how to avoid aliasing in the inversion process) 
\end{description}

Assuming a band limit condition on $f$ of the form \cref{eq:band limit2} with $B$ and $\Sigma$ fixed, the answer to Q1 depends on how we decide to sample $I_0 f$. Our strategy will be to fix a coordinate system to represent $I_0 f$ (either $(s,\alpha)$ or $(w,\t)$), and perform equispaced Cartesian sampling. \rev{Note that in fan-beam coordinates, this corresponds to uniform sampling with respect to the 'natural' geometric measure $d\Sigma^2$, i.e. the area form coming from the induced Sasaki metric on $\partial_+ SM$.} The question will then be to estimate the number of points in each coordinate in terms of the band limit of $f$. This is done as follows:

$\bullet$ Choosing fan-beam coordinates $(s,\alpha)$ on $\partial_+ SM$, and looking to construct a Cartesian Nyquist box in those coordinates, one must compute the quantities
\begin{align*}
	    b_s(\Sigma) = \max_{(x,v)\in \partial_+ SM} \left|[C_+(\Sigma)]_{(x,v)} (\partial_s)\right|, \quad b_\alpha (\Sigma) = \max_{(x,v)\in \partial_+ SM} \left|[C_+(\Sigma)]_{(x,v)} (\partial_\alpha)\right|,
\end{align*}
defined generally in \cref{eq:bnumbers}. These numbers measure the largest possible frequency extent in data space of a \rev{$\Sigma$}-band limited function, along the dual axes to the chosen coordinate system. In parallel coordinates $(w,\t)$, one computes instead the similarly defined quantities $b_w(\Sigma)$ and $b_\t (\Sigma)$. 

$\bullet$ Finally, the Nyquist criterion in fan-beam coordinates will require choosing stepsizes $h_s\le \frac{\pi}{B b_s}$ and $h_\alpha\le \frac{\pi}{B b_\alpha}$, leading to at least $N_s = \frac{L}{h_s} = B \frac{L b_s}{\pi}$ points along the $s$ axis and $N_\alpha = \frac{\pi}{h_\alpha} = B b_\alpha$ along the $\alpha$ axis.

Similarly, in parallel coordinates, stepsize requirements look like $h_w \le \frac{\pi}{B b_w}$ and $h_\t \le \frac{\pi}{B b_\t}$, leading to at least $N_w = \frac{L}{h_w} = B \frac{L b_w}{\pi}$ points along the $w$ axis, and at least $N_\t = B \frac{L b_\t}{2\pi}$ points along the $\t$ axis.

The quantities $b_\bullet (\Sigma)$ can be computed by running geodesics and including the values of $\eta_{x,v,t}$ for which $\varphi_{x,v}(t)$ belongs to the support of $f$. More precisely, if there exists a function $m\colon M \to \Rm_+$ 
such that $\Sigma_x = \{ \omega\in T^*_x M: \frac{m(x)}{c_\kappa(x)} |\omega|_e^2 \le 1 \}$ for all $x\in M$, then 
\begin{align}
    [C_+(\Sigma)]_{(x,v)} = \{ \lambda m(\gamma_{x,v}(t)) \eta_{x,v,t}: \lambda\in [-1,1],\ t\in (0,\tau(x,v)) \}, \qquad (x,v)\in \partial_+ SM.
    \label{eq:bowties}
\end{align}
The latter expression makes \rev{it} especially simple to numerically compute these sets for various functions $m$, and we now give three examples. 

In all cases considered below, the geometry and the $\Sigma$ constraint are rotation invariant, and thus the shape of $[C_+(\Sigma)]_{(x,v)}$ does not depend on $s$ in fan-beam coordinates, or does not depend on $w$ in parallel coordinates. In \cref{fig:bowties,fig:bowtiesE2,fig:bowtiesE3} below, for each coordinate system and different value of $\kappa$ (keeping $R=1$), we plot several representative sets \cref{eq:bowties} for a few values of $\alpha$, out of which we can extract the numbers $b_\bullet (\Sigma)$ for $\bullet \in \{s,\alpha,w,\t\}$. We tabulate the resulting number-of-sample factors $N_{(s,\alpha)} (\Sigma) = \frac{L}{\pi} b_s(\Sigma) b_\alpha(\Sigma)$ and $N_{(w,\t)} (\Sigma) = \frac{L^2}{2\pi^2} b_w(\Sigma) b_\t(\Sigma)$ in \cref{tab:boundingboxes}. We consider the following cases: 
\begin{align}
    \begin{split}
	\Sigma_1 = B_e^* M, &\qquad m = c_\kappa \\
	\Sigma_2 = B^* M, &\qquad m \equiv 1 \\
	[\Sigma_3]_x = \left\{
	\begin{array}{cc}
	    [B_e^* M]_x, & \text{if } |x|<0.75, \\
	    \{0\}, & \text{otherwise},
	\end{array}
	\right. & \qquad m(x) = 1_{\{|x|<0.75\}} c_\kappa(x).
    \end{split}
    \label{eq:Sigmas}
\end{align}

In the case $\Sigma = \Sigma_3$, short geodesics (for which $|\alpha|$ is close to $\pi/2$ or $|\t-L/4|$ is close to $L/4$) do not intersect the support of $f$ and as such should not be sampled. In particular, $\alpha \in (-\alpha_{m},\alpha_m)$, where $\alpha_m\in [0,\pi/2]$ is the largest value for which the geodesic $\gamma_{0,\alpha}$ intersects the set $\{|x|\le 0.75\}$ (that such an $\alpha_m$ exists comes from the fact that centered disks are convex for the geodesic flow in all geometries considered here). Via \cref{eq:CCD_FB2WT}, $\alpha_m$ gives rise to $\t_m\in [L/4,L/2]$, and the range of $\t$'s where data should be sampled becomes $[L/2-\t_m,\t_m]$. As a result, the number-of-sample factors become, in this case
\begin{align*}
    N_{(s,\alpha)}(\Sigma_3) = \frac{L b_s}{\pi} \cdot \frac{2\alpha_m b_\alpha}{\pi}, \qquad N_{(w,\t)}(\Sigma_3) = \frac{L b_w}{\pi} \cdot \frac{(2\t_m-L/2)b_\t}{\pi}.
\end{align*}
Note that $\alpha_m$ and $\t_m$ depend on $\kappa$, $R$ and the size of the support. 

\begin{figure}[htpb]
    \centering
    \begin{tabular}{cccccc}
	\includegraphics[height=0.08\textheight]{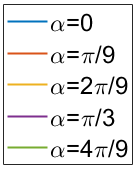} &
	\includegraphics[trim = 30 0 30 20, clip , height=0.07\textheight]{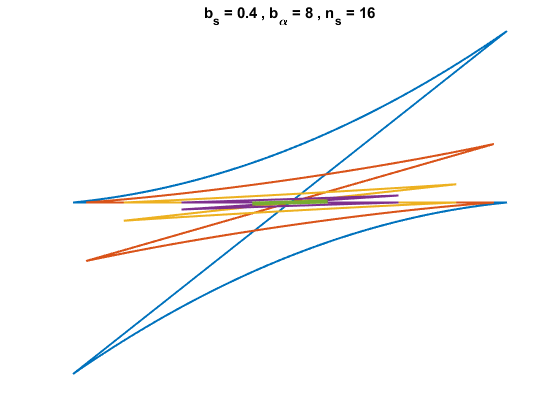} & 
	\includegraphics[trim = 30 0 30 20, clip , height=0.07\textheight]{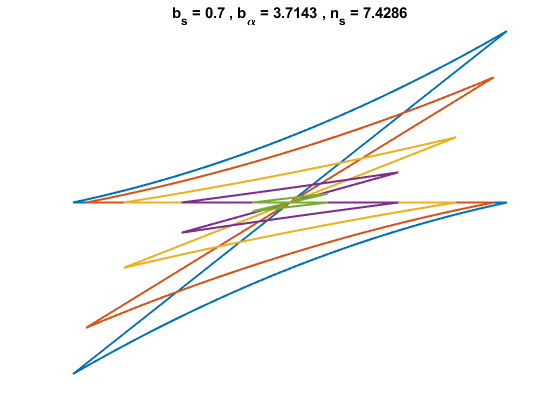}  & 
	\includegraphics[trim = 30 0 30 20, clip , height=0.07\textheight]{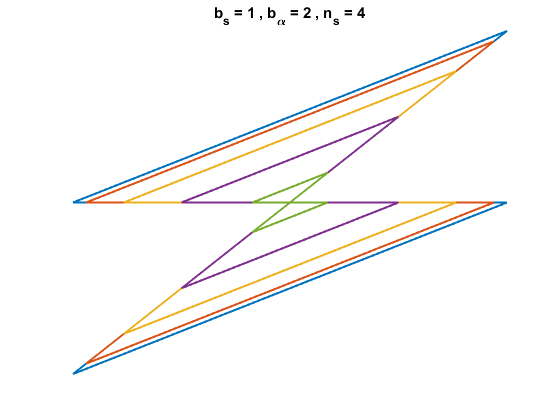} &
	\includegraphics[trim = 30 0 30 20, clip , height=0.07\textheight]{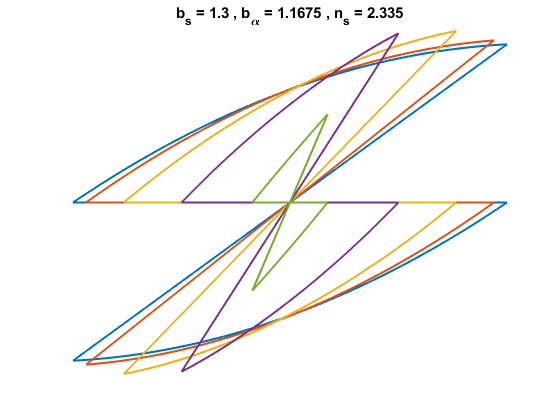} &
	\includegraphics[trim = 30 0 30 20, clip , height=0.07\textheight]{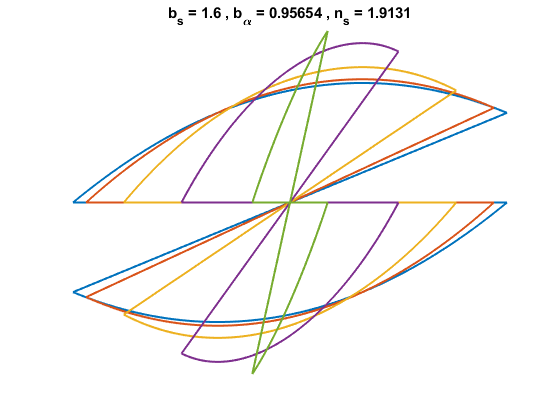} \\
	& $(0.4,8)$ & $(0.7,3.71)$ & $(1,2)$ & $(1.3,1.17)$ & $(1.6,0.96)$ \\[1mm]
	& \includegraphics[trim = 30 0 30 21, clip , height=0.07\textheight]{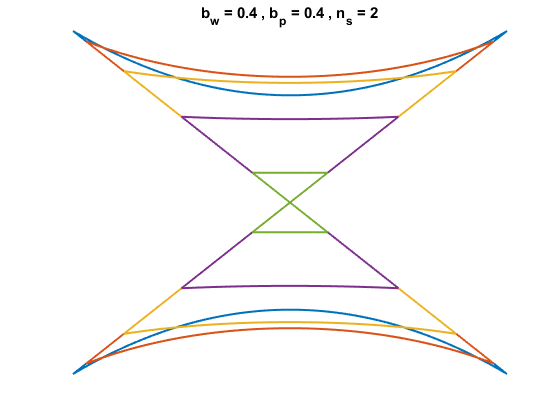} &
	\includegraphics[trim = 30 0 30 21, clip , height=0.07\textheight]{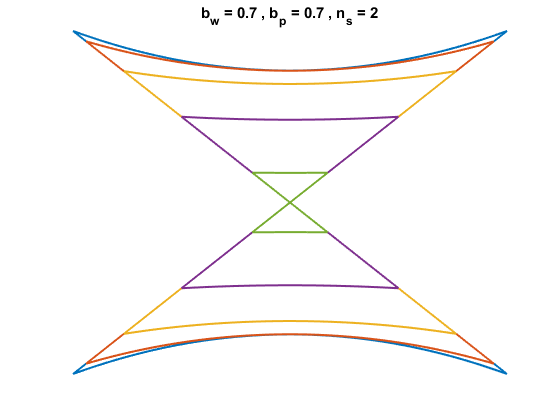}  &  
	\includegraphics[trim = 30 0 30 21, clip , height=0.07\textheight]{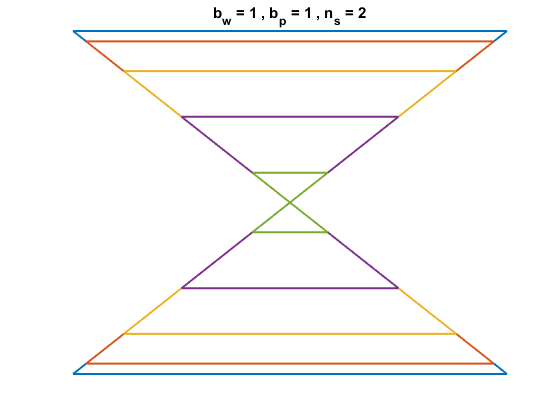} &
	\includegraphics[trim = 30 0 30 21, clip , height=0.07\textheight]{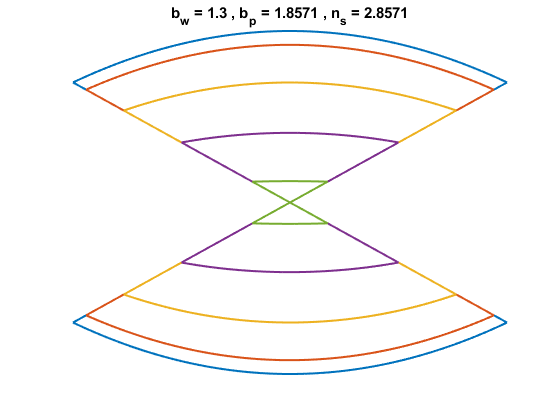} &
	\includegraphics[trim = 30 0 30 21, clip , height=0.07\textheight]{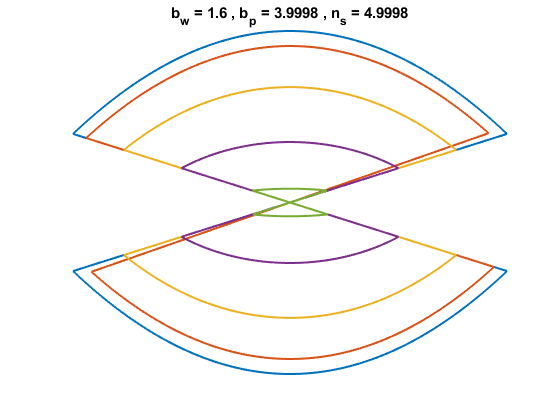} \\	
	& $(0.4,0.4)$ & $(0.7,0.7)$ & $(1,1)$ & $(1.3,1.86)$ & $(1.6,4)$
    \end{tabular}
    \caption{Case $\Sigma = \Sigma_1$. In each figure, we superimpose $[C_+ (\Sigma)]_{(\alpha)}$ for $\alpha\in \{j\pi/9, 0\le j\le 4\}$. Left to right: $\kappa = \{-0.6,-0.3,0,0.3,0.6\}$. Top row: fan-beam coordinates, bowties are rescaled in a box $[-b_s,b_s]\times [-b_\alpha,b_\alpha]$ with $(b_s,b_\alpha)$ given below each plot. Bottom row: parallel coordinates, bowties are rescaled in a box $[-b_w,b_w]\times [-b_\t,b_\t]$ with $(b_w,b_\t)$ given below each plot.}
    \label{fig:bowties}
\end{figure}

\begin{figure}[htpb]
    \centering
    \begin{tabular}{cccccc}
	\includegraphics[height=0.08\textheight]{./figs/alphas.png} &	
	\includegraphics[trim = 30 0 30 20, clip , height=0.07\textheight]{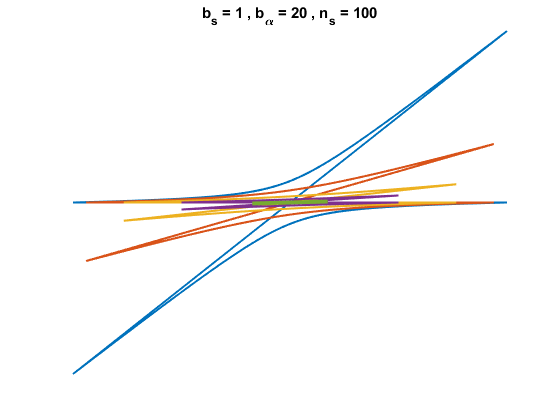} & 
	\includegraphics[trim = 30 0 30 20, clip , height=0.07\textheight]{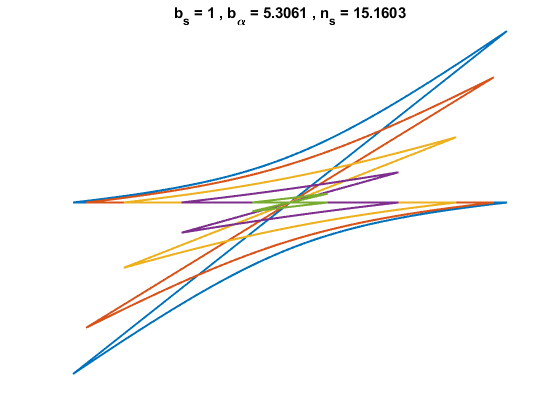}  & 
	\includegraphics[trim = 30 0 30 20, clip , height=0.07\textheight]{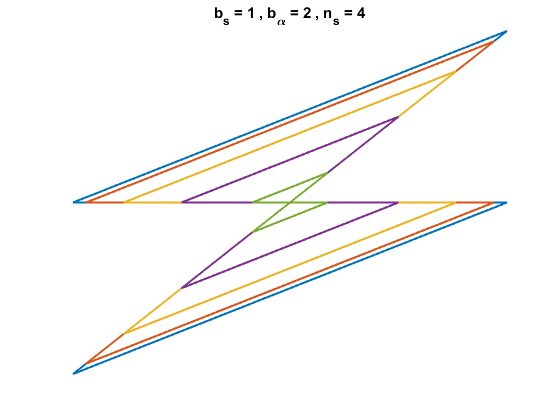} &
	\includegraphics[trim = 30 0 30 20, clip , height=0.07\textheight]{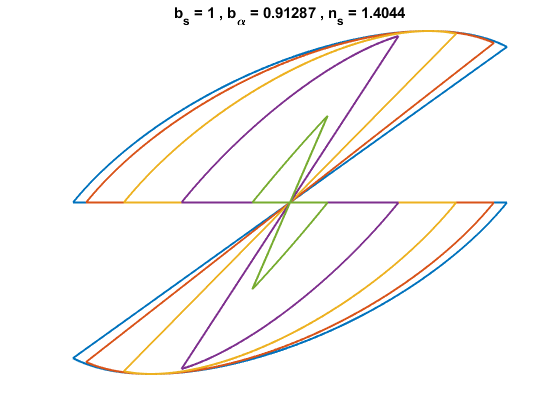} &
	\includegraphics[trim = 30 0 30 20, clip , height=0.07\textheight]{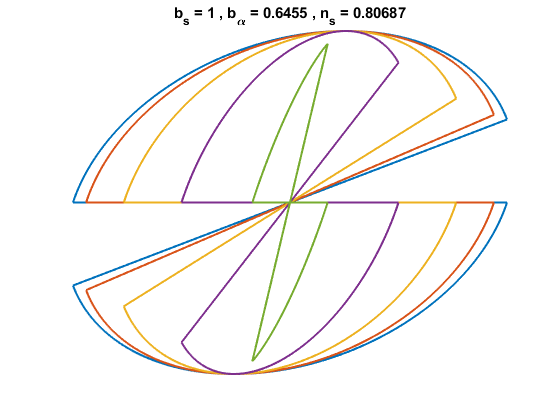} \\
	& $(1,20)$ & $(1,5.31)$ & $(1,2)$ & $(1,0.91)$ & $(1,0.65)$ \\[1mm]
	& \includegraphics[trim = 30 0 30 21, clip , height=0.07\textheight]{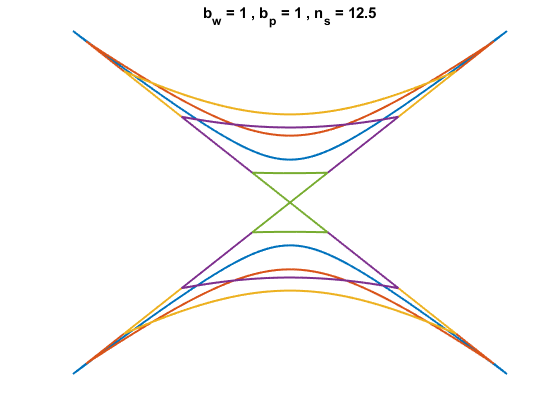} &
	\includegraphics[trim = 30 0 30 21, clip , height=0.07\textheight]{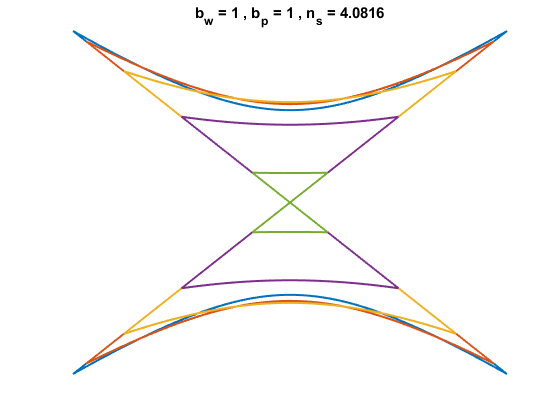}  &  
	\includegraphics[trim = 30 0 30 21, clip , height=0.07\textheight]{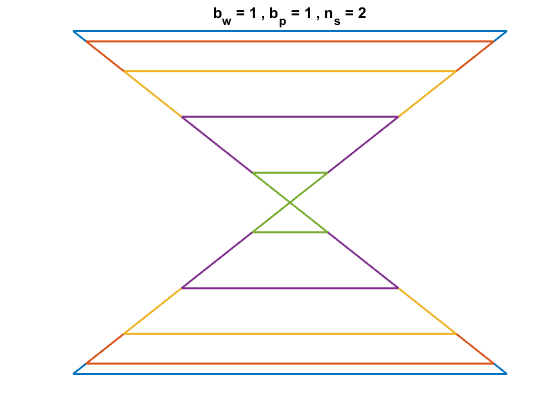} &
	\includegraphics[trim = 30 0 30 21, clip , height=0.07\textheight]{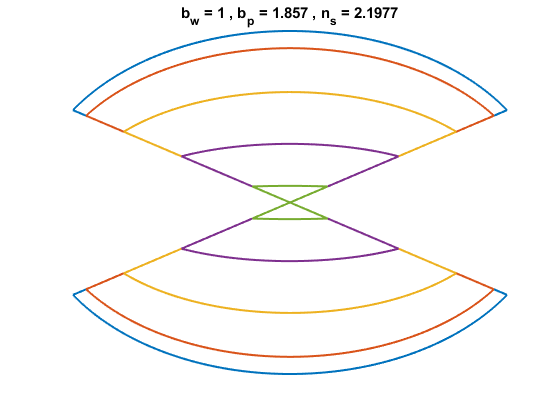} &
	\includegraphics[trim = 30 0 30 21, clip , height=0.07\textheight]{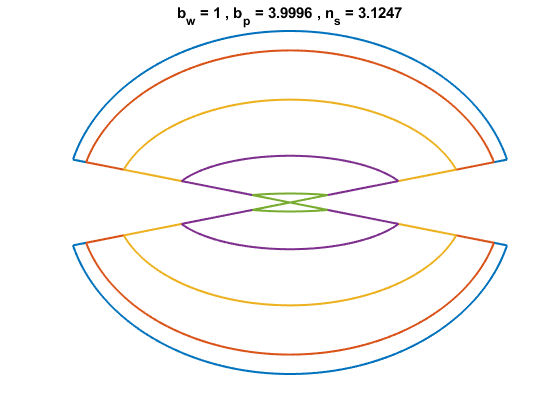} \\	
	& $(1,1)$ & $(1,1)$ & $(1,1)$ & $(1,1.86)$ & $(1,4.0)$ 	
    \end{tabular}
    \caption{Same as \cref{fig:bowties} with $\Sigma = \Sigma_2$.}
    \label{fig:bowtiesE2}
\end{figure}

\begin{figure}[htpb]
    \centering
    \begin{tabular}{cccccc}
	\includegraphics[height=0.08\textheight]{./figs/alphas.png} &	
	\includegraphics[trim = 30 0 30 20, clip , height=0.07\textheight]{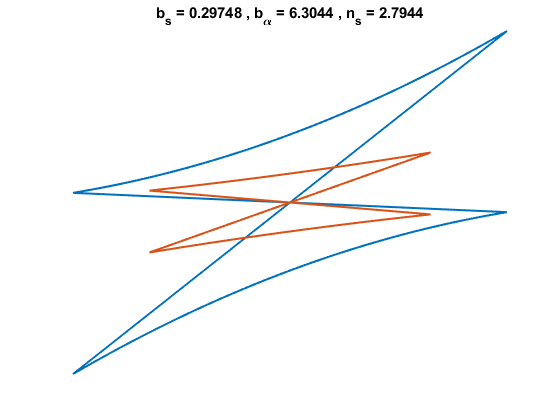} & 
	\includegraphics[trim = 30 0 30 20, clip , height=0.07\textheight]{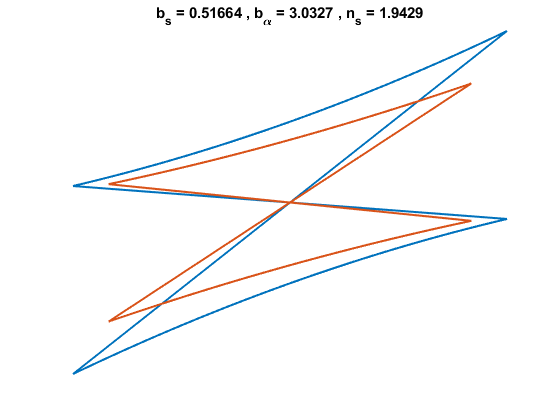}  & 
	\includegraphics[trim = 30 0 30 20, clip , height=0.07\textheight]{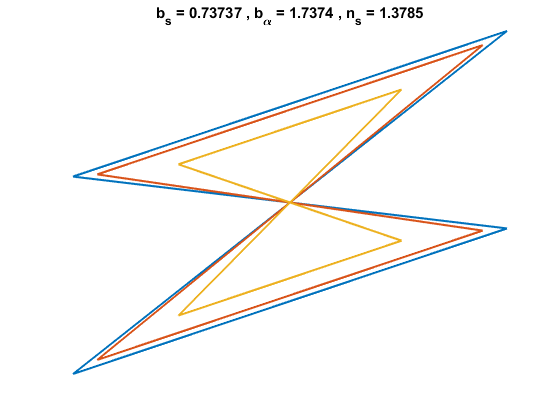} &
	\includegraphics[trim = 30 0 30 20, clip , height=0.07\textheight]{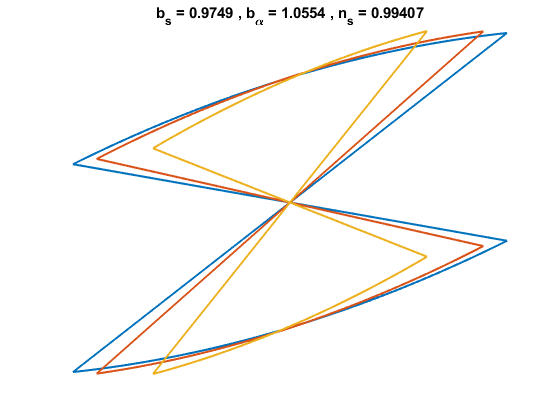} &
	\includegraphics[trim = 30 0 30 20, clip , height=0.07\textheight]{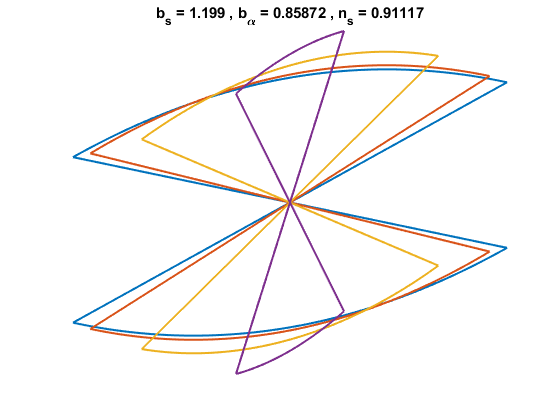} \\
	& $(0.3,6.3)$ & $(0.52,3.03)$ & $(0.74,1.74)$ & $(0.97,1.06)$ & $(1.2,0.86)$ \\[1mm]
	& \includegraphics[trim = 30 0 30 21, clip , height=0.07\textheight]{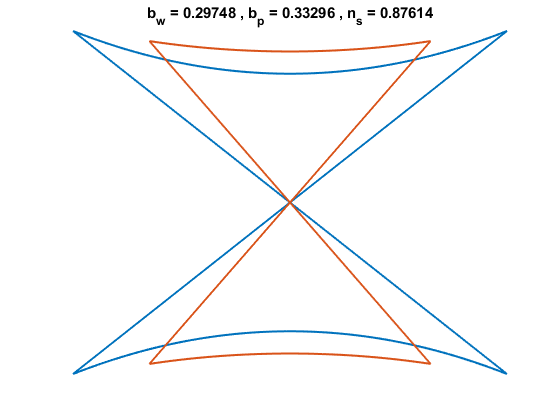} &
	\includegraphics[trim = 30 0 30 21, clip , height=0.07\textheight]{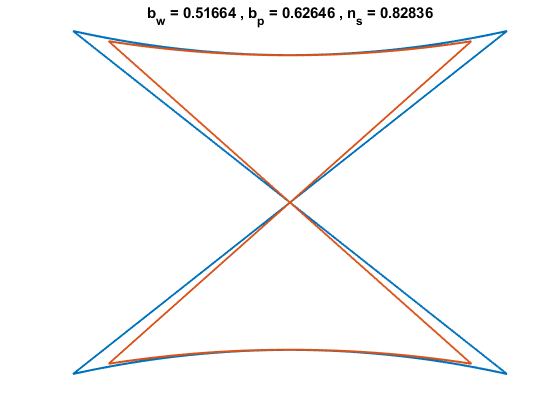}  &  
	\includegraphics[trim = 30 0 30 21, clip , height=0.07\textheight]{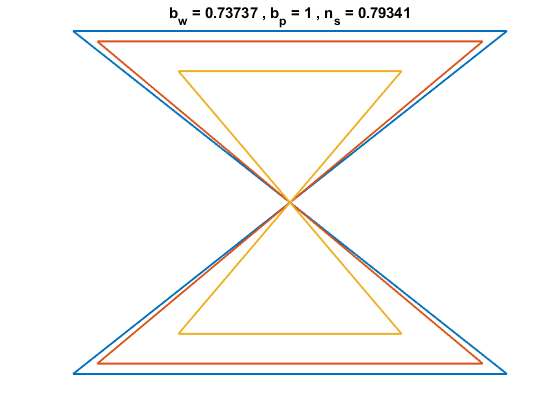} &
	\includegraphics[trim = 30 0 30 21, clip , height=0.07\textheight]{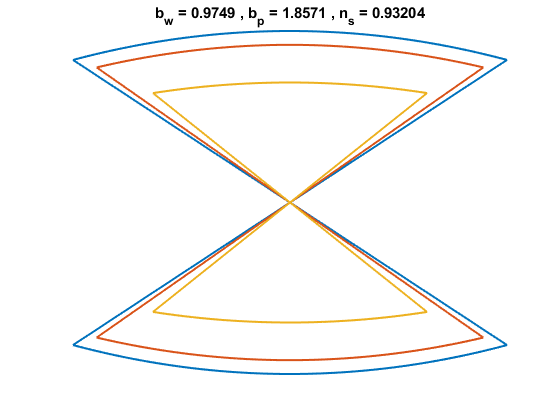} &
	\includegraphics[trim = 30 0 30 21, clip , height=0.07\textheight]{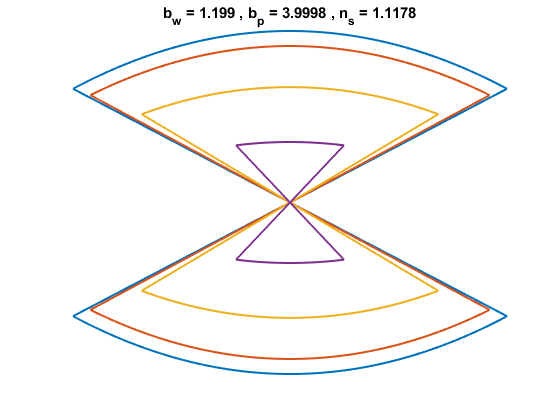} \\	
	& $(0.3,0.33)$ & $(0.52,0.63)$ & $(0.74,1)$ & $(0.97,1.86)$ & $(1.2,4)$ 	
    \end{tabular}
    \caption{Same as \cref{fig:bowties} with $\Sigma = \Sigma_3$ (support constraint).}
    \label{fig:bowtiesE3}
\end{figure}

\begin{table}
    \centering
    \begin{tabular}{|c||c|c|c|c|c|}
	\hline
	$\kappa$ & $-0.6$ & $-0.3$ & $0$ & $0.3$ & $0.6$ \\
	\hline
	$N_{(s,\alpha)}(\Sigma_1)$ & $16$ & $7.43$ & $4$ & $2.34$ & $1.91$ \\
	$N_{(w,\t)} (\Sigma_1) $ & $2$ & $2$ & $2$ & $2.86$ & $5$ \\	
	\hline\hline
	$N_{(s,\alpha)}(\Sigma_2)$ & $100$ & $15.2$ & $4$ & $1.4$ & $0.81$ \\
	$N_{(w,\t)}(\Sigma_2)$ & $12.5$ & $4.1$ & $2$ & $2.20$ & $3.12$ \\	
	\hline\hline
	$N_{(s,\alpha)}(\Sigma_3)$ & 2.80 & 1.94 & 1.38 & 1 & 0.91 \\
	$N_{(w,\t)}(\Sigma_3)$ & 0.88 & 0.83 & 0.79 & 0.93 & 1.12 \\	
	\hline
    \end{tabular}
    \caption{Number-of-sample factors for fan-beam coordinates $N_{(s,\alpha)}$ and parallel coordinates $N_{(w,\t)}$, for the sets $\Sigma_{1,2,3}$ defined by eq.~\cref{eq:Sigmas}.}
    \label{tab:boundingboxes}
\end{table}

In light of \cref{fig:bowties} through \cref{fig:bowtiesE3} and \cref{tab:boundingboxes}, we can draw the following conclusions: 
\begin{itemize}
    \item In both coordinate systems, long geodesics mostly produce larger bowties, and as such, drive the global sampling rate. \rev{In fan-beam coordinates, this is because longer geodesics allow for larger values of the Jacobi functions $a,b$ appearing in \cref{eq:etasal}. In parallel coordinates, ratios of the form $b(t-\tau/2)/b(\tau/2)$ appearing in \cref{eq:etaCCD_WT} alleviate this effect, but the $\mu$ factor there induces a global shrinking of the bowtie as geodesics become more tangential to the boundary and hence, by convexity, shorter. } An exception to this is when $\kappa R^2\to 1$ and $\Sigma = \Sigma_1$ in fan-beam coordinates (see top-right plot of \cref{fig:bowties}). 
    \item Negative curvature geometries are better sampled in the parallel coordinate system than in fan-beam (since $N_{(s,\alpha)}>N_{(w,\t)}$). This defect was previously observed in \cite{Monard2013} and explained in \cref{sec:sensitivity}. Band limit constraints using $\Sigma_1$ rather than $\Sigma_2$ lead to smaller sampling rates. 
    \item Positive curvature geometries are better sampled in fan-beam geometry than parallel. Band limit constraints using $\Sigma_2$ rather than $\Sigma_1$ lead to smaller sampling rates. 
    \item In all cases, \cref{tab:boundingboxes} shows that incorporating support constraints in the sampling scheme can greatly reduce the sampling rate, as can be seen from the drop in number-of-sample factors from $\Sigma_1$ to $\Sigma_3$. 
    \item Parallel coordinates may be amenable to non-Cartesian sampling. In particular, in negative curvature, the fact that the bowties don't stick out at the top makes it more amenable to staggered tiling, thereby improving the sampling rate, in a similar fashion to \cite[Sec. III.3]{Natterer2001}. 
\end{itemize}

\smallskip
\noindent{\bf Constraints on the inversion method.} In the case of simple surfaces, an approximate (up to smoothing error) reconstruction formula was first given in \cite{Pestov2004}. In the case of simple surfaces of constant curvature, this reconstruction formula is exact and takes the following form (written in the notation of \cite[Eq. (10) with $W\equiv 0$]{Monard2015}): 
\begin{align}
    f = I_0^{-1} I_0 f \qquad \text{where} \qquad I_0^{-1} := \frac{1}{8\pi} I_\perp^* A_+^* H_- A_-,
    \label{eq:I0inv}
\end{align}
and where 
\begin{itemize}
    \item $A_-$ denotes extension from $\partial_+ SM$ to $\partial SM$ by oddness w.r.t. scattering relation,
    \item $H_-$ is the odd Hilbert transform on the circle fibers of $\partial SM$, 
    \item $A_+^* u(x,v) = u(x,v)+u(S(x,v))$ for $u\colon \partial SM\to \Cm$ and $(x,v)\in \partial_+ SM$,
    \item $I_\perp^* w$ means: extend $w$ from $\partial_+ SM$ to $SM$ as constant along geodesics, apply $X_\perp$ to the result, then average over the fibers of $SM$. 
\end{itemize}
All these building blocks are classical FIOs. As such, guaranteeing sharp rates for non-aliased inversions may require a more in-depth study of the canonical relation of each component. In particular, even if we can define $I_0^{-1}$ 
whose canonical relation encodes the inverse of that of $I_0$, it can be that the sampling requirements on each building block making $I_0^{-1}$ may require {\em higher} sampling rates than the ones predicted simply by looking the inverse of the canonical relation of $I_0$.  

As this issue can be overcome by upsampling the data $I_0 f$ before applying $I_0^{-1}$, and for the sake of brevity, we will assume that the inversion formula \cref{eq:I0inv} does not generate aliasing by choosing 'high enough' resolution when computing $I_0^{-1}$. 

\smallskip
\noindent{\bf Example 5. Minimum sampling rates.} As an example of how to use these rates, we display on \cref{fig:rates} reconstructions of the function $f_0$ seen \cref{fig:f0}. Since $\kappa = -0.3$, $R=1$ and $f_0$ is essentially $B\cdot \Sigma_1$-band limited with $B = 100$, we read off $b_s = 0.7$ and $b_\alpha = 3.71$ from \cref{fig:bowties}. 

Then the minimum sampling rate should be $N_s\cdot N_\alpha$, where $N_s = B \frac{L}{\pi} b_s = 200$ and $N_\alpha = B b_\alpha = 371$. For the sampling rates $N_s\times N_\alpha$, $N_s/2\times N_\alpha/2$ and $N_s/3\times N_\alpha/3$, we compute $I_0 f$, upsample to a fixed grid (to avoid unnecessary errors due to the inversion process), visualize the global Fourier transform of the data, and run a filtered-backprojection. One may see on the Fourier plot when aliasing occurs, and as a result, how the reconstruction gets affected.  

\begin{figure}[htpb]
    \centering
    \includegraphics[height=0.19\textheight]{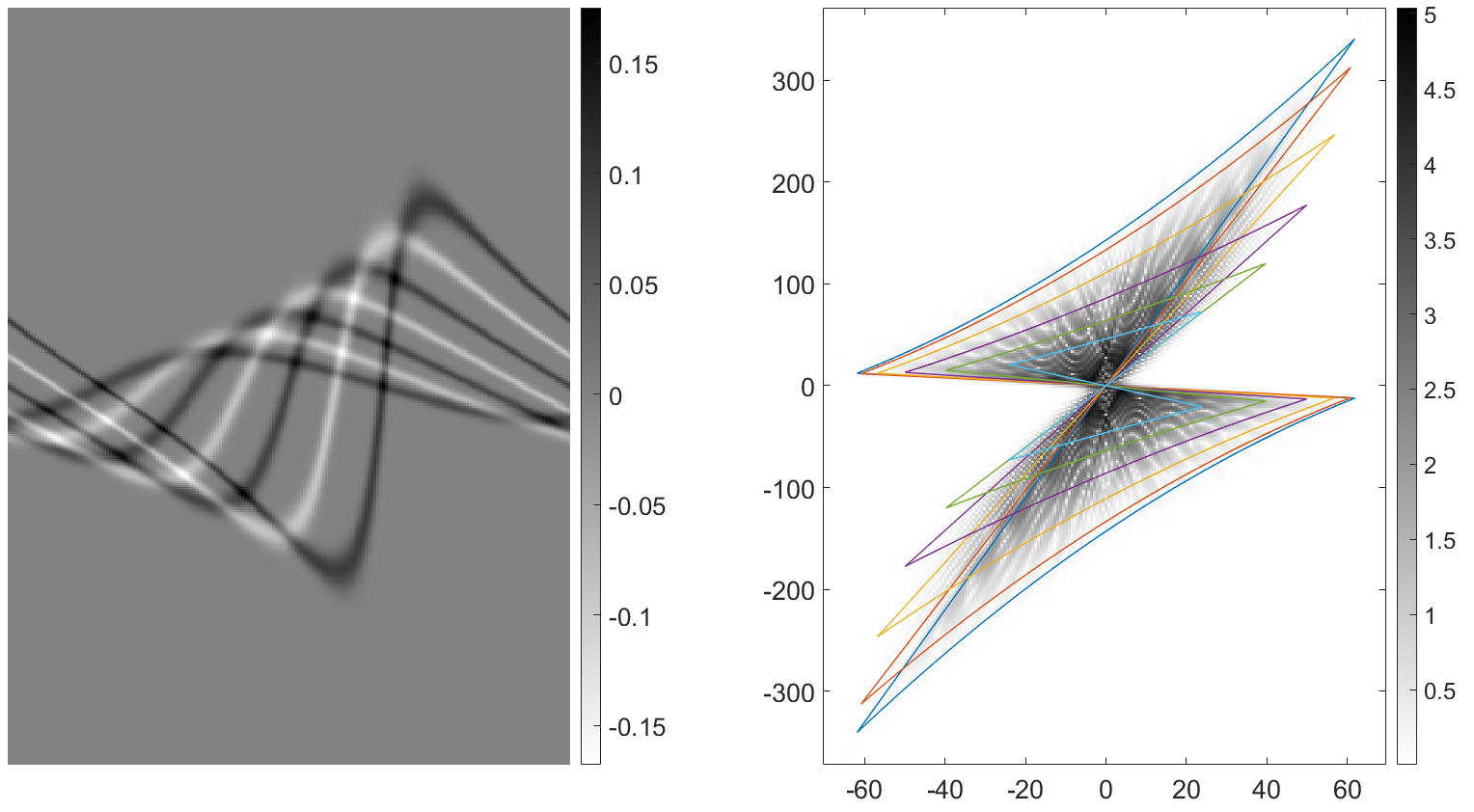} \;\;
    \includegraphics[height=0.19\textheight]{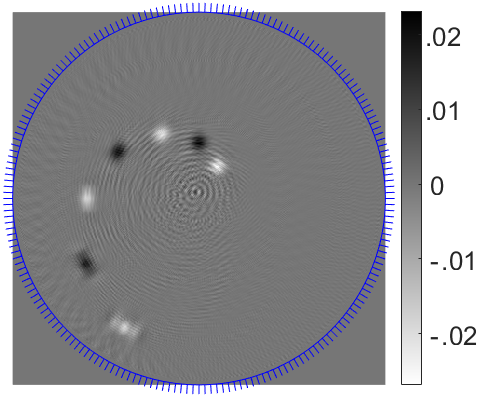} \\
    \includegraphics[height=0.19\textheight]{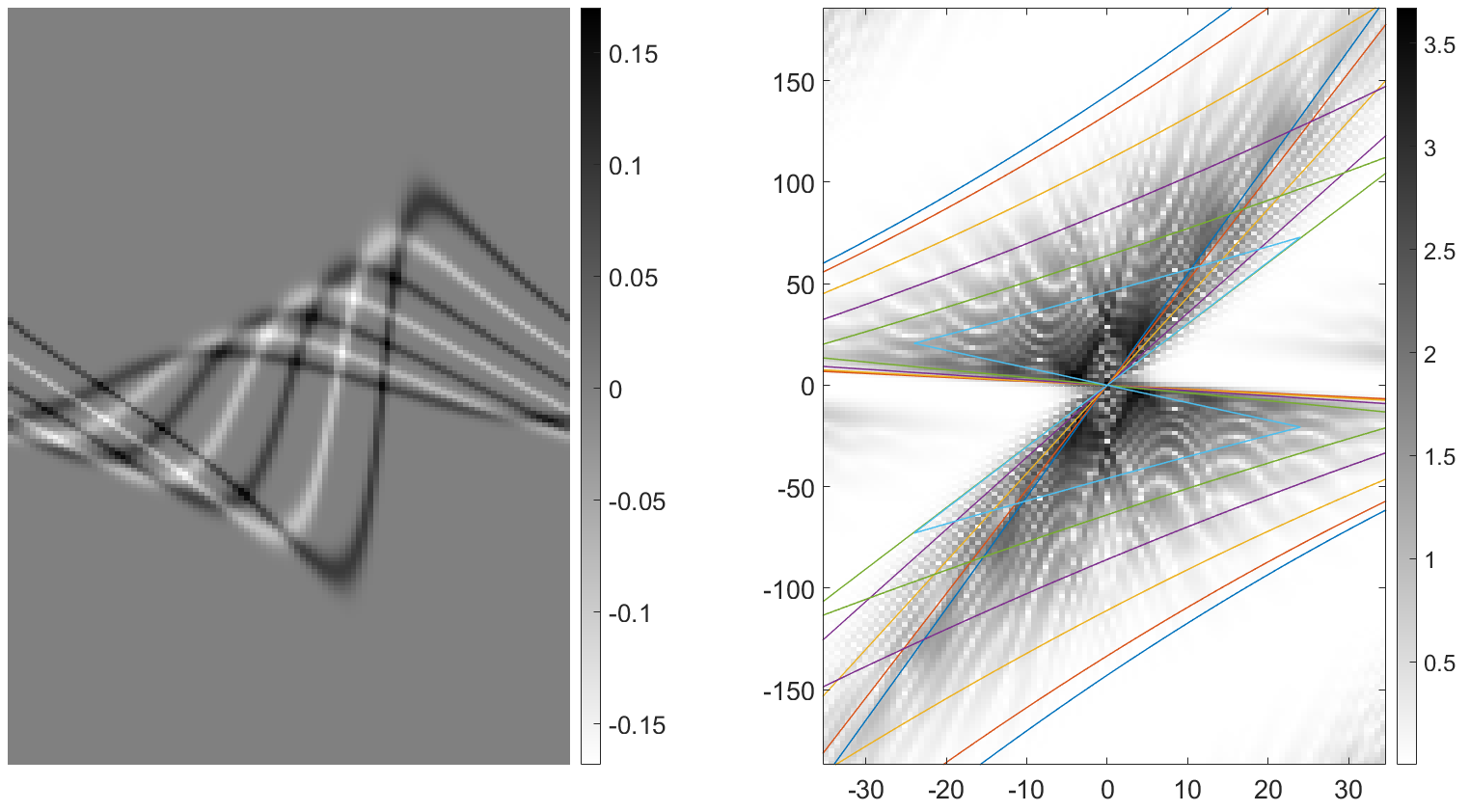} \;\;
    \includegraphics[height=0.19\textheight]{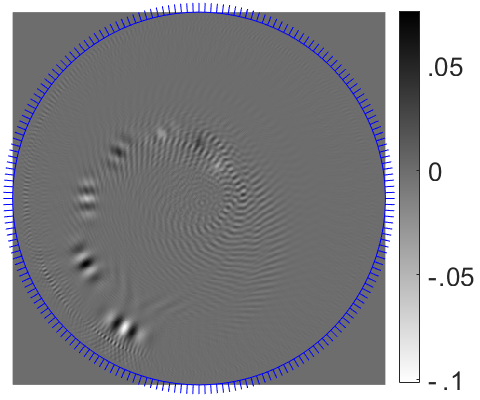} \\
    \includegraphics[height=0.19\textheight]{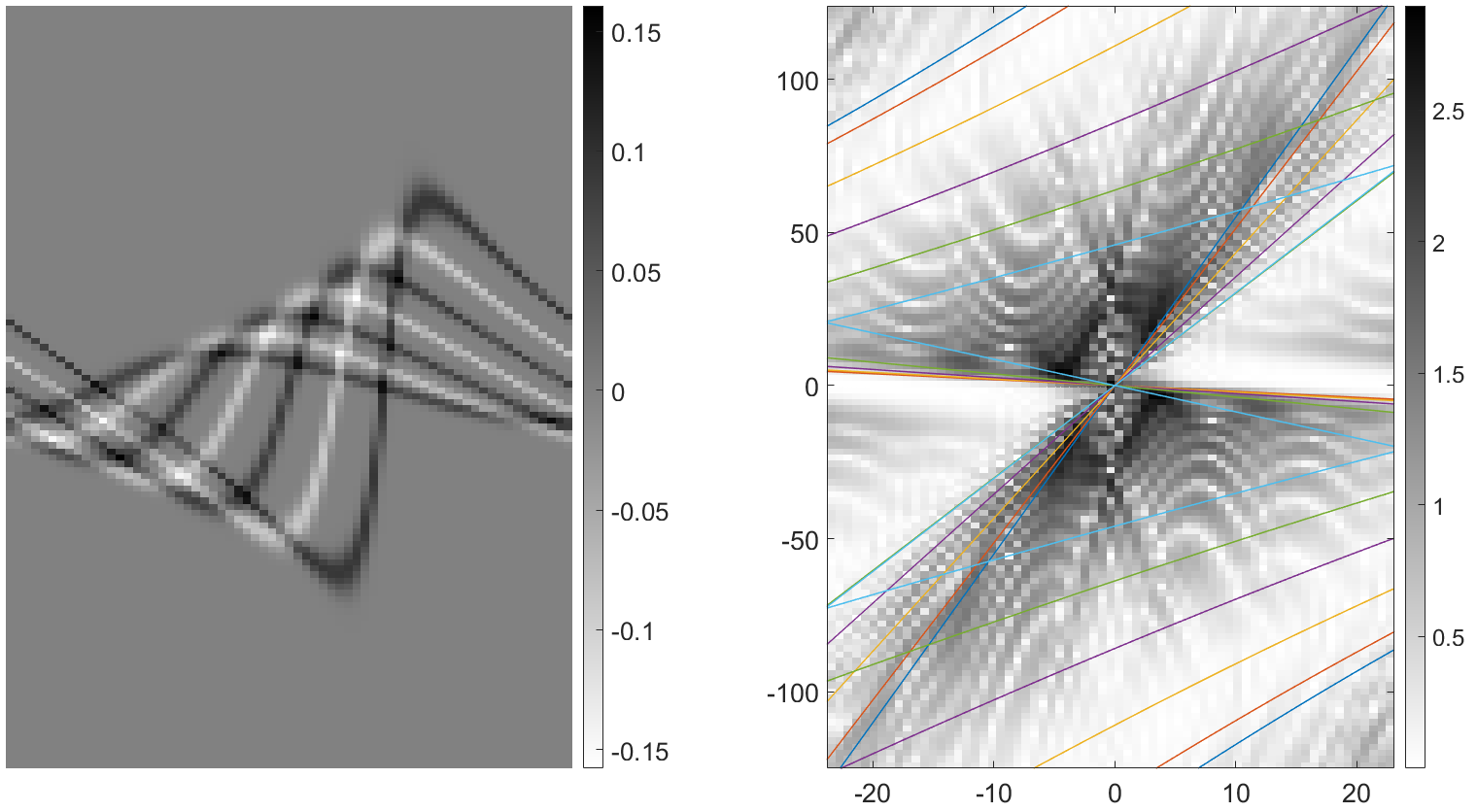} \;\;
    \includegraphics[height=0.19\textheight]{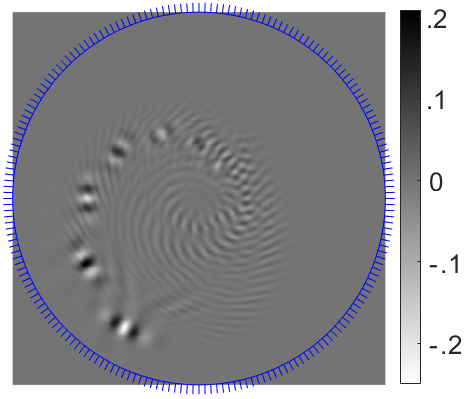}    
    \caption{Example 5. Left to right: $I f_0$ in fan-beam coordinates; the Fourier transform of $I f_0$ (with $B\cdot \Sigma_1$ bowties superimposed); pointwise error after reconstruction. Top to bottom: sampling at $B$, $B/2$, $B/3$. In the bottom row, the error structure shows variations conormal to the longest geodesics passing through each gaussian, as an indication that these singularities are the first ones to become aliased. The high-frequency oscillations at the center of the reconstructions are due to the aliasing on the top-left and bottom-right corners of the bottom Fourier plot, as will be explained in \cref{sec:aliasing}.}
    \label{fig:rates}
\end{figure}

\subsection{Aliasing and artifact prediction} \label{sec:aliasing}

The next question we address is the following: 
\begin{description}
    \item[Q3] Given an essentially $B$-band limited function $f$ and sub-sampled X-ray data for that function, can we predict the location and frequency of reconstruction artifacts ?
\end{description}
 
As explained in \cref{sec:samplingFIOs}, aliasing artifacts due to sub-sampling can be modeled as an h-FIO with explicit canonical relation (periodization in the dual variable), and the reconstruction from sub-sampled data can be understood by composing that canonical relation with the inverse of that of $I_0$. Since the answer depends on the choice of coordinate system for $I_0 f$, we will limit ourselves to experiments in fan-beam coordinates, and provide some heuristic expectations about parallel coordinates. 

\smallskip
\noindent{\bf Backprojection geometry.} Fixing $(x,v)\in \partial_+ SM$, a singularity of the form $\lambda \eta_{x,v,u\tau(x,v)}$ in data space will give rise to, upon applying an aliasing-free version of $I_0^{-1}$, a singularity at $\lambda (\dot \gamma_{x,v}(u\tau(x,v)))_\perp^\flat$. Freezing $(x,v)$, the relevant parameters are $\lambda\in \Rm\backslash \{0\}$ and $u\in [0,1]$ (rescaled time variable along the geodesic). Fixing a geometry $(\kappa, R)$, a point $(x,v)\in \partial_+ SM$ and a coordinate system $(s,\alpha)$ or $(w,\t)$, one may partition the cotangent fiber $T^*_{(x,v)}(\partial_+ SM)$ into a region which is the microlocal kernel of $I_0^{-1}$, and on the complement, we can draw iso-$u$ and iso-$\lambda$ to implicitly predict the locations and frequencies of backprojected singularities. \Cref{fig:Backproj} visualizes $(u,\lambda)$ coordinate axes in $(\eta_s,\eta_\alpha)$ and $(\eta_w,\eta_\t)$ coordinates above the point $(s,\alpha) = (0,0)$ (whose geodesic passes through the center), for $R=1$ and several choices of $\kappa$.

\begin{figure}[htpb]
    \centering
    \begin{subfigure}[b]{\textwidth}
        \centering
	\begin{tabular}{ccccc}
	    \includegraphics[width=0.16\textwidth]{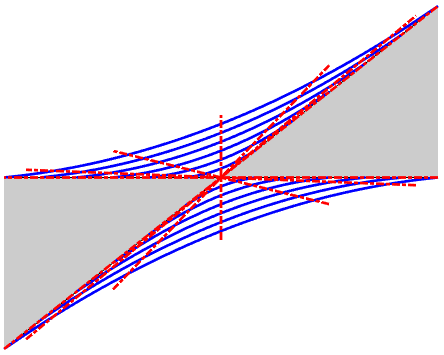} &
	    \includegraphics[width=0.16\textwidth]{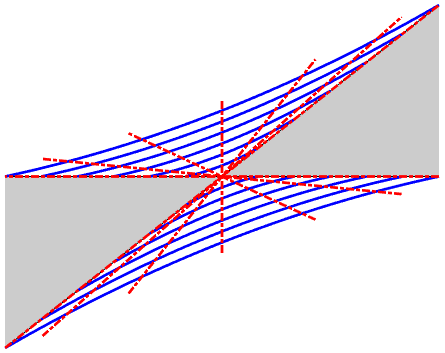} &
	    \includegraphics[width=0.16\textwidth]{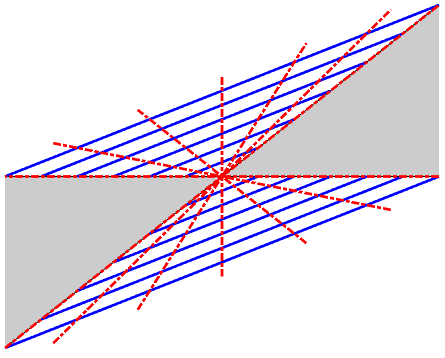} &
	    \includegraphics[width=0.16\textwidth]{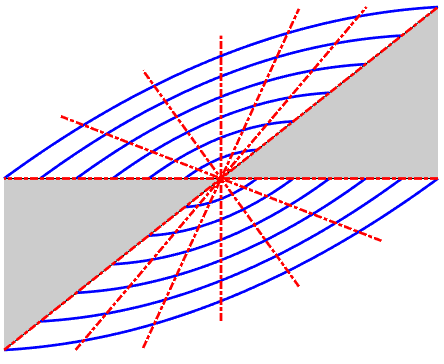} &
	    \includegraphics[width=0.16\textwidth]{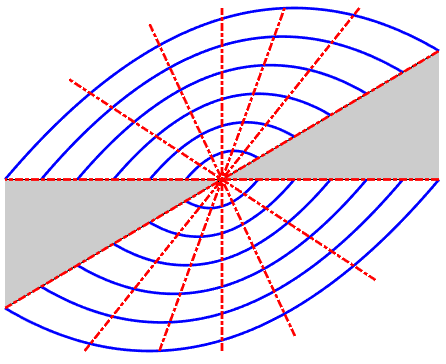} \\
	    $\kappa=-0.6$ & $\kappa=-0.3$ & $\kappa=0$ & $\kappa=0.3$ & $\kappa=0.6$
	\end{tabular}
	\caption{Axes: $(\eta_s,\eta_\alpha)$}
	\label{fig:BackprojFB}
    \end{subfigure} 
    \begin{subfigure}[b]{\textwidth}
	\centering
	\begin{tabular}{ccccc}
	    \includegraphics[width=0.16\textwidth]{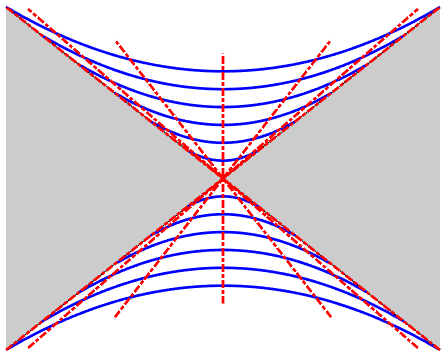} &
	    \includegraphics[width=0.16\textwidth]{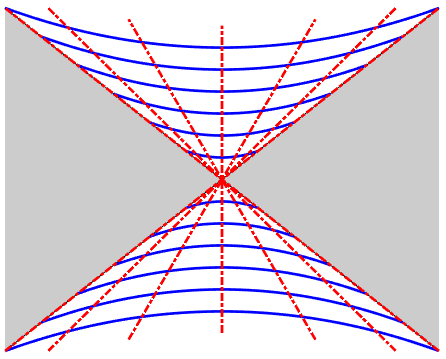} &
	    \includegraphics[width=0.16\textwidth]{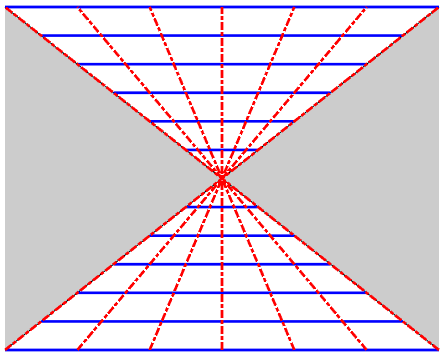} &
	    \includegraphics[width=0.16\textwidth]{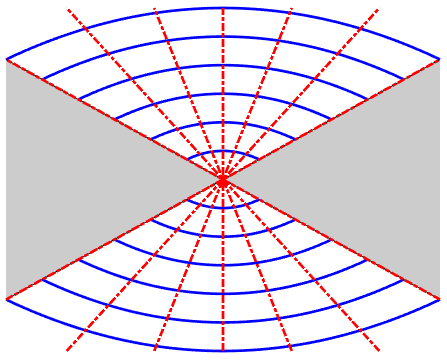} &
	    \includegraphics[width=0.16\textwidth]{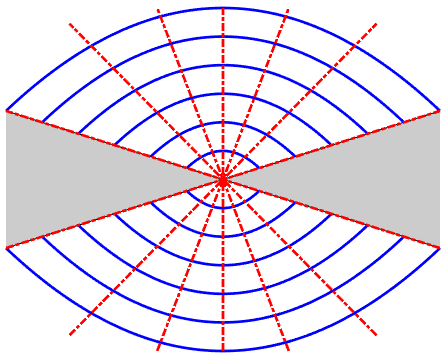} \\
	    $\kappa=-0.6$ & $\kappa=-0.3$ & $\kappa=0$ & $\kappa=0.3$ & $\kappa=0.6$
	\end{tabular}
	\caption{Axes: $(\eta_w,\eta_\t)$}
	\label{fig:BackprojWP}
    \end{subfigure} 
    \caption{Backprojection geometry on the fiber above $(s,\alpha) = (0,0)$. Equispaced iso-$\lambda$ in solid blue (increasing outward), equispaced iso-$u$ in dashed red (increasing clockwise from $0$ to $1$ on each non-shaded region), microlocal kernel in shaded gray. The picture may differ when looking at $\alpha\ne 0$, but is independent of $s$. }
    \label{fig:Backproj}
\end{figure}

\smallskip
\noindent{\bf Example 6: aliasing in fan-beam coordinates.} Fix $R = 1$ and $\kappa = 0.4$, and consider the function $f$ to be a coherent state with ``main'' frequency $B = 100$ ($h = .01$) conormal to $\gamma_{s,\alpha}(u\tau)$ with $s = L/4$, $\alpha = \pi/8$ and $u = 0.8$. We use the main frequency instead of the band limit as parameter for visual convenience. 

After finding $(N_s, N_\alpha)$ corresponding to resolving details with band limit $B$, we visualize reconstructions coming from data of size $(C_s N_s, C_\alpha N_\alpha)$, where $C_s, C_\alpha$ are undersampling $(C<1)$ or oversampling $(C>1)$ constants. Note that to avoid additional discretization issues due the inversion algorithm, we upsample all data sets to $(2N_s,2N_\alpha)$ before applying filtered-backprojection. The outcomes are displayed \cref{fig:aliasingCPC}. Each column corresponds to a different value for the constants $(C_s,C_\alpha)$, and we make the following comments: 
\begin{description}
    \item[\cref{fig:aliasingCPC1}] Benchmark/reference case. The original function is accurately reconstructed with no artifact. Neither singularity is aliased. 
    \item[\cref{fig:aliasingCPC2}] The left singularity (only) is aliased and gives rise to an artifact of higher frequency than the original, as is illustrated in the bottom Fourier Transform plot, and the fact the periodized singularity is outside the solid bowtie on the second row. 
    \item[\cref{fig:aliasingCPC3}] Both singularities are aliased and land into the microlocal kernel region for the backprojection operator. There are very weak artifacts of 10\% the original magnitude. 
    \item[\cref{fig:aliasingCPC4}] Both singularities are aliased, each gives rise to artifacts along the geodesic, one with lower frequency than the original (coming from the right singularity), one with higher frequency (coming from the left singularity).
\end{description}

\smallskip
\noindent{\bf Heuristic comments about parallel coordinates.} We point out what one should expect in the case of parallel coordinates: 
\begin{itemize}
    \item Unlike fan-beam coordinates, uniform Cartesian sampling in parallel coordinates provides the same resolution on {\em both} $C_+(\Omega)$ and $C_-(\Omega)$ for any $\omega\in \WFH(f)$. In particular, the scenario of ``partial aliasing'' as in \cref{fig:aliasingCPC2} cannot happen. 
    \item Undersampling in $w$ (resp. $\t$) will result in horizontal (resp. vertical) folding of the bowties in \cref{fig:BackprojWP}. Folded singularities can land in the microlocal kernel of $I_0^{-1}$, or in the upper and lower triangular zones, where the iso-$u$ and iso-$\lambda$ curves on \cref{fig:BackprojWP} help locate the location and frequency of the artifacts as in the fan-beam case. 
\end{itemize}

\begin{figure}[htpb]
    \centering
    \begin{subfigure}[b]{0.24\textwidth}
	\centering
	\includegraphics[width=\textwidth]{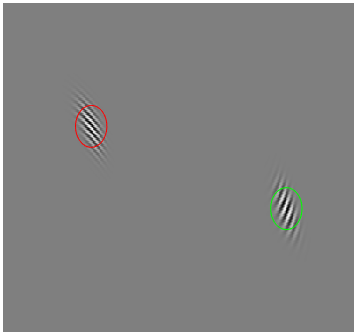} \\
	\includegraphics[trim = 15 15 15 15, clip, width=0.48\textwidth]{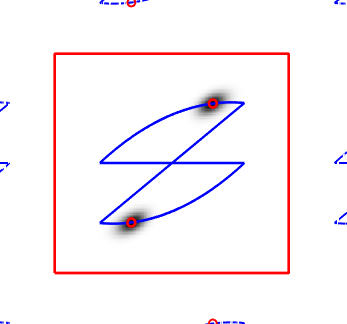}
	\includegraphics[trim = 15 15 15 15, clip, width=0.48\textwidth]{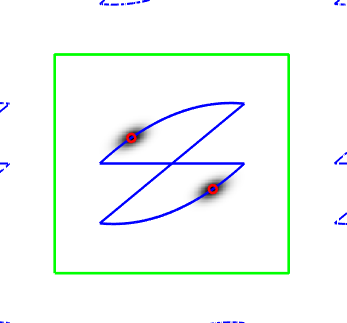} \\
	\includegraphics[width=\textwidth]{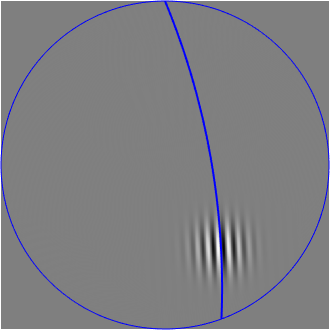} \\
	\includegraphics[width=\textwidth]{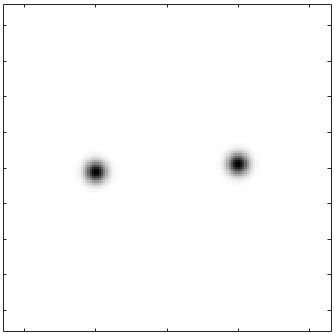}	
	\caption{\!$(C_s,C_\alpha)\!=\!(1.5,1.5)$\!}
	\label{fig:aliasingCPC1}
    \end{subfigure}
    \begin{subfigure}[b]{0.24\textwidth}
	\centering
	\includegraphics[width=\textwidth]{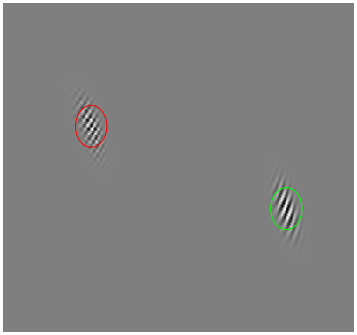}
	\includegraphics[trim = 15 15 15 15, clip, width=0.48\textwidth]{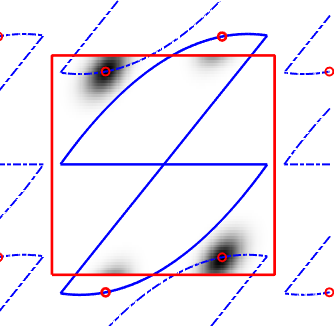}
	\includegraphics[trim = 15 15 15 15, clip, width=0.48\textwidth]{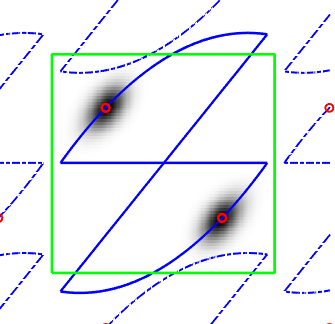} \\
	\includegraphics[width=\textwidth]{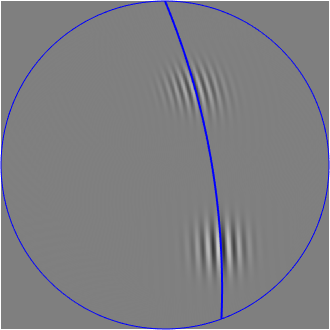} \\
	\includegraphics[width=\textwidth]{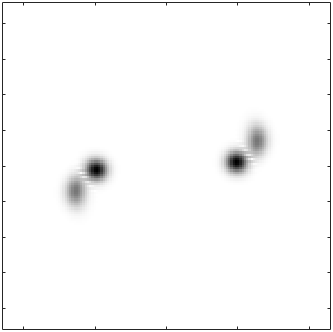}	
	\caption{$(C_s,C_\alpha)=(1,0.7)$}
	\label{fig:aliasingCPC2}
    \end{subfigure}    
    \begin{subfigure}[b]{0.24\textwidth}
	\centering
	\includegraphics[width=\textwidth]{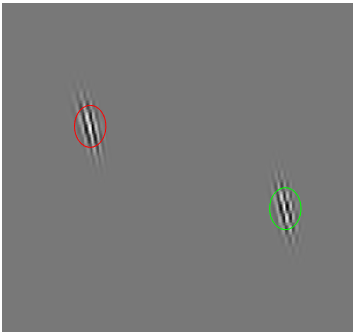}
	\includegraphics[trim = 15 15 15 15, clip, width=0.48\textwidth]{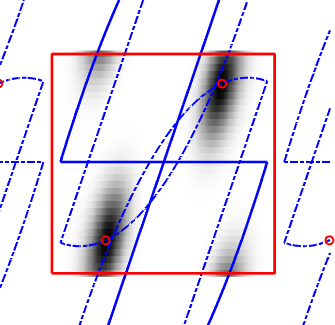}
	\includegraphics[trim = 15 15 15 15, clip, width=0.48\textwidth]{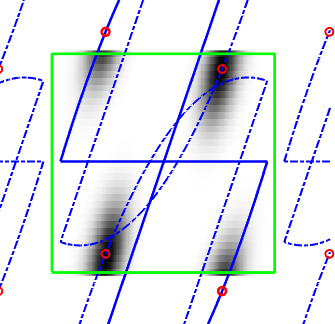} \\
	\includegraphics[width=\textwidth]{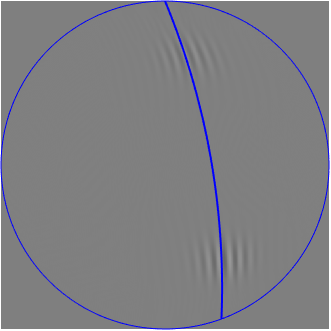} \\
	\includegraphics[width=\textwidth]{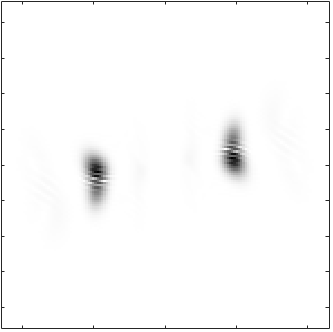}	
	\caption{$(C_s,C_\alpha)=(1,0.3)$}
	\label{fig:aliasingCPC3}
    \end{subfigure}        
    \begin{subfigure}[b]{0.24\textwidth}
	\centering
	\includegraphics[width=\textwidth]{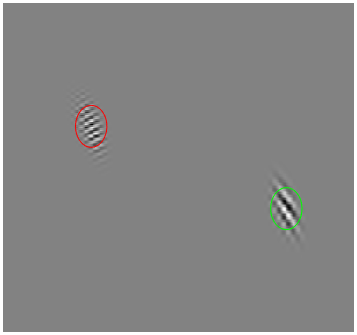}
	\includegraphics[trim = 15 15 15 15, clip, width=0.48\textwidth]{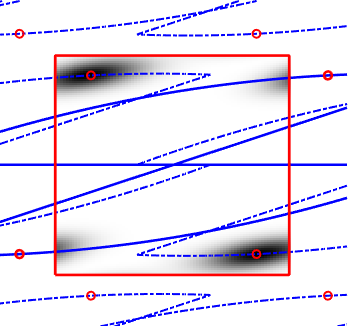}
	\includegraphics[trim = 15 15 15 15, clip, width=0.48\textwidth]{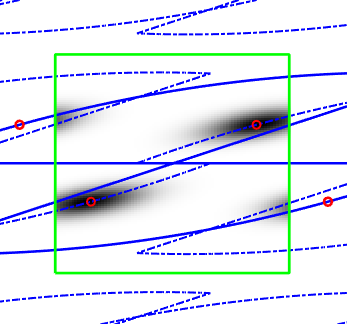} \\
	\includegraphics[width=\textwidth]{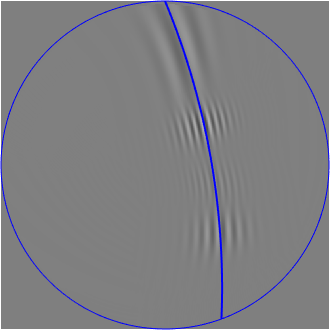} \\
	\includegraphics[width=\textwidth]{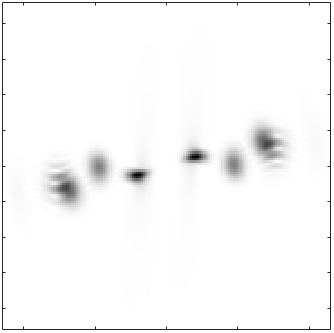}	
	\caption{$(C_s,C_\alpha)=(0.4,1)$}
	\label{fig:aliasingCPC4}
    \end{subfigure}
    \caption{Aliasing experiment in $(R,\kappa) = (1,0.4)$ geometry, in fan-beam coordinates. For each column, top to bottom: X-ray data in fan-beam; windowed Fourier Transform around each wave packet in data space; reconstructed function and possible artifacts (the whole row is on the same color scale); Fourier transform of the reconstruction (the whole row has the same axis scale but the color axis is rescaled for contrast).}
    \label{fig:aliasingCPC}
\end{figure}

\subsection{Further upsampling strategies to reduce the sampling limit}

In this last experiment, we illustrate the ideas explained in \cref{sec:slanted}, in the simplest case where the matrix $W$ is diagonal. We illustrate this idea in fan-beam coordinates and negative curvature with parameter $\kappa = -0.3$.  
In the case $\kappa = 0$, some of these ideas date back to \cite{Natterer1993}.  

Consider the function from \cref{fig:Ex4}, whose Fourier transform modulus consists of two gaussians center at the central frequency $B = 100$. The Nyquist cartesian box containing the frequency set of $I_0 f$ is ${\mathcal  P} = [-b_s,b_s]\times [-b_\alpha,b_\alpha]$ (the red box in \cref{fig:Ex7_2_1}). However, that frequency set is included in the parallelogram ${\mathcal  B}\subset {\mathcal  P}$ of vertices $(-b_s,-b_\alpha)$, $(b_s,0)$, $(b_s,b_\alpha)$, $(-b_s,0)$(the red parallelogram in \cref{fig:Ex7_2_2}). The set ${\mathcal  B}$ tiles the plane twice as densely as ${\mathcal  P}$ without intersections, and thus allows for a twice-coarser vertical sampling rate.

We sample $I_0 f$ at a non-aliasing horizontal rate $1.2\cdot b_s = 240$, while the vertical rate, $0.6\cdot b_\alpha=223$ is aliasing for a box-based inversion, non-aliasing for a parallelogram-based inversion, see \cref{fig:Ex7}. 

\begin{figure}[htpb]
    \centering
    \includegraphics[height=0.14\textheight]{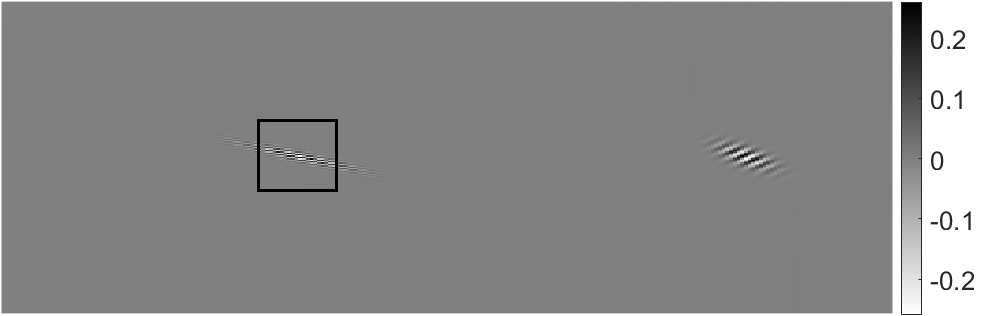}
    \includegraphics[height=0.14\textheight]{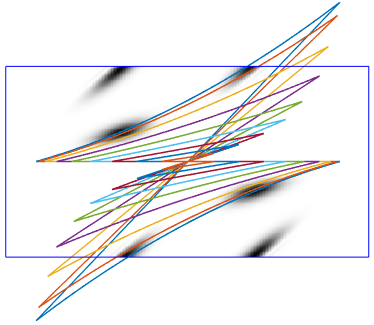}
    \caption{Left: vertically subsampled data, the boxed singularity on the left is aliased. Right: modulus of Fourier transform of $I_0 f$ with bowties superimposed. The surrounding blue box is the frequency range allowed by the sampling.}
    \label{fig:Ex7}
\end{figure}

We upsample $I_0 f$ to a grid three times the size along each axis, in two ways: 
\begin{itemize}
    \item Using regular Lanczos-3 sampling along coordinate axes. See \cref{fig:Ex7_2_1}. 
    \item Via Fourier, by periodizing the spectrum 3 times along the coordinate axes, multiplying by the characteristic function of ${\mathcal  B}$, and taking the inverse Fourier transform. See \cref{fig:Ex7_2_2}.
\end{itemize} 

We run an inversion on each upsampled data set. The reconstructions show artifacts in the first case, no artifacts in the second, see \cref{fig:Ex7_2}. 

\begin{figure}[htpb]
    \centering
    \begin{subfigure}[b]{\textwidth}
	\centering
	\includegraphics[height=0.12\textheight]{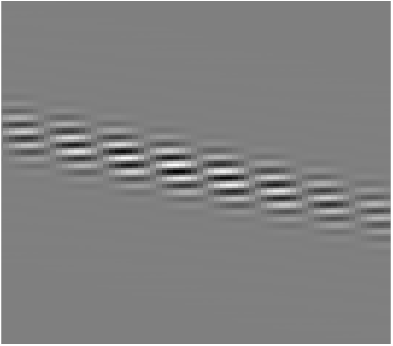}
	\includegraphics[height=0.12\textheight]{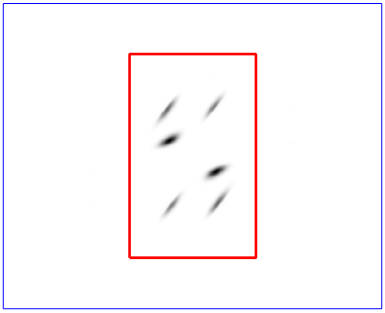} 
	\includegraphics[height=0.12\textheight]{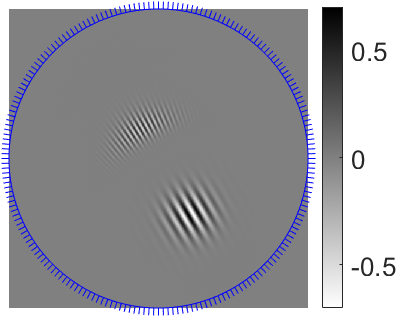} 
	\includegraphics[height=0.12\textheight]{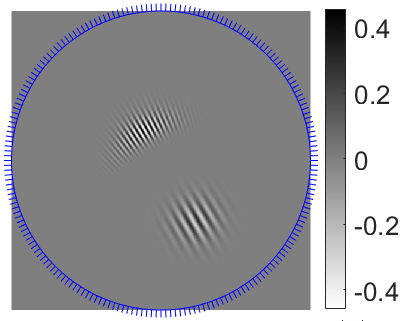}
	\caption{Upsampling method based on ${\mathcal  P}$. The singularity remains aliased after upsampling. Reconstruction has aliasing artifacts.}
	\label{fig:Ex7_2_1}
    \end{subfigure}
    \begin{subfigure}[b]{\textwidth}
	\centering
	\includegraphics[height=0.12\textheight]{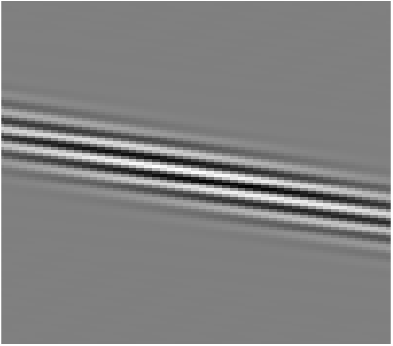}
	\includegraphics[height=0.12\textheight]{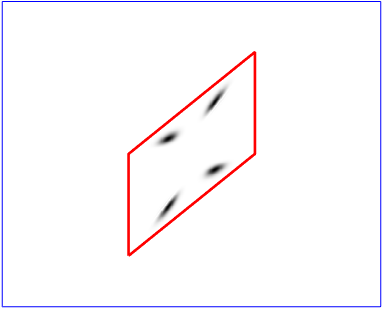}
	\includegraphics[height=0.12\textheight]{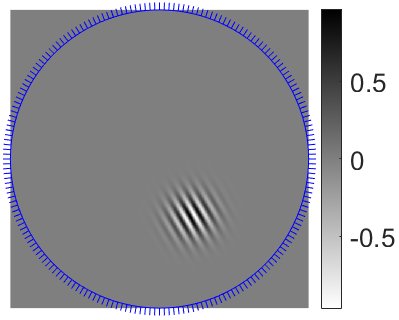} 
	\includegraphics[height=0.12\textheight]{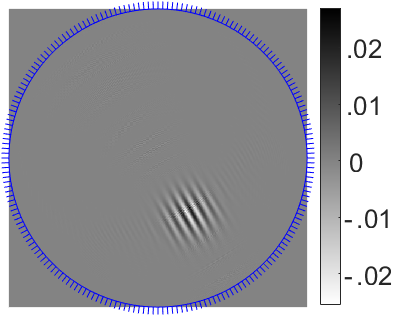}
	\caption{Upsampling method based on ${\mathcal  B}$. The singularity is properly recovered after upsampling. Reconstruction has no aliasing artifacts.}
	\label{fig:Ex7_2_2}
    \end{subfigure} 
    \caption{Left to right: zoom on the left singularity (boxed in \cref{fig:Ex7}) after upsampling; modulus of Fourier transform of upsampled data; reconstruction; pointwise error.}
    \label{fig:Ex7_2}
\end{figure}

\section{Conclusions}
\label{sec:conclusions}  
For the geodesic X-ray transform $If$ on simple surfaces, we developed a theory for determining the sharp discretization rates of $If$ given an a priori approximate band limit of $f$, as the step size tends to zero. In particular, we show how the discretization rate of $If$ would be determined by that of $f$ in numerical simulations. If the discretization/sampling condition is not satisfied, then $If$ is aliased, and the reconstructed $f$ is aliased as well. Unlike what happens in classical sampling theory, the aliasing artifacts of $f$ are not local: they appear at different locations, together with changes of the magnitude and the orientation of the frequencies; and we characterize those artifacts. 
We estimate sharply the resolution limit of the reconstructed $f$ imposed by the sampling rate of the data $If$. 

We express the necessary conditions in terms of the Jacobi fields of the metric and of the coordinate system chosen. We analyze constant curvature disks in more detail, where the sharp sampling conditions are explicitly computed. In particular, we demonstrate the effect of the curvature and of the boundary curvature going to zero or having length going to infinity, on the sampling rates. 

The analysis is based on the asymptotic sampling theory as the step size gets smaller and smaller, developed by the second author in \cite{Stefanov2018}, and it uses semiclassical microlocal methods. In particular, the notion of band limit is played by the semiclassical wave front set and the canonical relation of $I$, as an Fourier Integral Operator, is the one relating the ``band limits'' of $f$ and $If$.

\appendix
\section{Proof of \cref{lem:Psi}} \label{sec:app}

We first compute the differential of the scattering relation defined in \cref{eq:scatrel}. In fact, upon defining the function $\tilde \tau\colon \partial SM \to \Rm$ as the fiberwise odd extension of $\tau$, it is known that $\tilde\tau$ is smooth (see e.g. \cite[Lemma 4.1.1]{Sharafudtinov1994}), and $S$ smoothly extends to $\partial SM$ through the relation
\begin{align*}
    S \colon \partial SM\to \partial SM, \qquad S(x,v) = \varphi_{\tilde\tau(x,v)} (x,v).
\end{align*}

Denote $\G{\cdot}{\cdot}$ the Sasaki metric on $SM$ and $\g{\cdot}{\cdot}$ the metric on $M$. On $\partial SM$ we pull the Sasaki metric back from $SM$ via the inclusion $\partial SM\to SM$. With $\nu_x$ the unit inner normal at $x\in \partial M$, the oriented tangent vector is $t_x = -\nu_x^\perp$, and we denote $\HH_{(x,v)}$ the horizontal lift of $t_x$ at $(x,v)\in \partial_\pm SM$, i.e. horizontal-tangential differentiation (similarly, denote $N_{(x,v)}$ the horizontal lift of $\nu_x$). Then $(\HH,V)$ is an orthonormal frame for $T(\partial SM)$.

For $(x,v)\in \partial SM$, let us denote $\mu(x,v) = \g{\nu_x}{v} = \G{N_{(x,v)}}{X_{(x,v)}}$ and $\mu_\perp (x,v) = \g{\nu_x}{v_\perp} = - \G{N_{(x,v)}}{X_\perp|_{(x,v)}}$, and in particular, \rev{the relation
    \begin{align*}
	t_x = \g{-\nu_\perp}{v} v + \g{-\nu^\perp}{v_\perp} v_\perp = \mu_\perp v + \mu v_\perp
    \end{align*}
    induces, by horizontal lift, the relation $H = \mu_\perp X + \mu X_\perp$. 

}

\begin{lemma}\label{lem:scatrel} Let $(x,v) \in \partial SM$. Then we have the following 
    \begin{align*}
	\left[
	\begin{array}{c}
	    dS|_{(x,v)} \HH_{(x,v)} \\ dS|_{(x,v)} V_{(x,v)}
	\end{array}
    \right] = \left[
    \begin{array}{cc}
	\mu(x,v) \frac{a(x,v,\tilde\tau)}{\mu(S(x,v))} &  -\mu(x,v) \dot a(x,v,\tilde\tau) \\
	-\frac{b(x,v,\tilde\tau)}{\mu(S(x,v))} & \dot b(x,v,\tilde\tau) 
    \end{array}
\right] \left[
	\begin{array}{c}
	    \HH_{S(x,v)} \\ V_{S(x,v)}
	\end{array}
    \right].
    \end{align*}    
\end{lemma}

In particular, with the Wronskian property $a\dot b - b\dot a \equiv 1$, the Jacobian determinant of $S$ is indeed $\frac{\mu(x,v)}{\mu(S(x,v))}$ as is well-known.
\begin{proof}  
    With $a,b$ defined in \cref{eq:ab}, the differential of the geodesic flow is given by 
    \begin{align*}
	d\varphi_t|_{(x,v)} X_{(x,v)} &= X(t), \\
	d\varphi_t|_{(x,v)} X_{\perp (x,v)} &= a(x,v,t) X_{\perp}(t) - \dot a(x,v,t) V(t) , \\
	d\varphi_t|_{(x,v)} V_{(x,v)} &= -b(x,v,t) X_\perp (t) + \dot b(x,v,t) V(t),    
    \end{align*}
    where for $Y\in \{X,X_\perp,V\}$, $Y(t)$ denotes $Y_{\varphi_t(x,v)}$. We first compute
    \begin{align*}
	dS|_{(x,v)} V &= V(\varphi_{\tilde\tau(x,v)}(x,v)) \\    
	&= d\varphi_{\tilde\tau}|_{(x,v)} V + V\tilde\tau(x,v) \frac{d\varphi_{\tilde\tau}}{dt} \\
	&= -b(x,v,\tilde\tau) X_\perp + \dot b(x,v,\tilde\tau) V + V\tilde\tau(x,v) X, 
    \end{align*}
    where all vectors sit above $S(x,v)$. We now use that the result must be tangent to $\partial SM$, that is,
    \begin{align*}
	0 &= \G{N_{S(x,v)}}{-b(x,v,\tilde\tau) X_\perp + \dot b(x,v,\tilde\tau) V + V\tilde\tau(x,v) X} \\
	&= b(x,v,\tilde\tau) \mu_\perp(S(x,v)) + 0 + V\tilde\tau(x,v) \mu(S(x,v)).  
    \end{align*}
    This gives the relation $V\tilde\tau(x,v) = b(x,v,\tilde\tau) \frac{-\mu_\perp(S(x,v))}{\mu(S(x,v))}$ and in turn

    \begin{align*}
	dS|_{(x,v)} V &= \frac{b(x,v,\tilde\tau)}{\mu(S(x,v))} ( - \mu_\perp(S(x,v)) X - \mu(S(x,v)) X_\perp) + \dot b(x,v,\tilde\tau) V \\
	&= - \frac{b(x,v,\tilde\tau)}{\mu(S(x,v))} \HH + \dot b (x,v,\tilde\tau) V. 
    \end{align*}
    Similarly, we compute 
    \begin{align*}
	dS|_{(x,v)} \HH &= \HH(\varphi_{\tilde\tau(x,v)}(x,v)) \\
	&= d\varphi_{\tilde\tau}|_{(x,v)} \HH + \HH \tilde\tau(x,v) X \\
	&= (\HH \tilde\tau(x,v) + \mu_\perp(x,v)) X + \mu(x,v) (a(x,v,\tilde\tau) X_\perp - \dot a(x,v,\tilde\tau) V).
    \end{align*}
    We now use that the result must be tangent to $\partial SM$, that is 
    \begin{align*}
	0 &= \G{N_{S(x,v)}}{dS|_{(x,v)} \HH} \\
	&= (\HH\tilde\tau(x,v)+\mu_\perp(x,v))\mu(S(x,v)) - \mu(x,v) a(x,v,\tilde\tau) \mu_\perp (S(x,v)) + 0. 
    \end{align*}
    Using this to replace the expression of $\HH \tilde\tau + \mu_\perp(x,v)$, we arrive at 
    \begin{align*}
	dS|_{(x,v)} \HH &= \mu(x, v) \left( \frac{a(x,v,\tilde\tau)}{\mu(S(x,v))} \left( \mu_\perp(S(x,v)) X + \mu(S(x,v)) X_\perp \right) - \dot a(x,v,\tilde\tau) V \right) \\
	&= \mu(x,v) \left( \frac{a(x,v,\tilde\tau)}{\mu(S(x,v))} \HH - \dot a(x,v,\tilde\tau) V \right), 
    \end{align*}
    hence the result.
\end{proof}

We are now ready to prove \cref{lem:Psi}. 

\begin{proof}[Proof of \cref{lem:Psi}]
    From the calculations of \cref{lem:scatrel}, the \rev{differential} of the map
    \begin{align*}
	E_x \colon \partial_+ S_x M \to \partial M \backslash \{x\}, \qquad E_x(v) := \pi (S(x,v)) 
    \end{align*}
    is easily computed as
    \begin{align*}
	dE_x (V) = d\pi|_{S(x,v)} dS|_{(x,v)} V = \frac{- b(x,v,\tau)}{ \mu(S(x,v))} \partial_{s'} = \frac{b(x,v,\tau)}{|\mu(S(x,v))|} \partial_{s'},
    \end{align*}
    where $\partial_{s'}$ is the unit tangent vector along the second copy of $\partial M$. Now observing that $\Psi(x,v) = (x,E_x(v))$, we directly compute
    \begin{align*}
	d\Psi (V) = \frac{-b(x,v,\tau)}{\mu\circ S}\ \partial_{s'}, \qquad d\Psi (H) = \partial_s + \frac{\mu}{\mu\circ S} a(x,v,\tau)\ \partial_{s'}. 
    \end{align*}    
    Hence the Jacobian determinant of $\Psi$ is $\frac{-b(x,v,\tau)}{\mu\circ S}$. If $(M,g)$ is simple, then the no-conjugate-point assumption implies that $b(x,v,\tau(x,v))\ne 0$ for all $(x,v)\in (\partial_+ SM)^{int}$, and thus $\Psi$ is a global diffeomorphism.
\end{proof}


\bibliographystyle{siamplain}


\end{document}